\documentclass[reqno]{amsart}
\usepackage{caption}
\usepackage{amsmath,amssymb,amsthm}
\usepackage{amsmath}
\usepackage{amsfonts}
\usepackage{float}
\usepackage{mathtools}
\usepackage{mathrsfs}  
\usepackage{youngtab}
\usepackage[pdftex]{graphicx}
\usepackage{tikz}
\usetikzlibrary{decorations.markings,arrows,automata,positioning}
\tikzset{partition/.style={fill,circle,inner sep=1pt}}
\usetikzlibrary{positioning}
\usetikzlibrary{decorations.pathreplacing}

\usetikzlibrary{matrix,arrows}
\usetikzlibrary{calc}
\tikzset{partition/.style={fill,circle,inner sep=1pt},
         part/.style={baseline=0,scale=0.5,bend left=45},
         partlabel/.style={below}}
\usetikzlibrary{shapes,arrows}
\usetikzlibrary{snakes}
\tikzstyle{pnt}=[draw,ellipse,fill,inner sep=1pt]

\tikzstyle{opnt}=[draw,ellipse,inner sep=1pt]
\tikzstyle{opnt}=[ ]
\tikzstyle{pntt}=[draw,ellipse,fill,inner sep=0.5pt]
\tikzstyle{point}=[draw,ellipse,fill,inner sep=2pt]

\usepackage{color}

\newcommand{\Hom}{\operatorname{Hom}}
\newcommand{\Trace}{\operatorname{Trace}}
\newcommand{\Mat}{\operatorname{Mat}}
\newcommand{\jmp}{\operatorname{jmp}}
\newcommand{\orb}{\operatorname{Orb}}
\newcommand{\cent}{\operatorname{Centralizer}}

\usepackage{amssymb}

\newtheorem{theorem}{Theorem}[section]
  \newtheorem{lemma}[theorem]{Lemma}
  \newtheorem{palgorithm}[theorem]{Separation of Variables (SOV) Approach}
  
  \newtheorem{proposition}[theorem]{Proposition}
  \newtheorem{corollary}[theorem]{Corollary}
  	\newtheorem{definition}[theorem]{Definition} 
  	\theoremstyle{definition}
  	\newtheorem{example}[theorem]{Example} 
  \theoremstyle{remark}  	
  	
  	\newtheorem{remark}[theorem]{Remark}
	\newtheorem{note}[theorem]{Note}
\allowdisplaybreaks

\newcommand{\Z}{\mathbb{Z}}


\begin{document}
\title[Separation of Variables]{Separation of Variables and the Computation of Fourier Transforms on Finite Groups, II}

\author{David Maslen}
\address{HBK Capital Management, New York, NY 10036}
\email{david@maslen.net}

\author{Daniel N. Rockmore}
\address{Departments of Mathematics and Computer Science, Dartmouth College, Hanover, NH 03755}
\email{rockmore@math.dartmouth.edu}
\thanks{The second author was partially supported by AFOSR Award FA9550-11-1-0166 and the Neukom Institute for Computational Science at Dartmouth College}

\author{Sarah Wolff}
\address{Department of Mathematics, Denison University, Granville, OH 43023}
\email{wolffs@denison.edu}
\thanks{The third author was partially supported by an NSF Graduate Fellowship.}

\subjclass[2000]{To be filled in.}

\date{\today}



\begin{abstract}
We present a general diagrammatic approach to the construction of efficient algorithms for computing the Fourier transform of a function on a finite group. By extending work which connects Bratteli diagrams to the construction of Fast Fourier Transform algorithms 
we make explicit use of the path algebra connection to the construction of Gel'fand-Tsetlin bases and work in the setting of quivers. We relate this framework to the construction of a {\em configuration space} derived from a Bratteli diagram. In this setting the complexity of an algorithm for computing a Fourier transform reduces to the calculation of the dimension of the associated configuration space. 
Our methods give  improved upper bounds for computing the Fourier transform for the general linear groups over finite fields, the classical Weyl groups, and homogeneous spaces of finite groups, while also  recovering the best known algorithms for the symmetric group and compact Lie groups.
\end{abstract}

\maketitle

\section{Introduction}\label{in}
The {\em Fast Fourier Transform} (FFT) remains among the most important family of algorithms in information processing \cite{RockmoreIEEE}. It efficiently computes the {\em discrete Fourier transform} (DFT) which is equivalent to the  matrix-vector multiplication
\begin{equation}
\Big( e^{2\pi i j k/n} \Big)_{j,k} \vec{f}
\end{equation}
for $i = \sqrt{-1}$, $j,k = 0,\dots n-1$, and $\vec{f}$ a complex-valued vector of length $n$ \cite{RockmoreIEEE}. This calculation can be framed in a number of ways. We take a representation theoretic point of view and cast the DFT as a change of basis in $\mathbb{C}[C_N]$, the complex group algebra of the cyclic group of order $N$, from a natural basis of group element indicator functions to a basis of irreducible matrix elements. This perspective (which is at times driven by applications \cite{rocksurvey}) suggests a generalization of the DFT to finite nonabelian groups $G$ as the computation of a change of basis in $\mathbb{C}[G]$ from a basis of indicator functions to a basis of irreducible matrix elements, which raises the kinds of  attendant questions of computational complexity (see e.g., \cite{MaslenNotices}) addressed herein. 

Let $T_G(R)$ denote the computational complexity of the Fourier transform on a group $G$ at a set of inequivalent irreducible representations $R$. Then $C(G)$ denotes the \textit{complexity of the group} $G$, defined as
$$C(G):=\min_R \{T_G(R)\}.$$ 

For $N$ a ``highly composite" number, Cooley and Tukey in 1965 famously presented an algorithm to show $C(\Z/N\Z)\leq O(N\log_2 N)$  \cite{cooleytukey}. Yavne \cite{yavne} later  showed that for $N=2^m$, $C(\Z/N\Z) \leq \frac{8}{3}N\log_2 N-\frac{16}{9}N-\frac{2}{9}(-1)^{\log_2(N)}+2$. More recently, Johnson and Frigo \cite{frigo} and Lundy and Van Buskirk \cite{lundy} have further reduced the total number of complex multiplications required, but without affecting the overall group complexity $C(\mathbb{Z}/N\mathbb{Z})$. 
More generally, for $A$ an abelian group of size $N$, various efficiencies can be combined to prove the complexity of the DFT on $A$ is bounded above by $O(N\log_2N)$ \cite{diaconis-fft}. 
The deep and ongoing study of this problem has been motivated by a wide range of applications in digital signal processing and beyond (see e.g. \cite{auslander, barros, astrophysics,  blindimage, rao,toli,vanloan}). 

The Cooley-Tukey algorithm is undoubtedly the most famous of the FFTs. It is a divide-and-conquer algorithm whose basic idea was first recorded by Gauss in unpublished work (see e.g. \cite{burr} for a brief history of the algorithm). The key step is to rewrite the DFT on a cyclic group $C_N$ as a linear combination of DFTs on $C_n < C_N$ (for $n\mid N$). Iterating this step for a chain of subgroups of $C_N$ yields  algorithms more efficient than a direct matrix-vector multiplication. 

This divide-and-conquer algorithm produces efficiencies by reducing the ``big" problem to smaller subproblems that have common structure and in fact are themselves, smaller versions of the original, that can be efficiently combined to produce the required result. In this paper we continue a line of work  that generalizes this approach to nonabelian groups \cite{MR-duco, MR-adapted, sovi,  rocksurvey}. In this case the common subproblems are repeated occurrences of particular pieces of matrix multiplications (e.g., repeated block and thus element-by-element multiplications) whose existence is ensured by working with very specific kinds of bases for the irreducible matrix representations (and associated matrix elements) enabled by choices of group factorizations. Thus, there is in a sense, ``divide-and-conquer" going on in both the group and its dual. 

The bases are encoded via paths in a Bratteli diagram attached to the group of interest, which in turn means that irreducible matrix elements correspond to pairs of paths in the diagram, which for a given group element may only be nonzero when of a particular form. I.e., the ``repeated units" of our divide-and-conquer amount to certain subgraphs of a Bratteli diagram and efficiencies are gained by recognizing their multiple appearances in the corresponding calculation. This is the  guts of the ``separation of variables" (SOV) approach  first introduced in \cite{sovi} and then extended in \cite{maslen} via a quiver-based formalism. 

In this paper we finally take on the problem of laying a proper axiomatic and logical foundation for this approach and in so doing also produce improved algorithms for 
the important families of classical Weyl groups $B_n$ and $D_n$ and the general linear groups over finite fields $GL_n(\mathbb{F}_q)$:

\begin{theorem}\label{Bnthm}
$C(B_n)\leq  n(2n-1)\vert B_n\vert.$
\end{theorem}

\begin{theorem}\label{Dnthm}
$C(D_n)\leq \frac{n(13n-11)}{2}\vert D_n\vert.$
\end{theorem}
\begin{theorem}\label{Gln} 
$C(Gl_n(\mathbb{F}_q))\leq \left(\frac{4^nq^{n+1}-q}{4q-1}\right)\vert Gl_n(q)\vert.$
\end{theorem}
Improvements for the complexity of Fourier transforms on related homogeneous spaces are also presented. For example, let $B_n/B_{n-k}$ denote the homogenous space  of the Weyl group $B_n$. 
\begin{theorem}\label{hombnthm}
$C(B_n/B_{n-k})\leq k(4n-2k-1)\frac{\vert B_n\vert}{\vert B_{n-k}\vert}.$ 
\end{theorem}
Moreover, our results extend to chains of semisimple algebras rather than just chains of group algebras. This will be explored in subsequent work \cite{algpaper}. 

In Section \ref{prelim} we outline the preliminaries needed for our results, including a discussion of the mainideas behind the SOV approach, necessarily recapitulated (in an abbreviated format) in order to make this paper as self-contained as possible (although we acknowledge -- given the title -- the dependence on part I \cite{sovi}). In Section \ref{sepvarstate} we present the improved SOV approach in detail, rewriting an iterated product in the path algebra as a sequence of bilinear maps on the newly defined ``configuration spaces" (vector spaces of quiver morphisms). In Section \ref{applications} we give factorizations and counts to prove the specific group complexity results (Theorems~\ref{Bnthm},\ref{Dnthm},\ref{Gln}) and also recover previously known methods for $S_n$ \cite{maslen} and compact Lie groups \cite{compactgroups}. The results in Section \ref{applications} depend on various important, but very technical details of the explicit computation of the configuration space dimensions. In order to bring the reader to the complexity results as quickly and directly as seems possible, we postpone the presentation of these details to Sections \ref{proofsofstuff} and \ref{generalcounts}. This includes  generalizations of some results of Stanley on differential posets \cite{diffposets, diffposets2} used to give explicit methods for finding these dimensions. This may be of independent interest. Some of the more laborious (but necessary) formalisms are collected in three short appendices. 

\section{Background}\label{prelim}

\subsection{Fourier transforms and the group algebra}\label{galgebra}
The usual discrete Fourier transform of  a finite data sequence may be viewed as a special case of Fourier transforms on finite groups, defined using group representations. Results here assume complex representations, unless spelled out otherwise, although most results go through more generally. For necessary background on the representation theory of finite groups we refer the reader to \cite{serre}.

\begin{definition}\label{groups} Let $G$ be a finite group and $f$ a complex-valued function on $G$. 
\begin{itemize}
\item[(i)] Let $\rho$ be a matrix representation of $G$. Then the \textbf{Fourier transform of} $\mathbf{f}$ \textbf{at} $\mathbf{\rho}$, denoted $\hat{f}(\rho)$, is the matrix sum
$$\hat{f}(\rho)=\sum_{s\in G} f(s)\rho(s).$$
\item[(ii)] Let $R$ be a set of matrix representations of $G$. Then the \textbf{Fourier transform of} $\mathbf{f}$ \textbf{on} $\mathbf{R}$ is the direct sum of Fourier transforms of $f$ at the representations in $R$:
$$\mathcal{F}_R(f)=\bigoplus_{\rho\in R} \hat{f}(\rho)\in\bigoplus_{\rho\in R} \Mat_{\dim\rho}(\mathbb{C}).$$

\end{itemize}
\end{definition}

When we compute the Fourier transform for  a complete set of inequivalent irreducible representations $R$ of $G$  we refer to the calculation as the \textbf{computation of a Fourier transform on $G$} (with respect to $R$).

\begin{definition}\label{complexdef} Let $G$ be  a finite group, $R$ a set of matrix representations of $G$.
\begin{itemize}
\item[(ii)] Let $+_G(R)$ (respectively, $\times_G(R)$) denote the minimum number of complex arithmetic additions (resp., multiplications)  needed to compute the Fourier transform of $f$ on $R$ via a straight-line program\footnote{A \textbf{straight-line program} is a list of instructions for performing the operations $\times, \div, +, -$ on inputs and precomputed values \cite{algebraiccomplexity}.} for an arbitrary complex-valued function $f$ defined on $G$. The \textbf{arithmetic complexity} of a Fourier transform on $R$, denoted $T_G(R)$, is given by $\max{(+_G(R), \times_G(R))}$.
\item[(ii)] The \textbf{complexity of the group} $G$, denoted $C(G)$ is defined by
$$C(G):=\min_R \{T_G(R)\},$$
where $R$ varies over all complete sets of inequivalent irreducible matrix representations of $G$.
\item[(iii)] The \textbf{reduced complexity}, denoted $t_G(R)$, is defined by $$t_G(R)=\frac{1}{\vert G\vert}T_G(R).$$
\end{itemize}
\end{definition}
Let $\rho_1,\dots,\rho_m$ be a complete set of inequivalent irreducible matrix representations of a group $G$ of dimensions $d_1,\dots,d_m,$ respectively. A direct computation of a Fourier transform would require at most $|G|\sum d_i^2=|G|^2$  arithmetic operations. Rewriting, for a direct computation we have 
$$C(G)\leq T_G(R)\leq \vert G\vert^2.$$
\textbf{Fast Fourier transforms} (FFTs) are algorithms for computing Fourier transforms that improve on this naive upper bound.  A priori, the number of operations needed to compute the Fourier transform may depend on the specific representations used. 

\begin{example}The classical DFT and FFT. 
For $G=C_N$, the cyclic group of order $N$, the irreducible representations are $1$-dimensional and defined by $\zeta_j\rightarrow\zeta_j^k,$ for $\zeta_j=e^{2\pi ij/N}$ and $k=0,\dots, N-1$ and $i=\sqrt{-1}$. The corresponding Fourier transform on $C_N$ is the usual discrete Fourier transform. Cooley and Tukey's algorithm showed that for a ``highly composite" integer $N$ (an integer $N$ that factors completely as a product of small prime numbers), $C(G)\leq O(N\log_2 N)$  \cite{cooleytukey}. 
\end{example}

The group algebra $\mathbb{C}[G]$ is the space of all formal complex linear combinations of group elements under the product
$$\left(\sum_{s\in G} f(s)s\right)\left(\sum_{t\in G} h(t)t\right)=\sum_{s,t\in G} f(s)h(t)st.$$

Elements of $\mathbb{C}[G]$ are in one-to-one correspondence with complex-valued functions on G, and the group algebra product corresponds to convolution of functions.

A complete set $R$ of inequivalent irreducible matrix representations of a group $G$ determines a basis for $\mathbb{C}[G]$ (via the irreducible matrix elements) and in this case the Fourier transform is an algebra isomorphism from $\mathbb{C}[G]$ to a direct sum of matrix algebras. We recover $f$ through the Fourier inversion formula.

\begin{theorem}[Fourier inversion (see e.g., \cite{diaconis})] Let $G$ be a finite group, $f$ a complex-valued function on $G$, and $R$ a complete set of inequivalent irreducible matrix representations of $G$. Then 
$$f(s)=\frac{1}{|G|}\sum_{\rho\in R}\dim_\rho \Trace(\hat{f}(\rho)\rho(s^{-1})).$$ 
\end{theorem}

Thus, the Fourier transform of a function $f$ on $G$ with respect to a complete set of inequivalent irreducible representations $R$ of $G$ is an algebra isomorphism 
$$ \mathbb{C}[G]\xrightarrow{\;\;\;\mathcal{F}_R\;\;\;}\bigoplus_{\rho\in R} M_{\dim(\rho)}(
\mathbb{C}),$$ and so as in \cite{clausen, maslen}:
\begin{lemma}\label{equivcomp}  The computation of the Fourier transform of a function $f$ on $G$ with respect to a complete set of irreducible  representations $R$ is equivalent to computation (rewriting) of
$$\sum_{s\in G} f(s)s$$ in the group algebra, relative to a fixed basis for $R$.
\end{lemma}


\subsection{Adapted bases, Bratteli diagrams, and quivers}\label{gel}
The fundamental idea of the SOV approach is a recasting of the Cooley-Tukey algorithm in terms of {\em graded quivers}, which is an elaboration of path algebras derived from Bratteli diagrams, which are motivated by the use of \textit{adapted} or \textit{Gel'fand-Tsetlin} bases for irreducible representations. 

\begin{definition} Given a group $G$ with subgroup $H\leq G$, a complete set $R$ of inequivalent irreducible matrix representations of $G$ is $\mathbf{H}$\textbf{-adapted} if there exists a complete set $R_H$ of inequivalent irreducible matrix representations of $H$ such that for all $\rho\in R$, $\rho\downarrow_H=\bigoplus \gamma_s$, for (not neccessarily distinct) representations $\gamma_s$ in $R_H$. The set $R$ is \textbf{adapted to the chain} $G=G_n> G_{n-1}>\cdots> G_0$ if for each $1\leq i\leq n$ there is a complete set $R_i$ of inequivalent representations of $G_i$ such that $R_i$ is $G_{i-1}$-adapted and $R_n=R$. A set of bases for the representation spaces that give rise to adapted representations is an \textbf{adapted basis}.
\end{definition}

For the FFT results of this paper we assume the ability to construct adapted sets of representations. This requirement is not a limitation, as any set of representations is equivalent to an adapted set of representations. One such construction is outlined in \cite{sovi}.

\begin{definition} A \textbf{quiver} $Q$ is a directed multigraph with vertex set $V(Q)$ and edge set $E(Q)$. For an arrow (directed edge) $e\in E(Q)$ from vertex $\beta$ to vertex $\alpha$, we call $\alpha$ the \textbf{target} of $e$ and $\beta$ the \textbf{source} of $e$. 
\end{definition}
%
%
%

Let $Q$ be a quiver. For each $e\in E(Q)$, let $t(e)$ denote the target of $e$ and $s(e)$ the source of $e$.
\begin{definition}
A quiver $Q$ is \textbf{graded} if there is a function $gr:V(Q)\rightarrow \mathbb{N}$ such that for each $e\in E(Q)$, $gr(t(e))>gr(s(e))$.  
\end{definition}

\begin{example}
Figure \ref{grquiv} is an example of a graded quiver. Each vertex $v$ is labeled by its grading, $gr(v)$.

\begin{figure}[H]\begin{center}
\begin{tikzpicture}[shorten >=1pt,node distance=2cm,on grid,auto,/tikz/initial text=] 
   \node[pnt] at (0, 0)(00){};
   \node[pnt] at (-2, 1)(11){};
   \node[pnt] at (-2,-1)(12){};
   \node[pnt] at (-4, 1)(21){};
   \node[pnt] at (-4, -1)(22){};
   \draw (00) node[above] {\footnotesize $0$};
   \draw (11) node[above] {\footnotesize $1$};
   \draw (12) node[above] {\footnotesize $1$};
   \draw (21) node[above] {\footnotesize $2$};
   \draw (22) node[above] {\footnotesize $2$};
    \path[->,every node/.style={font=\scriptsize}]
    (00) edge  node {} (11)
    (00) edge [bend left=55] node {} (11)
    (00) edge [bend right=55] node {} (21)
    (11) edge  node {} (21)
    (12) edge  node {} (21)
    (12) edge  node {} (22)
    (12) edge [bend left=55] node {} (22);
    
\end{tikzpicture}
\caption{A graded quiver.}
\label{grquiv}
\end{center}\end{figure}
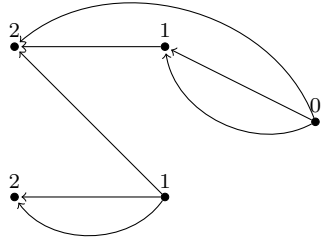
\end{example}

\begin{definition} A \textbf{Bratteli diagram} is a finite graded quiver such that: 
\begin{itemize}
\item[(i)] there is a unique vertex with grading $0$, called the \textbf{root},
\item[(ii)] if $v\in V(Q)$ is not the root then $v$ is the target of at least one arrow,
\item[(iii)] if $v\in V(Q)$ does not have grading of maximum value then $v$ is the source of at least one arrow,
\item[(iv)] for each $e\in E(Q)$, $gr(t(e))=1+gr(s(e))$. 
\end{itemize}
\end{definition}

\begin{example}
Note that the quiver of Figure \ref{grquiv} is not a Bratteli diagram. However, a slight modification produces the Bratteli diagram of Figure \ref{bratfirstex}. 

\begin{figure}[H]\begin{center}
\begin{tikzpicture}[shorten >=1pt,node distance=2cm,on grid,auto,/tikz/initial text=] 
   \node[pnt] at (0, 0)(00){};
   \node[pnt] at (-2, 1)(11){};
   \node[pnt] at (-2,-1)(12){};
   \node[pnt] at (-4, 1)(21){};
   \node[pnt] at (-4, -1)(22){};
   \draw (00) node[above] {\footnotesize $0$};
   \draw (11) node[above] {\footnotesize $1$};
   \draw (12) node[above] {\footnotesize $1$};
   \draw (21) node[above] {\footnotesize $2$};
   \draw (22) node[above] {\footnotesize $2$};
    \path[->,every node/.style={font=\scriptsize}]
    (00) edge  node {} (11)
    (00) edge node {} (12)
    (00) edge [bend left=55] node {} (11)
    (00) edge [bend right=55] node {} (11)
    (11) edge  node {} (21)
    (12) edge  node {} (21)
    (12) edge  node {} (22)
    (12) edge [bend left=55] node {} (22);
    
\end{tikzpicture}
\caption{A Bratteli diagram.}
\label{bratfirstex}
\end{center}\end{figure}
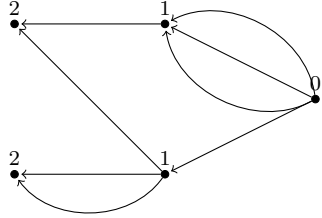
\end{example}
Consider a group algebra chain $\mathbb{C}[G_n] > \mathbb{C}[G_{n-1}] > \dots > \mathbb{C}[G_1]> \mathbb{C}[G_0] = \mathbb{C}$. To associate a Bratelli diagram to this chain we follow the language of \cite{seminormal}. Let $\rho$ be an irreducible representation of $G_{i}$, i.e., an irreducible $\mathbb{C}[G_i]$-module.  Upon restriction to $G_{i-1}$,  $\rho\downarrow_{G_{i-1}}$ decomposes as a direct sum of irreducible $\mathbb{C}[G_{i-1}]$-modules. For $\gamma$ an irreducible representation of $G_{i-1}$, let $M(\rho,\gamma)$ denote the multiplicity of $\gamma$ in $\rho\downarrow_{G_{i-1}}$. 
\begin{definition}\label{brattdef} For a chain of group algebras $\mathbb{C}[G_n] > \mathbb{C}[G_{n-1}] > \dots > \mathbb{C}[G_0],$ the \textbf{associated Bratteli diagram} is described by
\begin{itemize}
\item[(i)] The vertices of grading $i$ are labeled by the (equivalence classes of) irreducible representations of $G_i$;

\item[(ii)] A vertex labeled by an irreducible representation $\gamma$ of $G_{i-1}$ is connected to a vertex labeled by an irreducible representation $\rho$ of $G_{i}$ by $M(\rho,\gamma)$ arrows. 
\end{itemize} 

\end{definition}

\begin{example}  Figure \ref{BDexamples} shows two examples of Bratteli diagrams, with the gradings listed at the top. 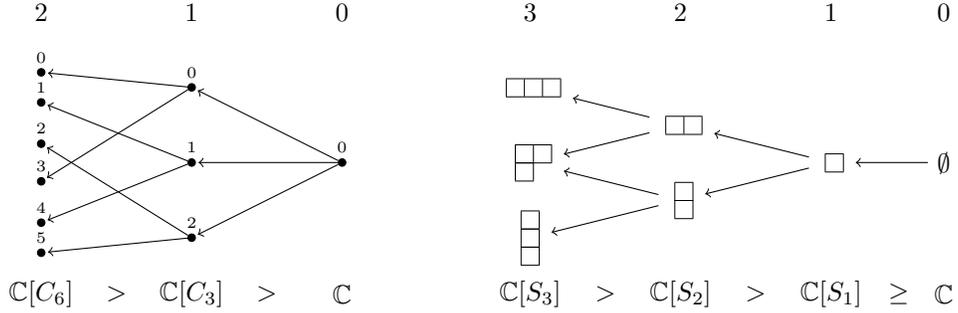
\begin{figure}[H]\begin{center}
\begin{tikzpicture}[shorten >=1pt,node distance=2cm,on grid,auto,/tikz/initial text=] 
   \node at (0,2) (q) {0};
   \node at (-2,2) (q) {1};
   \node at (-4,2) (q) {2};

   \node[pnt] at (0,0) (q_0) {};
   \node at (0,.2) (0) {\tiny{0}};
      \node at (0,-1.75) {$\mathbb{C}$};
   \node[pnt] at (-2,1) (q_1) {};
   \node at (-2,1.2) (1) {\tiny{0}};
   \node[pnt] at (-2,0) (q_2) {};
   \node at (-2,0.2) (2) {\tiny{1}};
   \node[pnt] at (-2,-1) (q_3) {};
   \node at (-2,-.8) (3) {\tiny{2}};
      \node at (-1,-1.75) {$>$};
      \node at (-2,-1.75) {$\mathbb{C}[C_3]$};
   \node[pnt]  at (-4,1.2)  (z_1) {};
   \node  at (-4,1.4)  (1) {\tiny{0}};
   \node[pnt]   at (-4,.8) (z_2)  {};
   \node   at (-4,1) (2)  {\tiny{1}};
   \node[pnt]   at (-4,.25) (z_3)   {};
   \node   at (-4,.45) (3)   {\tiny{2}};
   \node[pnt]   at (-4,-.25) (z_4)  {};
   \node   at (-4,-.05) (4)  {\tiny{3}};
   \node[pnt]   at (-4,-.8) (z_5)  {};
   \node   at (-4,-.6) (5)  {\tiny{4}};
   \node[pnt]   at (-4,-1.2) (z_6)  {};
   \node   at (-4,-1.0) (6)  {\tiny{5}};
\node at (-3,-1.75) {$>$};
\node at (-4,-1.75) {$\mathbb{C}[C_6]$};

\begin{scope}[shift={(-.5,0)}]
   \node at (8.5,2) (q) {0};
   \node at (7,2) (q) {1};
   \node at (5,2) (q) {2};
   \node at (3,2) (q) {3};

  \node at (8.5,0) (00) {$\emptyset$};
  \node at (7.9,-1.75) {$\geq$};
   \node at (8.5,-1.75) {$\mathbb{C}$};

  \node at (7,0) (11) {\tiny{ $\yng(1)$}};
   \node at (5,.5) (21) {\tiny{ $\yng(2)$}};
   \node at (5,-.5) (22) {\tiny{ $\yng(1,1)$}};
   \node at (3,1) (31) {\tiny{ $\yng(3)$}};
   \node at (3,0) (32) {\tiny{ $\yng(2,1)$}};
   \node at (3,-1) (33) {\tiny{$ \yng(1,1,1)$}};
   \node at (7,-1.75) {$\mathbb{C}[S_1]$};
   \node at (6,-1.75) {$>$};
   \node at (5,-1.75) {$\mathbb{C}[S_2]$};
   \node at (4,-1.75) {$>$};
   \node at (3,-1.75) {$\mathbb{C}[S_3]$};
\end{scope}
   \path[->]
    (00) edge node {} (11)
    (11) edge node {} (21)
    (11) edge node {} (22)
    (21) edge node {} (31)
    (21) edge node {} (32)
    (22) edge node {} (32)
    (22) edge node {} (33);

   \path[->,every node/.style={font=\scriptsize}]
   (q_0) edge (q_1)
   (q_0) edge (q_2)
   (q_0) edge (q_3) 
   (q_1) edge (z_1)
   (q_1) edge (z_4)
   (q_2) edge (z_2)
   (q_2) edge (z_5)
   (q_3) edge (z_3)
   (q_3) edge (z_6);

\end{tikzpicture}
\caption{Bratteli diagrams for $C_6$ (left) and $S_3$ (right).}\label{BDexamples}
\end{center}\end{figure}
On the left of Figure 3 we see the Bratteli diagram for a chain of group algebras for $C_6$ while on the right we see the Bratteli diagram for a chain of group algebras for the symmetric group $S_3$, viewing $S_i$ as the subgroup of $S_n$ that fixes the elements $\{i+1,\dots,n\}$.
Note that we distinguish $\mathbb{C}[S_1]$ from $\mathbb{C} (= C[S_0])$ only so that vertices at level $i$ correspond to representations of $\mathbb{C}[S_i]$. 

For the group algebra $\mathbb{C}[C_N]$, irreducible representations are naturally indexed  by the integers $0,\dots,N-1$, while for $\mathbb{C}[S_n]$, the irreducible representations are indexed by partitions of $n$ (as determined by Young in \cite{yg1}; see \cite{jameskerber} for an introduction to the representation theory of $S_n$). 

Both Bratteli diagrams of Figure \ref{BDexamples} are examples of \textbf{multiplicity-free} diagrams in that there is at most one edge from any vertex of grading $i$ to any vertex of grading $i+1$.

\end{example}

Given a Bratteli diagram $\mathcal{B}$, there is a canonical chain of algebras associated to $\mathcal{B}$ called the {\em chain of path algebras}.
\begin{definition} Let $\mathcal{B}$ be a Bratteli diagram. The \textbf{path algebra (at level i)}, denoted $\mathbb{C}[\mathcal{B}_i]$, is the $\mathbb{C}$-vector space with basis given by ordered pairs of paths  of length $i$ in $\mathcal{B}$ which start at the root and end at the same vertex at level $i$. 
\end{definition}

\begin{example}\label{pathalgsym}
In the Bratteli diagram $\mathcal{B}$ of Figure \ref{BDexamples2} associated to the chain $\mathbb{C}[S_3] >\mathbb{C}[ S_2] > \mathbb{C}[S_1]\geq \mathbb{C},$ let $P_1,P_2,P_3,P_4$ be the paths from the root to level $3$ in $\mathcal{B}$, labeled from top to bottom. Then the path algebra $\mathbb{C}[\mathcal{B}_3]$ has basis $\{(P_1,P_1), (P_2,P_2), (P_2,P_3), (P_3,P_2), (P_3,P_3), (P_4,P_4)\}.$

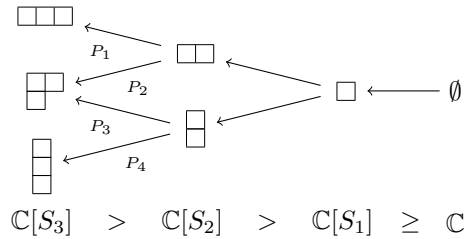
\begin{figure}[H]\begin{center}
\begin{tikzpicture}[shorten >=1pt,node distance=2cm,on grid,auto,/tikz/initial text=] 
   
  \node at (8.5,0) (00) {$\emptyset$};
  \node at (7.9,-1.75) {$\geq$};
   \node at (8.5,-1.75) {$\mathbb{C}$};

  \node at (7,0) (11) {\tiny{ $\yng(1)$}};
   \node at (5,.5) (21) {\tiny{ $\yng(2)$}};
   \node at (5,-.5) (22) {\tiny{ $\yng(1,1)$}};
   \node at (3,1) (31) {\tiny{ $\yng(3)$}};
   \node at (3,0) (32) {\tiny{ $\yng(2,1)$}};
   \node at (3,-1) (33) {\tiny{$ \yng(1,1,1)$}};
   \node at (7,-1.75) {$\mathbb{C}[S_1]$};
   \node at (6,-1.75) {$>$};
   \node at (5,-1.75) {$\mathbb{C}[S_2]$};
   \node at (4,-1.75) {$>$};
   \node at (3,-1.75) {$\mathbb{C}[S_3]$};

   \path[->]
    (00) edge node {} (11)
    (11) edge node {} (21)
    (11) edge node {} (22)
    (21) edge node {\tiny$P_1$} (31)
    (21) edge node {\tiny$P_2$} (32)
    (22) edge node {\tiny$P_3$} (32)
    (22) edge node {\tiny$P_4$} (33);

\end{tikzpicture}
\caption{Paths $P_1,P_2,P_3,P_4$ labeled according to their last steps. }\label{BDexamples2}
\end{center}\end{figure}
\end{example}
Note that for a vertex $v$, labeled by a representation $\rho$, the dimension of $\rho$ is given by the number of paths from the root to $v$. Moreover, each path corresponds to a subgroup-equivariant embedding of $\mathbb{C}$ into the representation space of $\rho$ (for more details, see Appendix \ref{appB}).

Further, $\mathbb{C}[\mathcal{B}_i]$ embeds into $\mathbb{C}[\mathcal{B}_{i+1}]$ as a subalgebra by mapping any pair of paths $(P,Q)\in\mathbb{C}[\mathcal{B}_i]$ to the sum $$\sum_e (e\circ P, e\circ Q),$$
over all arrows $e$ such that the source of $e$ is the target of $P$ (equivalently, of $Q$), and $\circ$ denotes concatenation of paths. Thus, elements in these subalgebras are effectively determined by the initial ``legs" (or ``bubbles") of their paths. This is also equivalent to a choice of basis  in the corresponding Wedderburn decomposition of the group algebra as a direct sum of matrix algebras, recognizing that for a given element, a number (equal to the total number of distinct paths that have the common middle ``source" of tail of $P$) of irreducible matrix elements will take on the same value. Identification of this kind of common ``unit" (formalized by the injection of one quiver into another) is the fundamental observation and technique of the quiver-based SOV approach. 

Multiplication in the path algebra $\mathbb{C}[\mathcal{B}_i]$ linearly extends $(P,Q)*(P',Q')=\delta_{QP'}(P,Q')$
$$\sum_{(P,Q)} a_{PQ}(P,Q)*\sum_{(P',Q')} b_{P'Q'}(P',Q')=\sum\left(\sum_{Q} a_{PQ}b_{QQ'}\right) (P,Q')$$
and is illustrated in Figure \ref{multinpath}. The first arrow represents gluing two pairs of paths along identical middle paths $Q=P'$ and the second arrow represents summation over all possible gluings.
\begin{figure}[H]\begin{center}
\begin{tikzpicture}[shorten >=1pt,node distance=2cm,on grid,auto,/tikz/initial text=] 
   \node[pnt] at (-3,0) (q_0) {};
   \node[pnt] (q_2) [left of=q_0] {};
   \node[pnt] (q_1) [below of=q_0] {};
   \node[pnt] (q_3) [below of=q_2] {};
   \node at (-2,-1)(y1){$\longrightarrow$};
   \node[pnt] at (-1,-1) (y4) {};
   \node[pnt] (y5) [right of=y4] {};
   \node at (2,-1)(y6){$\longrightarrow$};
   \node[pnt] at (3,-1) (y8) {};
   \node[pnt] (y9) [right of=y8] {};
   \path[->,every node/.style={font=\scriptsize}]
    (q_0) edge [bend left=55] node {$Q$} (q_2)
    (q_0) edge [bend right=55] node {$P$} (q_2)
    (q_1) edge [bend left=55] node {$Q'$} (q_3)
    (q_1) edge [bend right=55] node {$P'$} (q_3)
    (y5) edge [bend left=55] node {$Q'$} (y4)
    (y5) edge  node {$Q=P'$} (y4)
    (y5) edge [bend right=55] node {$P$} (y4)
    (y9) edge [bend left=55] node {$Q'$} (y8)
    (y9) edge  [bend right=55] node {$P$} (y8);    
\end{tikzpicture}
\caption{Multiplication in the path algebra.}
\label{multinpath}
\end{center}
\end{figure}
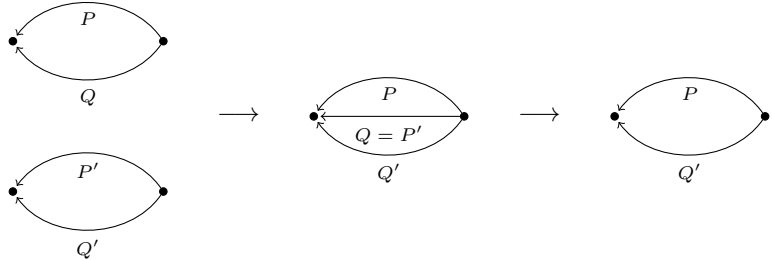

For a Bratteli diagram $\mathcal{B}$ with highest grading $n$ associated to a chain of group algebras, consider the associated chain of path algebras:
$\mathbb{C}[\mathcal{B}_n]>\mathbb{C}[\mathcal{B}_{n-1}]>\cdots> \mathbb{C}[\mathcal{B}_1]>\mathbb{C}[\mathcal{B}_0]=\mathbb{C}.$
It is not too difficult to see that there exists an isomorphism between these algebra chains.
\begin{lemma}\label{bratchain}Let $\mathbb{C}[G]=\mathbb{C}[G_n] > \mathbb{C}[G_{n-1}] > \dots > \mathbb{C}[G_1]> \mathbb{C}[G_0] = \mathbb{C}$ be a chain of group algebras with Bratteli diagram $\mathcal{B}$. Then the chain of path algebras associated to $\mathcal{B}$ is isomorphic to the group algebra chain.
\end{lemma}
For further explanation see Appendix \ref{appB} and Section 2.3 of \cite{towers}.
\begin{remark}  Quivers were first introduced by Gabriel  in the study of modular representation theory \cite{gabriel}. Bratteli diagrams were first introduced to classify inductive limits of $C^*$-algebras \cite{bratteli}. After Elliot's use of Bratteli diagrams in the classification of AF-algebras \cite{elliot}, these ideas motivated a program to classify $C^*$-algebras in terms of their \textit{K-theory} \cite{classification}. In terms of the representation theory of semisimple algebras, Bratteli diagrams have been used to explicitly construct complete sets of irreducible representations that are analogs of \textit{Young's seminormal form} in the symmetric group, and to describe restriction relations of representations \cite{orell, grood, halver, leduc-ram}. 
\end{remark}

\subsection{Gel'fand-Tsetlin bases}
The analogous concept in the path algebra of  
\newline
adapted bases associated to a group algebra chain is a \textit{system of Gel'fand Tsetlin bases}.
\begin{definition}\label{defgel} Let $\mathcal{B}$ be the Bratteli diagram associated to a chain of group algebras. A \textbf{system of Gel'fand-Tsetlin bases for} $\mathbf{\mathcal{B}}$ consists of a collection of bases for the representation spaces $\{V_\alpha\vert\;\alpha\in V(\mathcal{B})\}$ of the representations corresponding to $\alpha$ indexed by paths from the root to $\alpha$, along with maps from the paths to the basis vectors; i.e., a set of basis vectors along with knowledge of the path corresponding to each vector.
\end{definition}
\begin{example} Let $\mathcal{B}$  be the Bratteli diagram of Figure \ref{BDexamples} associated to the chain $\mathbb{C}[S_3] >\mathbb{C}[ S_2] > \mathbb{C}[S_1]\geq\mathbb{C}$. Then for the paths $P_1,P_2,P_3,P_4$ defined in Example \ref{pathalgsym}, a basis $\{w_{P_2},w_{P_3}\}$ for the two-dimensional representation space $V_{\tiny{\yng(2,1)}}$ is part of a system of Gel'fand-Tsetlin bases for $\mathcal{B}$. Note that the entries of the matrix of this representation are indexed by pairs $\{(w_{P_i},w_{P_j})\mid i,j=1,2\}$ and so correspond to basis elements of the path algebra $\mathbb{C}[\mathcal{B}_3]$. 
\end{example}

Systems of Gel'fand-Tsetlin bases were originally developed by Gel'fand and Tsetlin to calculate the matrix coefficients of compact groups \cite{tsetlin}. Clausen was the first to apply them to the efficient computation of Fourier transforms on finite groups \cite{clausen}. 

In Remark \ref{gel=adapt} of Appendix \ref{appB}, we show systems of Gel'fand-Tsetlin bases for the chain of path algebras corresponding to a group algebra chain are equivalent to adapted bases for the chain of subgroups. The notion of an adapted basis coincides with that of a set, for each $1\leq i\leq n$, of $G_i$-equivariant maps between the representation spaces of representations in $R_i$ and those in $R_{i+1}$. For further details, see Appendix \ref{appB}.

Gel'fand-Tsetlin bases provide a means to better understand the isomorphism of Lemma \ref{bratchain} between a chain of group algebras and the corresponding chain of path algebras. Since Gel'fand-Tsetlin bases are indexed by paths in $\mathcal{B}$ and a basis for the path algebra $\mathbb{C}[\mathcal{B}_n]$ consists of pairs of paths, we identify the group algebra $\mathbb{C}[G]$ with its realization in coordinates relative to the Gel'fand-Tsetlin basis, indexed by pairs of paths of length $n$ in $\mathcal{B}$ that share the same endpoint. For $s\in G$ let $\tilde{s}:=\sum_{(P,Q)\in\mathbb{C}[\mathcal{B}_n]} [s]_{P,Q}(P,Q).$ These are the coordinates of $s$ in the path algebra basis. 
Then Lemma \ref{equivcomp} becomes
\begin{lemma}\label{equivcomp2}  The computation of the Fourier transform of a function $f$ on a group $G$ with respect to a complete set of inequivalent irreducible representations $R$ is the same as computation of $$\sum_{s\in G}f(s)\tilde{s},$$  expressing it in terms of a Gel'fand Tsetlin basis for the path algebra $\mathbb{C}[\mathcal{B}_n]$ associated to $\mathbb{C}[G]$.
\end{lemma}

\begin{example}\label{ygorth}
Young's orthogonal form gives an example of a complete set of irreducible matrix representations for $S_n$ adapted to the chain $S_n>S_{n-1}>\cdots>S_1$. Since restriction of representations from $S_n$ to $S_{n-1}$ is multiplicity-free, the basis vectors of a system of Gel'fand-Tsetlin bases for the irreducible representations relative to this chain are determined up to scalar multiplies, and in the case of $n=3$, the paths are the paths $P_1, P_2, P_3, P_4$ of Example \ref{pathalgsym}. In \cite{maslen}, Maslen gives an efficient algorithm for computation of the Fourier transform of a function on $S_n$ by considering the computation of $\displaystyle\sum_{s\in S_n}f(s)s$ in the group algebra for $S_n$ relative to this Gel'fand-Tsetlin basis.   
\end{example}



\section{The Separation of Variables Approach}\label{sepvarstate}

In this section we describe the main components of the SOV approach. The heart of the idea involves expressing a path algebra element as a factorization over subsets of the Bratteli diagram in such a way as to disentangle the dependencies in the sum. To do so we first factor the Fourier transform through the subalgebras $\mathbb{C}[G_i]$. If we do this for a simple two-step chain, 
$\mathbb{C}[G]>\mathbb{C}[H]>\mathbb{C}$, we get a corresponding factorization (under the identification with the path algebra given by Lemma \ref{equivcomp2})
\begin{equation}
\mathcal{F}:=\sum_{s\in G}f(s)\tilde{s}= \sum_{y\in Y}\sum_{h\in H}f(yh)\tilde{y}\tilde{h}=\sum_{y\in Y}\tilde{y}\sum_{h\in H}f(yh)\tilde{h}=\sum_{y\in Y}\tilde{y}F_y,
\end{equation}
for $Y$ a set of coset representatives for $G/H$ such that for each $y\in Y,$
$$F_y=\sum_{h\in H}f_y(h)\tilde{h}\in\mathbb{C}[\mathcal{B}_{H}]$$ with $f_y(h):=f(yh)$.
This factorization allows us to obtain a simple, but key complexity estimate: given a set of coset representatives $Y$ for $G/H$ with $F_y$ (for each $y\in Y$) an arbitrary element of $\mathbb{C}[\mathcal{B}_H]$, define  
$$m_G(R,Y,H)=\frac{1}{\vert G\vert}\times
    \left\{
     \begin{array}{llll}
        \text{minimum number of operations required to compute }\\
    \sum_{y\in Y} \tilde{y}F_y\text{ in a system of Gel'fand-Tsetlin bases for } \mathcal{B}.\\
     \end{array} 
   \right. $$

\begin{lemma}\label{factorsum} 

Let $H$ be a subgroup of $G$, $R$ a complete $H$-adapted set of inequivalent irreducible matrix representations of $G$, and $Y\subseteq G$ a set of coset representatives for $G/H$. Let $\mathcal{B}$ be the Bratteli diagram of the group algebra chain $\mathbb{C}[G]>\mathbb{C}[H]>\mathbb{C}$, with corresponding path algebra chain $\mathbb{C}[\mathcal{B}_G]>\mathbb{C}[\mathcal{B}_H]>\mathbb{C}$. Then
$$t_G(R)\leq t_H(R_H)+m_G(R,Y,H).$$
\end{lemma}
\begin{proof}
 For $G$ a group with subgroup $H$, let $\mathcal{B}$ be the Bratteli diagram of the group algebra chain $\mathbb{C}[G]>\mathbb{C}[H]>\mathbb{C}$. Denote the path algebra chain by $\mathbb{C}[\mathcal{B}_G]>\mathbb{C}[\mathcal{B}_H]>\mathbb{C}$. Then by Lemma \ref{equivcomp}, computation of the Fourier transform of a function $f$ on $G$ at $R$ is equivalent to computation of $\mathcal{F}:=\sum_{s\in G} f(s)\tilde{s}$ in $\mathbb{C}[\mathcal{B}_G]$  expressing group algebra elements in coordinates relative to a Gel'fand-Tsetlin basis for $\mathcal{B}$. Let $H$ be a subgroup of $G$ and $Y\subseteq G$ a set of coset representatives for $G/H$. Then
\begin{equation}\label{yf_y}
\mathcal{F}:=\sum_{s\in G}f(s)\tilde{s}= \sum_{y\in Y}\sum_{h\in H}f(yh)\tilde{y}\tilde{h}=\sum_{y\in Y}\tilde{y}\sum_{h\in H}f(yh)\tilde{h}=\sum_{y\in Y}\tilde{y}F_y,
\end{equation}
where for each $y\in Y,$
$$F_y=\sum_{h\in H}f_y(h)\tilde{h}\in\mathbb{C}[\mathcal{B}_{H}]$$ with $f_y(h):=f(yh)$. Then to compute $\mathcal{F}$, first compute $F_y\in\mathbb{C}[\mathcal{B}_{H}]$ for all $y\in Y$ relative to a system of Gel'fand-Tsetlin bases for the chain $\mathbb{C}[\mathcal{B}_{H}]>\mathbb{C}$ corresponding to $R_H$, by means of a Fourier transform on $H$. This requires at most $\frac{ |G|}{|H|} T_H(R_H)$ scalar operations. Next, express the elements $F_y$ in coordinates relative to a system of Gel'fand-Tsetlin bases for the path algebra chain $\mathbb{C}[\mathcal{B}_{G}]>\mathbb{C}[\mathcal{B}_{H}]> \mathbb{C}$ corresponding to $R$. This requires no additional arithmetic operations. Finally, compute $\mathcal{F}$ using Equation (\ref{yf_y}), which (by definition) requires at most $\vert G\vert m_G(R,Y,H)$ operations. Thus,
$$T_G(R)\leq \frac{|G|}{|H|} T_H(R_H)+\vert G\vert m_G(R,Y,H),$$ and dividing by $\vert G\vert$ proves the lemma.

\end{proof}

Lemma \ref{factorsum} is a restatement of Lemma 2.10 of \cite{maslen} and Proposition 1 of \cite{diarock}.  It shows that  to compute the Fourier transform of a complex function defined on $G$ at a set of $H$-adapted representations, we compute
$$\mathcal{F}_Y:=\sum_{y\in Y}\tilde{y}F_y,$$ for $Y$ a set of coset representatives for $G/H$, or, equivalently, for ease of notation
$$\mathcal{F}_Y:=\sum_{y\in \tilde{Y}}yF_y,$$ for $\tilde{Y}=\{\tilde{y}\mid y\in Y\}$. (In the case of $C_N/C_n$ ($n | N$) this is basically the Cooley-Tukey algorithm.) In doing so, the complexity estimate ``reduces" to a close study of the computation of $F_Y$. This idea can  be iterated through a chain of subgroups assuming a set of representations $R$ adapted to a chain $G=G_n> G_{n-1}>\cdots> G_0=e$, let  $Y_i$ be a set of coset representatives for $G_i/G_{i-1}$. Iteration  of Lemma \ref{factorsum} gives \begin{equation}\label{iterate}t_G(R)\leq t_{G_0}(R_{G_0})+\sum_{i=1}^n M_{G_i}(R_{G_i},Y_i, G_{i-1}).\end{equation}




The heart of the SOV approach is the efficient computation of  $\mathcal{F}_Y$. It comprises  three main steps:
\begin{enumerate}
\item[1] For each $y\in \tilde{Y}$, factor $yF_y=x_1\cdots x_m$ in such a way as to enable rearrangements allowing for  $\mathcal{F}_Y$ to be a recursively structured summation.
\item[2] Each factor $x_i$ will correspond to an element of the path algebra of a particular form, and thus a particular subgraph of the Bratteli diagram. These subgraphs can be given a vector space structure through an identification with a space of quiver morphisms.
\item[3] By virtue of the vector space identification, the element multiplication $x_ix_{i+1}$ becomes  a bilinear map whose complexity can be calculated directly in terms of the dimension of the derived space of graph morphisms.
\end{enumerate}

To give the general idea, the ``gluing" and summing operations that are multiplication in the path algebra (cf. Figure \ref{multinpath}) mean that only certain kinds of ``middle paths" contribute when two path algebra elements are multiplied. I.e., only certain kinds of quivers can be combined to create the target quiver. A complexity estimate thus becomes counting the number of subgraphs (subquivers) wherein this compatibility is respected. This is just a 
%
counting of the number of occurrences of subquiver $\mathcal{Q}$ in the corresponding Bratteli diagram $\mathcal{B}$. Ultimately, this is the number of \emph{morphisms} from $\mathcal{Q}$ into $\mathcal{B}$ (see Definition \ref{defmorph}). We give a general example below. 
\begin{example}\label{firstpaths is it?}
Suppose $y\in \mathbb{C}[G]$ factors as $y=x_1x_2$ with $x_i\in \mathbb{C}[G_{i+2}]\cap\cent\mathbb{C}[G_{i}]$. Express $x_i$ in Gelf'and-Tsetlin coordinates as $\tilde{x}_i=\sum_{(P,Q)}[x_i]_{PQ}(P,Q)$. An application of Schur's Lemma and standard facts about Gel'fand-Tsetlin bases show that $[x_i]_{P,Q}$ is $0$ unless $P$ and $Q$ are paths in $\mathcal{B}$ that agree from level $i+1$ to level $n$, and from level $0$ to level $i-1$, as in the quivers $Q_i$ of the lefthand side of Figure \ref{correspondingones} (see also \cite{rockmassurvey}). The product $\tilde{x}_1\tilde{x}_2$ is indexed by any triple of paths resulting from gluing $Q_2$ to $Q_1$ obtained by identifying the ``bottom" path of $Q_1$ with the ``top" path of $Q_2$, but these triples must simultaneously maintain the structures of $Q_1$ and $Q_2$ (the quiver on the righthand side of Figure \ref{correspondingones}). The complexity count is thus the careful counting of these compatible structures, which can be recast as the computation of the dimension of a space of quiver morphisms. 
\begin{figure}[H]\begin{center}
\begin{tikzpicture}[shorten >=1pt,node distance=2cm,on grid,auto,/tikz/initial text=] 
  \node at (-4.25,-1) (q) {$Q_1$};
   \node[pnt] at (-2,0) (02) {};
   \node[pnt] at (-6.5,0) (n2) {};
   \node[pnt] at (-3,0) (11) {};
   \node at (-2.75,-.2) (l) {\footnotesize$1$};
   \node[pnt] at (-4.5,0) (31) {};
   \node at (-4.8,-.2) (l) {\footnotesize$3$};
   \node at (-3.75,0) (c) {\tiny{$\tilde{x}_1$}};
  \begin{scope}[shift={(0,-2)}] 
    \node at (-4.25,-1) (q) {$Q_2$};
   \node[pnt] at (-2,0) (021) {};
   \node[pnt] at (-6.5,0) (n21) {};
   \node[pnt] at (-3.5,0) (111) {};
   \node at (-3.25,-.2) (l) {\footnotesize$2$};
   \node[pnt] at (-5,0) (311) {};
   \node at (-5.3,-.2) (l) {\footnotesize$4$};
   \node at (-4.25,0) (c) {\tiny{$\tilde{x}_2$}};
   \end{scope}
   \draw (02) node[below] {\footnotesize $0$};
   \draw (n2) node[below] {\footnotesize $n$};

   \draw (021) node[below] {\footnotesize $0$};
   \draw (n21) node[below] {\footnotesize $n$};
   \path[->,every node/.style={font=\scriptsize}]
    (021) edge node {} (111)
    (311) edge node {} (n21)
    (111) edge [bend left=55] node {} (311)
    (111) edge [bend right=55] node {} (311)
    (02) edge node {} (11)
    (31) edge node {} (n2)
    (11) edge [bend left=55] node {} (31)
    (11) edge [bend right=55] node {} (31);  
\begin{scope}[shift={(7,-1)}] 
   \node[pnt] at (-2.5,0) (002) {};
   \node[pnt] at (-6,0) (442) {};
   \node[pnt] at (-7,0) (nn2) {};
   \node[pnt] at (-3.5,0) (111) {};
   \node at (-3.25,-.2) (l) {\footnotesize$1$};
   \node[pnt] at (-5,0) (331) {};
    \node[pnt] at (-4.25,-.37) (221) {};

   \node at (-5.3,.2) (l) {\footnotesize$3$};
   \node at (-4.25,0) (c) {\tiny{$\tilde{x}_1$}};
   \node at (-5.25,-.3) (c) {\tiny{$\tilde{x}_2$}};
   \end{scope}
   \draw (442) node[below] {\footnotesize $4$};
   \draw (221) node[below] {\footnotesize $2$};
   \draw (002) node[below] {\footnotesize $0$};
   \draw (nn2) node[below] {\footnotesize $n$};
   \path[->,every node/.style={font=\scriptsize}]
    (002) edge node {} (111)
    (331) edge node {} (nn2)
    (111) edge [bend left=55] node {} (331)
    (221) edge [bend left=55] node {} (442)
    (111) edge [bend right=55] node {} (331);  

\end{tikzpicture}
\caption{Examples of a quiver factorization. Note that $Q_1$ and $Q_2$ are both subquivers of the quiver on the righthand side.}
\label{correspondingones}
\end{center}\end{figure}
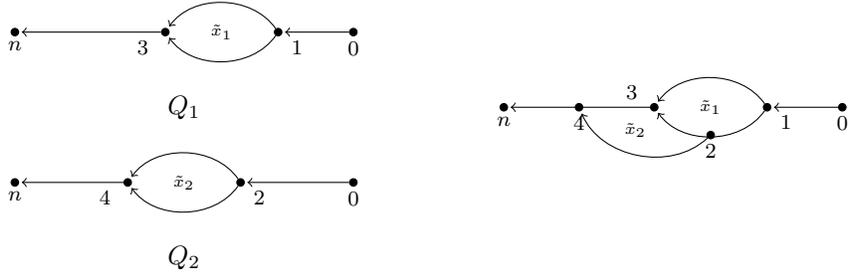

For products with more factors we iterate this gluing process.  Example \ref{quiv} below gives further details.  The SOV approach consists of factoring $\tilde{y}F_y$, forming the graph (akin to the righthand side of Fig.~\ref{correspondingones}) and determining the subgraphs (like the lefthand side of Fig.~\ref{correspondingones})) corresponding to each individual product.  

 \end{example}

\begin{definition}\label{widef} Let $\mathcal{B}$ be a Bratteli diagram with highest grading at least $n$ corresponding to a group chain for $G$. For a path algebra product $x_1\cdots x_m$,
let $i^+$ denote the smallest integer such that $x_i\in \mathbb{C}[\mathcal{B}_{i^+}]$ and let $i^-$ denote the largest integer less than or equal to $i^+$ such that $x_i\in\cent(\mathbb{C}[\mathcal{B}_{i^-}])$. Then for $1\leq i\leq m$ define $$X_i:=\mathbb{C}[\mathcal{B}_{i^+}]\cap\cent(\mathbb{C}[\mathcal{B}_{i^-}]).$$ \end{definition}

To each space $X_i$, associate the quiver $Q_i$ of Figure \ref{Qi}. (Note that $Q_i$ is also the quiver associated to every element of $X_i$.) We show in Section \ref{proofsofstuff} that $X_i$ has dimension equal to the number of occurrences of $Q_i$ in the Bratteli diagram $\mathcal{B}$. Denote this number by $\#\Hom(Q_i;\mathcal{B})$.  An ``occurrence" of $Q_i$ is the same as an injective map from $Q_i$ into $\mathcal{B}$. Thus, $\#\Hom(Q_i;\mathcal{B})$ is also the dimension of this space of morphisms of $Q_i$ into $\mathcal{B}$.

\begin{figure}[H]\begin{center}
\begin{tikzpicture}[shorten >=1pt,node distance=2cm,on grid,auto,/tikz/initial text=] 
  \node at (-4.25,-1) (q) {$Q_i$};
   \node[pnt] at (-2,0) (02) {};
   \node[pnt] at (-6.5,0) (n2) {};
   \node[pnt] at (-3.5,0) (11) {};
   \node at (-3.25,-.2) (l) {\footnotesize$i^-$};
   \node[pnt] at (-5,0) (31) {};
   \node at (-5.3,-.2) (l) {\footnotesize$i^+$};
   \draw (02) node[below] {\footnotesize $0$};
   \draw (n2) node[below] {\footnotesize $n$};
   \path[->,every node/.style={font=\scriptsize}]
    (02) edge node {} (11)
    (31) edge node {} (n2)
    (11) edge [bend left=55] node {} (31)
    (11) edge [bend right=55] node {} (31);  

\end{tikzpicture}
\caption{The quiver associated to $X_i$ and $x_i$.}
\label{Qi}
\end{center}\end{figure}
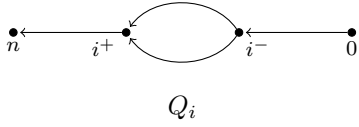

In this setting (bilinear) group algebra multiplication is transformed into a bilinear map on products of associated spaces of quiver morphisms.  Call this map $*$. 
As the notation and details are more technical than illuminating, we defer the explicit definition of $*$ and discussion of its properties to Section \ref{proofsofstuff}. However, even with deferring this we can present the algorithm. Keep in mind the identification of the  group algebra and the path algebra. 

\begin{palgorithm}\label{alg2}
\begin{itemize}
\item[]
\item[I.] Choose $m\in\mathbb{N}$ and a subset $X\subseteq (\mathbb{C}[\mathcal{B}_n])^m=\mathbb{C}[\mathcal{B}_n]\times\cdots\times\mathbb{C}[\mathcal{B}_n]$ such that $\vert X\vert=\vert Y\vert$ and for each $y\in Y$ there exists $(x_1,\dots, x_m)\in X$ with $\tilde{y}F_y=x_1\cdots x_m.$ Thus, $X$ can be thought of as a choice of factorization into   $m$ elements (some of which may be the identity) of each term $\tilde{y}F_y$. 
\item[II.] For $1\leq i\leq m$ let $X_i$ be as in Definition \ref{widef}. For $\sigma\in S_m$, let $w_i= x_{\sigma(i)}$. The   bilinear map $*$ 
is such that 
 $x_1\cdots x_m=(((w_1*w_2)*w_3)\cdots*w_m),$

\item[III.] For $  0\leq i<m$, let $W_i=\{(w_{i+1},\dots, w_m)\mid (x_1,\dots,x_m)\in X\}$. Let $
W_m=\emptyset. $ Note that $W_i\subseteq X_{\sigma(i+1)}\times\cdots\times X_{\sigma(m)}$.
\item[IV.] Define a sequence of functions $L_i$ recursively by:
$$\begin{array}{ll}
L_1(w_2,\dots,w_m)=&\displaystyle\sum_{(w_1,w_2,\dots,w_m)\in W_0} w_1, \\
L_2(w_{3},\dots,w_m)=&\displaystyle\sum_{(w_2,w_{3},\dots, w_m)\in W_{1}} L_{1}(w_2,\dots,w_m)*w_2.\\
 L_i(w_{i+1},\dots,w_m)=&\displaystyle\sum_{(w_i,w_{i+1},\dots, w_m)\in W_{i-1}} (L_{i-1}(w_i,\dots,w_m)*w_i).
\end{array}$$
\end{itemize}
\end{palgorithm}
\begin{theorem}\label{algthm} For $L_i$ as defined above,
$$L_m:=L_m(\emptyset)=\sum_{(w_1,\dots, w_m)\in W_0} (((w_1*w_2)*w_3)\cdots*w_m)=\sum_{y\in Y}\tilde{y}F_y$$
\end{theorem} 
\begin{proof}
Follows from II. and induction.
\end{proof}
\begin{example}\label{quiv} Suppose $\tilde{y}=x_1x_2x_3$, with $$\begin{array}{ll} 1^+=7,& 1^-=4,\\ 2^+=3,& 2^-=1,\\ 3^+=5,& 3^-=2.\end{array}$$ Figure \ref{exQ1} shows the quivers $Q_i$ and the quiver $\mathcal{Q}$ formed by gluing $Q_1$ to $Q_2$ to $Q_3$.
\begin{figure}[H]\begin{center}
\begin{tikzpicture}[shorten >=1pt,node distance=2cm,on grid,auto,/tikz/initial text=] 
   \node at (-5.5,0) (q) {$Q_{1}$};
   \node[pnt] at (-2,0) (02) {};
   \node[pnt] at (-5,0) (n2) {};
   \node[pnt] at (-2.5,0) (41) {};
   \node[pnt] at (-4.5,0) (71) {};
   \node at (-5.5,-1.5) (q) {$Q_{2}$};
   \node[pnt] at (-2,-1.5) (03) {};
   \node[pnt] at (-5,-1.5) (n3) {};
   \node[pnt] at (-2.5,-1.5) (12) {};
   \node[pnt] at (-4.5,-1.5) (32) {};   
\begin{scope}[shift={(0,-1.5)}] 
   \node at (-5.5,-1.5) (q) {$Q_{3}$};
   \node[pnt] at (-2,-1.5) (04) {};
   \node[pnt] at (-5,-1.5) (n4) {};
   \node[pnt] at (-2.5,-1.5) (23) {};
   \node[pnt] at (-4.5,-1.5) (53) {};   
\end{scope}
\begin{scope}[shift={(3,-2)}] 
   \node at (-.25,1.5) (q) {\large $\mathcal{Q}$};
   \node at (-.75,.5) (q) {\tiny $Q_{1}$};
   \node at (1,.25) (q) {\tiny $Q_{2}$};
   \node at (.25,0) (q) {\tiny $Q_{3}$};
   \node[pnt] at (1,0) (211) {};    
   \node[pnt] at (-.5,.1) (511) {};
   \node[pnt] at (2,.5) (011) {};
   \node[pnt] at (1.5,0.5) (111) {};
   \node[pnt] at (.5,.5) (311) {};
   \node[pnt] at (0,0.5) (411) {};
   \node[pnt] at (-1.5,.5) (711) {};
   \node[pnt] at (-2.5,.5) (n11) {};
\end{scope}

   \draw (41) node[below] {\footnotesize $4$};
   \draw (71) node[below] {\footnotesize $7$};
   \draw (12) node[below] {\footnotesize $1$};
   \draw (02) node[below] {\footnotesize $0$};
   \draw (n2) node[below] {\footnotesize $n$};
   \draw (32) node[below] {\footnotesize $3$};
   \draw (23) node[below] {\footnotesize $2$};
   \draw (03) node[below] {\footnotesize $0$};
   \draw (n3) node[below] {\footnotesize $n$};
   \draw (53) node[below] {\footnotesize $5$};
   \draw (04) node[below] {\footnotesize $0$};
   \draw (n4) node[below] {\footnotesize $n$};
   \draw (011) node[below] {\tiny $0$};
   \draw (111) node[below] {\tiny $1$};
   \draw (211) node[below] {\tiny $2$};
   \draw (311) node[below] {\tiny $3$};
   \draw (411) node[below] {\tiny $4$};
   \draw (511) node[below] {\tiny $5$};
   \draw (711) node[below] {\tiny $7$};
   \draw (n11) node[below] {\tiny $n$};
   \path[every node/.style={font=\scriptsize}]
    (02) edge node {} (41)
    (71) edge node {} (n2)
    (41) edge [bend left=55] node {} (71)
    (41) edge [bend right=55] node {} (71)    
    (03) edge node {} (12)
    (32) edge node {} (n3)
    (12) edge [bend left=55] node {} (32)
    (12) edge [bend right=55] node {} (32)    
    (04) edge node {} (23)
    (53) edge node {} (n4)
    (23) edge [bend left=55] node {} (53)
    (23) edge [bend right=55] node {} (53)
    (011) edge node {} (411)
    (711) edge node {} (n11)
    (111) edge [bend left=35] node {} (211)
    (211) edge [bend left=35] node {} (311)
    (411) edge [bend left=65] node {} (711)
    (411) edge [bend right=65] node {} (711)
    (211) edge [bend left=65] node {} (511);
\end{tikzpicture}
\caption{A triple product of quivers.}
\label{exQ1}
\end{center}\end{figure}
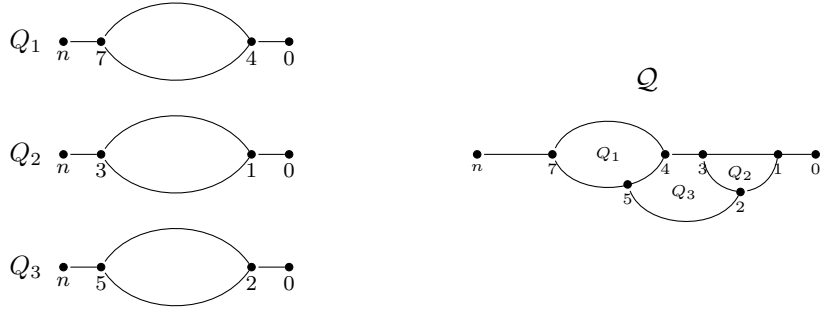
For $\sigma=(123)$, $w_1*w_2*w_3=x_2*x_3*x_1$. The complexity of $x_2*x_3$ is $\#\Hom(Q_2\cup Q_3;\mathcal{B})$, where $Q_2\cup Q_3$ is as in Figure \ref{exQ}, the subquiver of $\mathcal{Q}$ corresponding to $Q_2$ and $Q_3$ (note that in Figure \ref{exQ} we show only   the subquiver formed by the segments of $Q_2\cup Q_3$  where not all three -- top, bottom and the summed over middle -- of the paths agree).
The complexity of $(x_2*x_3)*x_1$ is $\#\Hom((Q_2\triangle Q_3)\cup Q_1;\mathcal{B})$, where $Q_2\triangle Q_3$ is the quiver of Figure \ref{exQ} associated to the space containing $x_2*x_3$. Note that as per the notation $Q_2\triangle Q_3$ is in fact the \textit{symmetric difference} of $Q_2$ and $Q_3$, i.e., the edges of $Q_2\cup Q_3$ not in $Q_2\cap Q_3$ (see Definition \ref{symdif}).

\begin{figure}[H]\begin{center}
\begin{tikzpicture}[shorten >=1pt,node distance=2cm,on grid,auto,/tikz/initial text=] 
  
\begin{scope}[shift={(3,-2)}] 
   \node at (0.5,1.5) (q) {\large $Q_2\cup Q_3$};
   \node[pnt] at (1,0) (211) {};    
   \node[pnt] at (-.5,.1) (511) {};
   \node[pnt] at (1.5,0.5) (111) {};
   \node[pnt] at (.5,.5) (311) {};
   \node[pnt] at (0,0.5) (411) {};
 
\end{scope}
\begin{scope}[shift={(10,-2)}] 
 \node at (0,1.5) (q) {\large $(Q_2\triangle Q_3)\cup Q_1$};
   \node[pnt] at (1,0) (212) {};    
   \node[pnt] at (-.5,.1) (512) {};
   \node[pnt] at (1.5,0.5) (112) {};
   \node[pnt] at (.5,.5) (312) {};
   \node[pnt] at (0,0.5) (412) {};
   \node[pnt] at (-1.5,.5) (712) {};
\end{scope}
   \draw (111) node[below] {\tiny $1$};
   \draw (211) node[below] {\tiny $2$};
   \draw (311) node[below] {\tiny $3$};
   \draw (411) node[below] {\tiny $4$};
   \draw (511) node[below] {\tiny $5$};
   
   \draw (112) node[below] {\tiny $1$};
   \draw (212) node[below] {\tiny $2$};
   \draw (312) node[below] {\tiny $3$};
   \draw (412) node[below] {\tiny $4$};
   \draw (512) node[below] {\tiny $5$};
   \draw (712) node[below] {\tiny $7$};
   \path[every node/.style={font=\scriptsize}]
    (111) edge node {} (411)
    (111) edge [bend left=35] node {} (211)
    (211) edge [bend left=35] node {} (311)
    (411) edge [bend left=45] node {} (511)
    (211) edge [bend left=65] node {} (511)
  
    (112) edge node {} (412)
    (112) edge [bend left=35] node {} (212)
    (412) edge [bend left=65] node {} (712)
    (412) edge [bend right=65] node {} (712)
    (212) edge [bend left=65] node {} (512);
\end{tikzpicture}
\caption{Illustrations of the union (left) and symmetric difference and union (right) of quivers of Figure 8.}
\label{exQ}
\end{center}\end{figure}
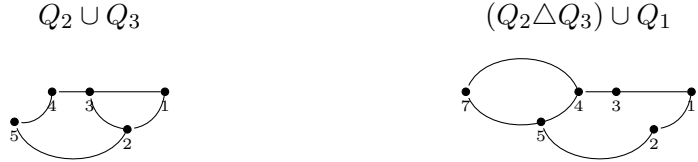
\end{example}

\begin{lemma}\label{counting}
For  $Q_i$ (respectively, $Q_j$) the quiver associated to $X_i$ (respectively, $X_j$), computation of $x_i*x_j$ requires at most $\#\Hom(Q_i\cup Q_j;\mathcal{B})$ scalar multiplications and fewer additions. \end{lemma} 
We postpone the proof of this key counting lemma to Section \ref{proofsofstuff}.
With Lemma \ref{counting} we now have our main general result: 

\begin{theorem}\label{efficiencyfirststate}
For $x_i$ and $\sigma$ as above, let $Q_i^\sigma$ denote the quiver associated to $w_i=x_{\sigma(i)}$. Then we may compute $\sum_{y\in Y} \tilde{y}F_y$ in at most $$\sum_{i=1}^{m-1}\vert W_{i-1}\vert\#\Hom((Q_1^\sigma\triangle \cdots \triangle Q_i^\sigma)\cup Q_{i+1}^\sigma;\mathcal{B})$$ multiplications and fewer additions.
\end{theorem}
\begin{proof}
To compute $\sum \tilde{y}F_y$, apply the SOV approach as follows:
\begin{itemize}
\item[]Stage $0$: Find $W_0$ by reordering $X$.
\item[]Stage $1$: Compute $L_1$ for all $(w_2,\dots,w_m)$ in $W_1$.
\item[]Stage $i$: Compute $L_i$ given $W_{i-1}$ and $L_{i-1}$. 
\end{itemize}
Stages $0$ and $1$ require no multiplications. For $2\leq i\leq m$,  condition (2) and the definition of $*$ implies that stage $i$ requires $\vert W_{i-1}\vert\#\Hom((Q_1^\sigma\triangle\cdots \triangle Q_i^\sigma)\cup Q_{i+1}^\sigma;\mathcal{B})$ multiplications.
\end{proof}
For an explicit additions count, see Theorem \ref{efficiency} in Section \ref{proofsofstuff}.

\section{The Complexity of Fourier Transforms on Finite Groups}\label{applications}
The SOV approach computes path algebra sums by first factoring each element and then translating multiplication into maps indexed by subgraphs. The complexity is determined by the size of the factorization sets and the number of occurrences of these subgraphs in the Bratteli diagram. Thus, our main results require methods to determine these counts. In this section we demonstrate the subgraphs determined by the SOV apporach and defer the proofs of the complexity counts to Section \ref{generalcounts} and the appendices. In this way we hope to give the visual sense (and attendant justification of the proofs) of the algorithm without an overload of technical notation.
\subsection{The Weyl Groups $B_n$ and $D_n$}
For our first application of the SOV approach we consider the Fourier transform of functions on the Weyl groups of type $B_n$ and $D_n$. We improve upon the results of \cite{sovi}.
\begin{theorem}[cf. Theorem \ref{Bnthm}]\label{Bnthm2}
Let $R$ be a complete set of irreducible matrix representations of (Weyl group) $B_n$ adapted to the subgroup chain
$B_n> B_{n-1}>\cdots> B_0=\{e\}.$ Then
$$C(B_n)\leq T_{B_n}(R)\leq n(2n-1)\vert B_n\vert.$$
\end{theorem}
\begin{proof}
Let $s_1,\dots,s_n$ denote the simple reflections for $B_n$, labeled as per the usual Dynkin diagram schema (see e.g., \cite{humphreys}) in Figure \ref{Bn}.
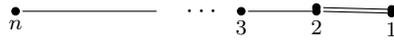
\begin{figure}[H]\begin{center}
\begin{tikzpicture}[shorten >=1pt,node distance=2cm,on grid,auto,/tikz/initial text=] 
   \node[pnt] at (6,.475) (1) {};
   \node[pnt] at (6,.525) (12) {};
   \node[pnt] at (5,.5) (2) {};
   \node[pnt] at (5,.55) (22) {};
   \node[pnt] at (4,.5) (3) {};
   \node at (3.5,.5) (q) {$\cdots$};   
   \node at (3,.5) (n-1) {};   
   \node[pnt] at (1,.5) (n) {};
   \draw (1) node[below] {\footnotesize $1$};
   \draw (2) node[below] {\footnotesize $2$};
   \draw (3) node[below] {\footnotesize $3$};
   \draw (n) node[below] {\footnotesize $n$};
   \path[every node/.style={font=\scriptsize}]
    (1) edge  node {} (2)
    (12) edge  node {} (22)
    (2) edge node {} (3)
    (n-1) edge node {} (n);               

\end{tikzpicture}
\caption{Dynkin diagram for $B_n$.}
\label{Bn}
\end{center}\end{figure}
Recall from \cite{sovi} that elements in a set of minimal coset representatives for $B_n/B_{n-1}$ have the following factorizations:
$$e,\;s_n,\;s_{n-1}s_n,\;\dots\; ,\;s_1\cdots s_n,\;s_2s_1\cdots s_n,\;\dots\;,\;s_n\cdots s_1\cdots s_n.$$
Then for $A_i=\{e,s_i\}=A_i'$, a complete set of coset representatives is contained in $Y=\{a_n\cdots a_2a_1a_2'\cdots a_n'\vert \; a_i,a_i'\in A_i\}$.   

Let $\mathcal{B}$ denote the Bratteli diagram associated to the chain 
$\mathbb{C}[B_{n}]> \mathbb{C}[B_{n-1}]>\cdots> \mathbb{C}$ and let $\{\mathbb{C}[\mathcal{B}_i]\}$ be the chain of path algebras associated to the chain $\{\mathbb{C}[B_{i}]\}$. Let $\tilde{Y}=\{\tilde{y}\mid y\in Y\}$, and similarly define $\tilde{A}_i, \tilde{A}_i'$ (where we continue to use $\tilde{}$ to denote the rewriting in path algebra coordinates). Note that $\tilde{A}_i, \tilde{A}_i'\subseteq \mathbb{C}[\mathcal{B}_i]\cap\cent( \mathbb{C}[\mathcal{B}_{i-2}])$. By Lemma \ref{factorsum}, computation of the Fourier transform of a complex function $f$ on $B_n$ is equivalent to computation of

$$\sum_{y\in\tilde{Y}} yF_y=\sum_{\substack{a_i\in \tilde{A}_i\\a_i'\in \tilde{A}_i'}} a_n\cdots a_2a_1a_2'\cdots a_n'F_{a_n\cdots a_2a_1a_2'\cdots a_n'}$$
for $F_y=F_{a_n\cdots a_2a_1a_2'\cdots a_n'}\in \mathbb{C}[\mathcal{B}_{n-1}]$. We now use the SOV Approach:
\begin{itemize}
\item[I.] Let $X=\{(a_n,\dots,a_2,a_1,a_2',\dots,a_n', F_{a_n\cdots a_2a_1a_2'\cdots a_n'})\vert\; a_i\in\tilde{A}_i,a_i'\in\tilde{A}_i'\}.$
\item[II.] Note that $$\begin{array}{ll} \displaystyle i^+=\max_{a_{n-i+1}\in\tilde{A}_{n-i+1}}\{c(a_{n-i+1})^+\}=n-i+1,&1\leq i\leq n,\\
\displaystyle i^+=\max_{a_{i-n+1}'\in\tilde{A}_{i-n+1}'}\{c(a_{i-n+1}')^+\}=i-n+1, &n<i<2n,\\
i^-=i^+-2,& 1\leq i< 2n,\\
{2n}^+=\max \{c(F_{a_n\cdots a_2a_1a_2'\cdots a_n'})^+\}=n-1,&\\
{2n}^-=0.&\end{array}$$ 

Fig. \ref{BnG} shows the various component subquivers corresponding to the coset representatives. They combine together as per Fig. 12 to give the factorization of $yF_y$. Thus, the algorithm proceeds by gluing together quivers $Q_{i}$ of Figure \ref{BnQ} (corresponding to elements of $\tilde{A}_j, \tilde{A}'_j, \text{ or }  F_y$, as per necessary) to build the quiver $\mathcal{Q}$ of Figure \ref{BnG}. The left column of Figure \ref{BnQ} shows the quivers $Q_{i}$ for $1\leq i\leq n$  and the right column shows the quivers $Q_{i}$ for $n+1\leq i\leq 2n$. 
\begin{figure}[H]\begin{center}
\begin{tikzpicture}[shorten >=1pt,node distance=2cm,on grid,auto,/tikz/initial text=,scale=0.7] 
   \node at (-9,11.5) (q) {$\tilde{A}_{n}$};
   \node[pnt] at (-2.5,11.5) (10) {};
   \node[pnt] at (-8,11.5) (1n) {};
   \node[pnt] at (-6,11.5) (1n-2) {};
   \node at (-9,10) (q) {$\tilde{A}_{n-1}$};
   \node[pnt] at (-2.5,10) (20) {};
   \node[pnt] at (-8,10) (2n) {};
   \node[pnt] at (-5,10) (2n-3) {};
   \node[pnt] at (-7,10) (2n-1) {};
   \node at (-5.5,8.5) (q) {$\vdots$};
\begin{scope}[shift={(0,-.5)}]
   \node at (-9,7.5) (q) {$\tilde{A}_3$};
   \node[pnt] at (-2.5,7.5) (30) {};
   \node[pnt] at (-8,7.5) (3n) {};
   \node[pnt] at (-3.5,7.5) (311) {};
   \node[pnt] at (-5.5,7.5) (333) {};
   \node at (-9,6) (q) {$\tilde{A}_2$};
   \node[pnt] at (-2.5,6) (40) {};
   \node[pnt] at (-8,6) (4n) {};
   \node[pnt] at (-4.5,6) (42) {};   
   \node at (-9,4.5) (q) {$\tilde{A}_1$};
   \node[pnt] at (-2.5,4.5) (50) {};
   \node[pnt] at (-8,4.5) (5n) {};
   \node[pnt] at (-3.5,4.5) (51) {};
\end{scope}
\begin{scope}[shift={(8,10)}] 
   \node at (-9,1.5) (q) {$\tilde{A}_2'$};
   \node[pnt] at (-2.5,1.5) (01) {};
   \node[pnt] at (-8,1.5) (n1) {};
   \node[pnt] at (-4.5,1.5) (21) {};
   \node at (-9,0) (q) {$\tilde{A}_3'$};
   \node[pnt] at (-2.5,0) (02) {};
   \node[pnt] at (-8,0) (n2) {};
   \node[pnt] at (-3.5,0) (11) {};
   \node[pnt] at (-5.5,0) (31) {};
   \node at (-5.5,-1.5) (q) {$\vdots$};   
\end{scope}   
\begin{scope}[shift={(0,11)}] 
   \node at (-1,-4) (q) {$\tilde{A}_{n-1}'$};
   \node[pnt] at (5.5,-4) (0new) {};
   \node[pnt] at (0,-4) (nnew) {};
   \node[pnt] at (3,-4) (n-3new) {};
   \node[pnt] at (1,-4) (n-1new) {};
   \node at (-1,-5.5) (q) {$\tilde{A}_{n}'$};
   \node[pnt] at (5.5,-5.5) (010) {};
   \node[pnt] at (0,-5.5) (n10) {};
   \node[pnt] at (2,-5.5) (n-21) {};
   \node at (-1,-7) (q) {$F_{y}$};
   \node[pnt] at (5.5,-7) (011) {};
   \node[pnt] at (0,-7) (n11) {};
   \node[pnt] at (1,-7) (n-11) {};
\end{scope}
   \draw (10) node[below] {\footnotesize $0$};
   \draw (1n) node[below] {\footnotesize $n$};
   \draw (1n-2) node[below] {\footnotesize $n-2$};
   \draw (20) node[below] {\footnotesize $0$};
   \draw (2n) node[below] {\footnotesize $n$};
   \draw (2n-1) node[below] {\footnotesize $n-1$};
   \draw (2n-3) node[below] {\footnotesize $n-3$};
   \draw (30) node[below] {\footnotesize $0$};
   \draw (3n) node[below] {\footnotesize $n$};
   \draw (333) node[below] {\footnotesize $3$};
   \draw (311) node[below] {\footnotesize $1$};
   \draw (40) node[below] {\footnotesize $0$};
   \draw (4n) node[below] {\footnotesize $n$};
   \draw (42) node[below] {\footnotesize $2$};
   \draw (50) node[below] {\footnotesize $0$};
   \draw (51) node[below] {\footnotesize $1$};
   \draw (5n) node[below] {\footnotesize $n$};
   \draw (01) node[below] {\footnotesize $0$};
   \draw (n1) node[below] {\footnotesize $n$};
   \draw (21) node[below] {\footnotesize $2$};
   \draw (02) node[below] {\footnotesize $0$};
   \draw (n2) node[below] {\footnotesize $n$};
   \draw (11) node[below] {\footnotesize $1$};
   \draw (31) node[below] {\footnotesize $3$};
   \draw (010) node[below] {\footnotesize $0$};
   \draw (n10) node[below] {\footnotesize $n$};
   \draw (n-21) node[below] {\footnotesize $n-2$};
   \draw (0new) node[below] {\footnotesize $0$};
   \draw (nnew) node[below] {\footnotesize $n$};
   \draw (n-1new) node[below] {\footnotesize $n-1$};
   \draw (n-3new) node[below] {\footnotesize $n-3$};
   \draw (011) node[below] {\footnotesize $0$};
   \draw (n11) node[below] {\footnotesize $n$};
   \draw (n-11) node[below] {\footnotesize $n-1$};
   \path[every node/.style={font=\scriptsize}]
    (10) edge node {} (1n-2)
    (1n-2) edge [bend left=55] node {} (1n)
    (1n-2) edge [bend right=55] node {} (1n)        
    (20) edge node {} (2n-3)
    (2n-3) edge [bend left=55] node {} (2n-1)
    (2n-3) edge [bend right=55] node {} (2n-1)        
    (2n-1) edge node {} (2n)
    (30) edge node {} (311)
    (311) edge [bend left=55] node {} (333)
    (311) edge [bend right=55] node {} (333)        
    (333) edge node {} (3n)
    (42) edge node {} (4n)
    (40) edge [bend left=55] node {} (42)
    (40) edge [bend right=55] node {} (42)
    (50) edge node {} (51)
    (51) edge node {} (5n)
    (5n) edge node {} (51)
    (21) edge node {} (n1)
    (01) edge [bend left=55] node {} (21)
    (01) edge [bend right=55] node {} (21)        
    (02) edge node {} (11)
    (31) edge node {} (n2)
    (11) edge [bend left=55] node {} (31)
    (11) edge [bend right=55] node {} (31)    
    (0new) edge node {} (n-3new)
    (n-1new) edge node {} (nnew)
    (n-3new) edge [bend left=55] node {} (n-1new)
    (n-3new) edge [bend right=55] node {} (n-1new)    
    (010) edge node {} (n-21)
    (n-21) edge [bend left=55] node {} (n10)
    (n-21) edge [bend right=55] node {} (n10)        
    (n-11) edge node {} (n11)
    (011) edge [bend left=25] node {} (n-11)
    (011) edge [bend right=25] node {} (n-11);    
    
\end{tikzpicture}
\caption{Component subquivers of the factorization of $yF_y$.}
\label{BnQ}
\end{center}\end{figure}
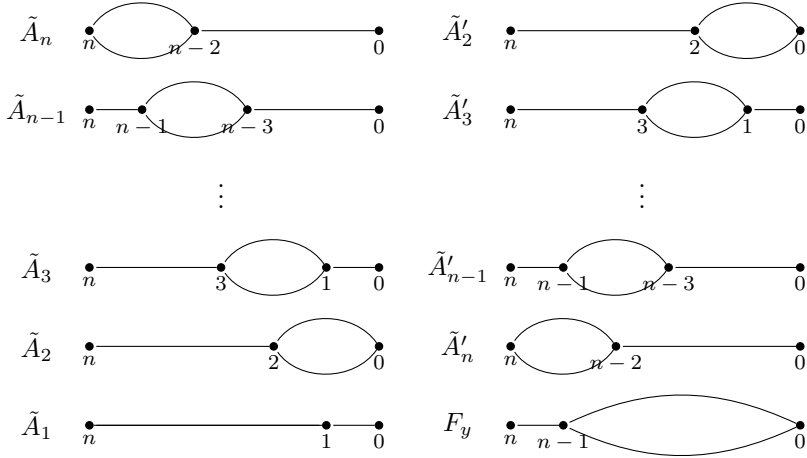
\begin{figure}\begin{center}
\begin{tikzpicture}[shorten >=1pt,node distance=2cm,on grid,auto,/tikz/initial text=] 
   
   \node at (1.25,1.5) (q) {\large $\mathcal{Q}$};
   \node at (4,.25) (q) {\tiny $\tilde{A}_3'$};
   \node at (3,.25) (q) {\tiny $\tilde{A}_4'$};
   \node at (2,.25) (q) {\tiny $\tilde{A}_5'$};
   \node at (5,.25) (q) {\tiny $\tilde{A}_2'$};
   \node at (4,.75) (q) {\tiny $\tilde{A}_3$};
   \node at (3,.75) (q) {\tiny $\tilde{A}_4$};
   \node at (2,.75) (q) {\tiny $\tilde{A}_5$};
   \node at (5,.75) (q) {\tiny $\tilde{A}_2$};   
   \node at (5.5,.5) (q) {\tiny $\tilde{A}_1$};   
   \node at (-1.5,.75) (q) {\tiny $\tilde{A}_{n-1}$};
   \node at (-.5,.75) (q) {\tiny $\tilde{A}_{n-2}$};
   \node at (-1.5,.25) (q) {\tiny $\tilde{A}_{n-1}'$};
   \node at (-.5,.25) (q) {\tiny $\tilde{A}_{n-2}'$};
   \node at (-2.5,.75) (q) {\tiny $\tilde{A}_{n}$};
   \node at (-2.5,.25) (q) {\tiny $\tilde{A}_n'$};
   \node at (1.25,-.75) (q) {\tiny $F_{y}$};
   \node[pnt] at (6,.5) (02) {};
   \node[pnt] at (5,0) (11) {};
   \node[pnt] at (5,.5) (12) {};
   \node[pnt] at (4,0) (21) {};
   \node[pnt] at (3,0) (31) {};
   \node[pnt] at (4,.5) (22) {};    
   \node[pnt] at (3,.5) (32) {};
   \node[pnt] at (2,.5) (42) {};
   \node[pnt] at (2,0) (41) {};
   \node[pnt] at (1,.5) (52) {};   
   \node[pnt] at (5,1) (13) {};
   \node[pnt] at (4,1) (23) {};
   \node[pnt] at (3,1) (33) {};
   \node[pnt] at (2,1) (43) {};
   \node at (.75,.25) (q) {$\dots$};
   \node at (.75,.75) (q) {$\dots$};
   \node[pnt] at (.5,0) (n-41) {};
   \node[pnt] at (-.5,.5) (n-32) {};
   \node[pnt] at (-.5,0) (n-31) {};
   \node[pnt] at (-1.5,0) (n-21) {};
   \node[pnt] at (-2.5,.5) (n-12) {};
   \node[pnt] at (-2.5,0) (n-11) {};
   \node[pnt] at (-3.5,.5) (n2) {};
   \node[pnt] at (-1.5,.5) (n-22) {};
   \node[pnt] at (.5,1) (n-43) {};
   \node[pnt] at (-.5,1) (n-33) {};
   \node[pnt] at (-1.5,1) (n-23) {};
   \node[pnt] at (-2.5,1) (n-13) {};
   \draw (11) node[below] {\tiny $1$};
   \draw (21) node[below] {\tiny $2$};
   \draw (31) node[below] {\tiny $3$};
   \draw (41) node[below] {\tiny $4$};
   \draw (02) node[below] {\tiny $0$};
   \draw (13) node[above] {\tiny $1$};
   \draw (23) node[above] {\tiny $2$};
   \draw (33) node[above] {\tiny $3$};
   \draw (43) node[above] {\tiny $4$};
   \draw (n-41) node[below] {\tiny $n-4$};
   \draw (n-21) node[below] {\tiny $n-2$};
   \draw (n-11) node[below] {\tiny $n-1$};
   \draw (n-13) node[above] {\tiny $n-1$};
   \draw (n2) node[below] {\tiny $n$};
   \draw (n-31) node[below] {\tiny $n-3$};
   \draw (n-43) node[above] {\tiny $n-4$};
   \draw (n-33) node[above] {\tiny $n-3$};
   \draw (n-23) node[above] {\tiny $n-2$};
   \path[every node/.style={font=\scriptsize}]
    (11) edge node {} (22)
    (11) edge node {} (21)
    (21) edge node {} (31)
    (02) edge node {} (22)
    (02) edge node {} (13)
    (11) edge node {} (02)
    (n-23) edge node {} (n-13)
    (n-21) edge node {} (n-11)
    (n-13) edge node {} (n2)
    (n-11) edge node {} (n2)    
    (n-12) edge node {} (n2)
    (21) edge node {} (32)
    (31) edge node {} (42)
    (32) edge node {} (42)
    (22) edge node {} (32)
    (42) edge node {} (52)
    (31) edge node {} (41)    
    (41) edge node {} (52)
    (13) edge node {} (43)
    (13) edge node {} (22)
    (23) edge node {} (32)
    (33) edge node {} (42)
    (43) edge node {} (52)
    (43) edge node {} (n-43)
    (n-43) edge node {} (n-23)
    (n-43) edge node {} (n-32)
    (n-33) edge node {} (n-22)
    (n-23) edge node {} (n-12)        
    (52) edge node {} (n-32)
    (41) edge node {} (n-41)
    (n-41) edge node {} (n-31)
    (n-41) edge node {} (n-32)
    (n-31) edge node {} (n-22)
    (n-31) edge node {} (n-21)
    (n-22) edge node {} (n-12)
    (n-32) edge node {} (n-22)    
    (n-21) edge node {} (n-12)
    (02) edge [bend left=45] node {} (n-11);

\end{tikzpicture}
\caption{The full quiver factorization of $yF_y$.}
\label{BnG}
\end{center}\end{figure}
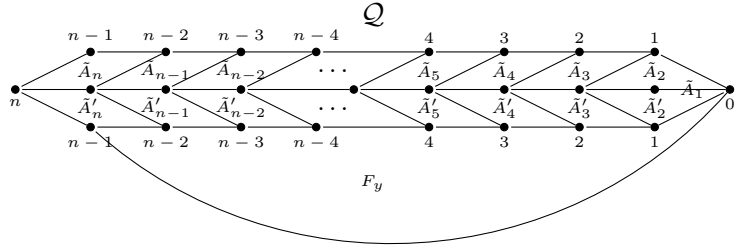
\item[III.] Let $\sigma\in S_{2n}$ be the permutation reordering $X$ so that $W_0$ is the set $\{(F_{a_n\cdots a_2a_1a_2\cdots a_n}, a_2',a_3',\dots a_n',a_1,a_2,a_3,\dots,a_n)\}.$ Then 
$$\begin{array}{l} W_1=\{(a_2',a_3',\dots a_n',a_1,a_2,a_3,\dots,a_n)\vert\; a_i\in \tilde{A}_i,a_i'\in \tilde{A}_i'\},\\
W_2=\{(a_3',\dots a_n',a_1,a_2,a_3,\dots,a_n)\vert\; a_i\in \tilde{A}_i, a_i'\in \tilde{A}_i'\},\\
\vdots\\
W_{2n-1}=\{(a_n)\vert\;a_n\in \tilde{A}_n\}.\end{array}$$
Note that
$$\begin{array}{ll}\vert W_{i-1}\vert=\vert \tilde{A}_i'\vert\cdots \vert \tilde{A}_n'\vert\vert \tilde{A}_1\vert\cdots \vert \tilde{A}_n\vert, &2\leq i\leq n,\\
\vert W_{i-1}\vert=\vert \tilde{A}_{i-n}\vert\cdots \vert \tilde{A}_n\vert,&n<i\leq 2n.\end{array}$$ 
\end{itemize}
By Theorem \ref{efficiencyfirststate}, we may compute $\sum yF_y$ in at most 
\begin{equation}\label{bythm2}\sum_{i=2}^{2n}\vert W_{i-1}\vert \#\Hom((Q_1^\sigma\triangle \cdots \triangle Q_i^\sigma)\cup Q_{i+1}^\sigma;\mathcal{B})\end{equation}
multiplications, with $(Q_1^\sigma\triangle\cdots \triangle Q_i^\sigma)\cup Q_{i+1}^\sigma$ as in Figure \ref{BnR}.
Thus, the complexity of the computation comes down to determining $\#\Hom((Q_1^\sigma\triangle \cdots \triangle Q_i^\sigma)\cup Q_{i+1}^\sigma;\mathcal{B})$, i.e., the number of occurrences of each quiver of Figure \ref{BnR} in the Bratteli diagram $\mathcal{B}$. Figure \ref{BnH} gives the general kinds of quivers that appear in Figure \ref{BnR}. The first $n$ quivers of Figure \ref{BnR} (the top row) have general form $\mathcal{H}_i^n$, as in Figure \ref{BnR}. The $n$th quiver (bottom left quiver of Figure \ref{BnR}) has form $\mathcal{K}^n$, while the remaining quivers have general form $\mathcal{J}_i^n$. 
\begin{figure}\begin{center}
\begin{tikzpicture}[shorten >=1pt,node distance=2cm,on grid,auto,/tikz/initial text=,scale=0.83] 
   \begin{scope}[shift={(-1,0)}]
   \node at (-3.75,1.5) (q) {$Q_1^\sigma\cup Q_2^\sigma$};
   \node at (-3,3.25) (q) {\tiny $\tilde{A}_2$};
   \node[pnt] at (-2,3.5) (00) {};
   \node[pnt] at (-3,3) (11) {};
   \node[pnt] at (-4,3.5) (12) {};
   \node[pnt] at (-4,3.5) (22) {};
   \node[pnt] at (-5.5,3) (n0) {};
   \end{scope}
   \begin{scope}[shift={(4,2.5)}]
   \node at (-4,-1) (q) {$(Q_1^\sigma\triangle Q_2^\sigma)\cup Q_3^\sigma$};
   \node at (-4,.75) (q) {\tiny $\tilde{A}_3$};
   \node[pnt] at (-2.25,1) (01) {};
   \node[pnt] at (-6,.5) (n1) {};
   \node[pnt] at (-3,.5) (13) {};
   \node[pnt] at (-4,.5) (23) {};
   \node[pnt] at (-4,1) (24) {};
   \node[pnt] at (-4.75,1) (34) {};
\end{scope}   
\node at (2.5,3) (q) {$\cdots$};
\begin{scope}[shift={(5.25,5.5)}] 
   \node at (-.25,-4) (q) {$(Q_1^\sigma\triangle \cdots \triangle Q_{n-1}^\sigma)\cup Q_{n}^\sigma$};
   \node at (-1,-2.25) (q) {\tiny $\tilde{A}_{n}$};
   \node[pnt] at (1.5,-2) (05) {};
   \node[pnt] at (-1.2,-2.5) (n-11) {};
   \node[pnt] at (-2,-2) (n5) {};
   \node[pnt] at (-1,-2) (n-12) {};
   \node[pnt] at (0,-2.5) (n-21) {};  
\end{scope} 
\begin{scope}[shift={(-5,-4)}] 
   \node at (0.25,1.5) (q) {$(Q_1^\sigma\triangle\cdots \triangle Q_n^\sigma)\cup Q_{n+1}^\sigma$};
   \node at (1.5,3.7) (q) {\tiny $\tilde{A}_{1}$};
   \node[pnt] at (2,3.5) (06) {};
   \node[pnt] at (1,3.5) (18) {};
   \node[pnt] at (-.5,3) (n-13) {};
   \node[pnt] at (-1.5,3.5) (n6) {};
\end{scope}
\begin{scope}[shift={(-2,-4)}]
   \node at (2,1.5) (q) {$(Q_1^\sigma\triangle\cdots \triangle Q_{n+1}^\sigma)\cup Q_{n+2}^\sigma$};
   \node at (2.75,3.75) (q) {\tiny $\tilde{A}_{2}'$};
   \node[pnt] at (3.75,3.5) (07) {};
   \node[pnt] at (2.75,3.5) (19) {};
   \node[pnt] at (2.75,4) (110) {};
   \node[pnt] at (2,3.5) (29) {};
   \node[pnt] at (1.25,3) (n-15) {};
   \node[pnt] at (.25,3.5) (n7) {};
\end{scope}
   \node at (2.5,-1) (q) {$\cdots$};
\begin{scope}[shift={(3.25,-3.5)}] 
   \begin{scope}[shift={(2,0)}]
  \node at (-.25,1) (q) {$(Q_1^\sigma\triangle \cdots \triangle Q_{2n-1}^\sigma)\cup Q_{2n}^\sigma$};
   \node at (-1,3.25) (q) {\tiny $\tilde{A}_{n}'$};
   \node[pnt] at (1.5,3) (09) {};
   \node[pnt] at (1,3.5) (115) {};
   \node[pnt] at (-1,2.5) (n-19) {};
   \node[pnt] at (-1,3) (n-110) {};
   \node[pnt] at (-1,3.5) (n-111) {};
   \node[pnt] at (0,3.5) (n-211) {};
   \node[pnt] at (-2,3) (n9) {};
\end{scope}
\end{scope}
   \draw (09) node[below] {\tiny $0$};
   \draw (n9) node[below] {\tiny $n$};
   \draw (n-19) node[below] {\tiny $n-1$};
   \draw (n-110) node[below] {\tiny $n-1$};
   \draw (n-111) node[above] {\tiny $n-1$};
   \draw (n-211) node[above] {\tiny $n-2$};
   \draw (07) node[below] {\tiny $0$};
   \draw (n7) node[below] {\tiny $n$};
   \draw (19) node[below] {\tiny $1$};
   \draw (110) node[above] {\tiny $1$};
   \draw (29) node[below] {\tiny $2$};
   \draw (n-15) node[below] {\tiny $n-1$};
   \draw (00) node[below] {\tiny $0$};
   \draw (n0) node[below] {\tiny $n-1$};
   \draw (01) node[above] {\tiny $0$};
   \draw (11) node[below] {\tiny $1$};
   \draw (12) node[above] {\tiny $1$};
   \draw (22) node[above] {\tiny $2$};
   \draw (01) node[above] {\tiny $0$};
   \draw (n1) node[below] {\tiny $n-1$};
   \draw (13) node[below] {\tiny $1$};
   \draw (23) node[below] {\tiny $2$};
   \draw (24) node[above] {\tiny $2$};
   \draw (34) node[above] {\tiny $3$};
   \draw (05) node[below] {\tiny $0$};
   \draw (n5) node[above] {\tiny $n$};
   \draw (n-21) node[below] {\tiny $n-2$};
   \draw (n-11) node[below] {\tiny $n-1$};
   \draw (n-12) node[above] {\tiny $n-1$};
   \draw (06) node[above] {\tiny $0$};
   \draw (18) node[above] {\tiny $1$};
   \draw (n-13) node[below] {\tiny $n-1$};
   \draw (n6) node[below] {\tiny $n$};
   \path[every node/.style={font=\scriptsize}]
    (00) edge node {} (11)
    (11) edge node {} (n0)
    (00) edge node {} (12)
    (12) edge node {} (22)
    (11) edge node {} (22)
    (00) edge [bend left=65] node {} (n0)
    (01) edge node {} (24)
    (13) edge node {} (n1)
    (24) edge node {} (34)
    (13) edge node {} (24)
    (23) edge node {} (34)
    (01) edge [bend left=65] node {} (n1)
    (05) edge node {} (n-12)
    (n-12) edge node {} (n5)
    (n-11) edge node {} (n5)
    (n-21) edge node {} (n-11)
    (n-21) edge node {} (n-12)
    (05) edge [bend left=65] node {} (n-11)
    (06) edge node {} (n6)
    (n-13) edge node {} (n6)
    (06) edge [bend left=65] node {} (n-13)
    (07) edge node {} (n7)
    (07) edge node {} (110)
    (110) edge node {} (29)
    (n-15) edge node {} (n7)
    (07) edge [bend left=65] node {} (n-15)
    (09) edge node {} (115)
    (115) edge node {} (n-211)
    (n9) edge node {} (n-111)
    (n-211) edge node {} (n-111)
    (n-110) edge node {} (n-211)
    (n-110) edge node {} (n9)
    (n-19) edge node {} (n9)
    (09) edge [bend left=65] node {} (n-19);
\end{tikzpicture}
\caption{Schematic of the stepwise aggregation of quivers as directed by SOV.}
\label{BnR}
\end{center}\end{figure}
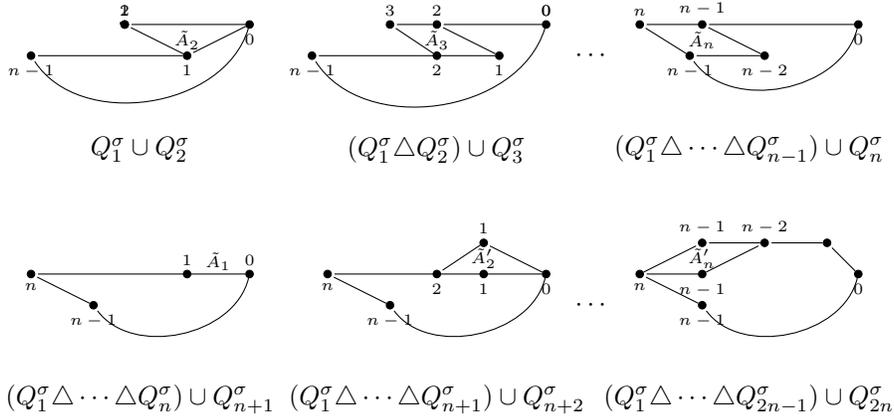
Then    
$$\begin{array}{ll}\#\Hom((Q_1^\sigma\triangle\cdots\triangle Q_{i-1}^\sigma)\cup Q_i^\sigma;\mathcal{B}))=\#\Hom(\mathcal{H}_i^n\uparrow\mathcal{Q};\mathcal{B}),&1< i\leq n,\\ \#\Hom ((Q_1^\sigma\triangle\cdots\triangle Q_{n}^\sigma)\cup Q_{n+1}^\sigma; \mathcal{B})=\#\Hom(\mathcal{K}^n\uparrow \mathcal{Q};\mathcal{B}),&\\
\#\Hom((Q_1^\sigma\triangle\cdots\triangle Q_{i-1}^\sigma)\cup Q_i^\sigma;\mathcal{B}))=\#\Hom(\mathcal{J}_{i-n}^n\uparrow \mathcal{Q};\mathcal{B}),&n+2\leq i\leq 2n,\end{array}$$
where for a subquiver $Q$ of $\mathcal{Q}$, $\Hom(Q\uparrow \mathcal{Q};\mathcal{B})$ denotes the number of quiver morphisms from $Q$ to $\mathcal{B}$ that extend to morphisms from $\mathcal{Q}$ to $\mathcal{B}$ (see Definition \ref{defmorph}).
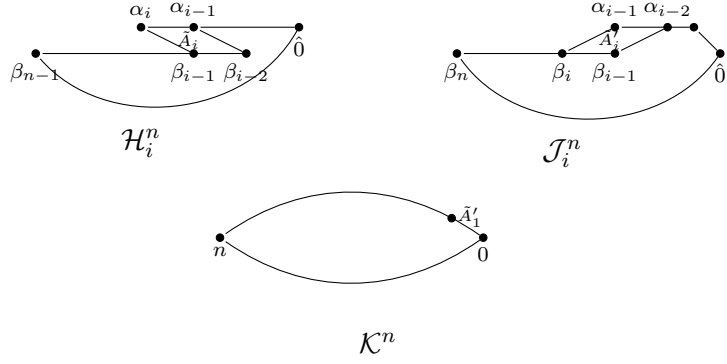
\begin{figure}\begin{center}
\begin{tikzpicture}[shorten >=1pt,node distance=2cm,on grid,auto,/tikz/initial text=,scale=0.7] 
   \node at (1.5,-1.2) (q) {\large $\mathcal{H}_i^n$};
   \node at (2.4,.75) (q) {\tiny $\tilde{A}_i$};
   \node[pnt] at (4.5,1) (04) {};
   \node[pnt] at (-.5,.5) (n4) {};
   \node[pnt] at (3.5,.5) (p5) {};
   \node[pnt] at (2.5,.5) (p+15) {};
   \node[pnt] at (2.5,1) (p+16) {};
   \node[pnt] at (1.5,1) (p+26) {};
\begin{scope}[shift={(8,-.5)}] 
   \node at (1.5,-.9) (q) {\large $\mathcal{J}_i^n$};
   \node at (2.4,1.25) (q) {\tiny $\tilde{A}_i'$};
   \node[pnt] at (4.5,1) (05) {};
   \node[pnt] at (4,1.5) (15) {};
   \node[pnt] at (-.5,1) (n5) {};
   \node[pnt] at (3.5,1.5) (i-26) {};
   \node[pnt] at (2.5,1) (i-15) {};
   \node[pnt] at (1.5,1) (i5) {};
   \node[pnt] at (2.5,1.5) (i-16) {};
\end{scope}
\begin{scope}[shift={(4.5,-4)}] 
   \node at (1.5,-1) (q) {\large $\mathcal{K}^n$};
   \node at (3.25,1.4) (q) {\tiny $\tilde{A}_1'$};
   \node[pnt] at (3.5,1) (06) {};
   \node[pnt] at (2.9,1.37) (16) {};
   \node[pnt] at (-1.5,1) (n6) {};
\end{scope}
   \draw (06) node[below] {\footnotesize $0$};
   \draw (n6) node[below] {\footnotesize $n$};
   \draw (05) node[below] {\footnotesize $\hat{0}$};
   \draw (n5) node[below] {\footnotesize $\beta_n$};
   \draw (i-26) node[above] {\footnotesize $\alpha_{i-2}$};
   \draw (i-15) node[below] {\footnotesize $\beta_{i-1}$};
   \draw (i-16) node[above] {\footnotesize $\alpha_{i-1}$};
   \draw (i5) node[below] {\footnotesize $\beta_{i}$};
   \draw (04) node[below] {\footnotesize $\hat{0}$};
   \draw (n4) node[below] {\footnotesize $\beta_{n-1}$};
   \draw (p5) node[below] {\footnotesize $\beta_{i-2}$};
   \draw (p+15) node[below] {\footnotesize $\beta_{i-1}$};
   \draw (p+16) node[above] {\footnotesize $\alpha_{i-1}$};
   \draw (p+26) node[above] {\footnotesize $\alpha_i$};
   \path[every node/.style={font=\scriptsize}]
    (04) edge node {} (p+16)
    (p+16) edge node {} (p+26)
    (p5) edge node {} (n4)
    (p5) edge node {} (p+16)
    (p+15) edge node {} (p+26)
    (04) edge [bend left=55] node {} (n4)
    (05) edge node {} (15)
    (15) edge node {} (i-16)
    (i-16) edge node {} (i5)
    (i-26) edge node {} (i-15)
    (i-15) edge node {} (n5)
    (06) edge [bend left=35] node {} (n6)                
    (06) edge [bend right=35] node {} (n6)                
    (05) edge [bend left=55] node {} (n5);                

\end{tikzpicture}
\caption{Various subquiver schema in the $B_n$ calculation.}
\label{BnH}
\end{center}\end{figure}
Lemma \ref{configspaces} in Section \ref{proofsofstuff} shows $$\#\Hom(\mathcal{K}^n\uparrow \mathcal{Q};\mathcal{B})=\vert B_n\vert.$$
In Appendix \ref{appC} we show:
$$\begin{array}{l}\#\Hom(\mathcal{J}_i^n\uparrow \mathcal{Q};\mathcal{B})\leq 2\vert B_n\vert,\\
\#\Hom(\mathcal{H}_i^n\uparrow \mathcal{Q};\mathcal{B})\leq \frac{4(i-1)}{n}\vert B_n\vert.\end{array}$$

Finally, note that since computing $*$ with the identity element $e$ requires no operations to compute (Corollary \ref{e}), for all $1\leq i\leq n$, $\vert \tilde{A}_i\vert=\vert \tilde{A}_i'\vert = 1.$

Plugging in to (\ref{bythm2}), we may compute $\sum_{y\in \tilde{Y}} yF_y$ in at most
$$\#\Hom(\mathcal{K}^n\uparrow \mathcal{Q};\mathcal{B}) + \sum_{i=2}^n \#\Hom(\mathcal{H}_i^n\uparrow \mathcal{Q};\mathcal{B}) + \sum_{j=2}^n \#\Hom(\mathcal{J}_j^n\uparrow \mathcal{Q};\mathcal{B})=(4n-3)\vert B_n\vert$$ 
multiplications (and fewer additions). By Lemma \ref{factorsum},
$$t_{B_n}(R)\leq t_{B_{n-1}}(R_{B_{n-1}})+4n-3,$$
and so $$t_{B_n}(R)\leq n(2n-1).$$
\end{proof}

Analogous arguments give the following result for Weyl groups of type $D_n$.
\begin{theorem}[cf. Theorem \ref{Dnthm}]
For the Weyl group $D_n$ and $R$ a complete set of irreducible matrix representations of $D_n$ adapted to the subgroup chain
$D_n> D_{n-1}>\cdots> D_0=\{e\},$
$$C(D_n)\leq T_{D_n}(R)\leq\frac{n(13n-11)}{2}\vert D_n\vert.$$  
\end{theorem}

\begin{proof}
Let $s_1,\dots,s_n$ denote the simple reflections for $D_n$, labeled according to its standard Dynkin diagram (see Figure \ref{Dn}).
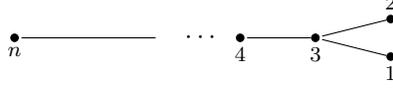
\begin{figure}[H]\begin{center}
\begin{tikzpicture}[shorten >=1pt,node distance=2cm,on grid,auto,/tikz/initial text=] 
   \node[pnt] at (6,.25) (1) {};
   \node[pnt] at (6,.75) (12) {};
   \node[pnt] at (5,.5) (2) {};
   \node[pnt] at (4,.5) (3) {};
   \node at (3.5,.5) (q) {$\cdots$};   
   \node at (3,.5) (n-1) {};   
   \node[pnt] at (1,.5) (n) {};
   \draw (1) node[below] {\footnotesize $1$};
   \draw (12) node[above] {\footnotesize $2$};
   \draw (2) node[below] {\footnotesize $3$};
   \draw (3) node[below] {\footnotesize $4$};
   \draw (n) node[below] {\footnotesize $n$};
   \path[every node/.style={font=\scriptsize}]
    (1) edge  node {} (2)
    (12) edge  node {} (2)
    (2) edge node {} (3)
    (n-1) edge node {} (n);               

\end{tikzpicture}
\caption{Dynkin diagram for $D_n$.}
\label{Dn}
\end{center}\end{figure}
Recall from \cite{sovi} that elements in a set of minimal coset representatives for $D_n/D_{n-1}$ have the following factorizations:
$$e,\;s_n,\;s_{n-1}s_n,\;\dots\;,\;s_3\cdots s_n,\; s_2s_3\cdots s_n,\;s_1s_3\cdots s_n,\;$$
$$s_1s_2s_3\cdots s_n,\;s_3s_1s_2s_3\cdots s_n,\;\dots\;,\;s_n\cdots s_3s_2s_1s_3\cdots s_n.$$
Then for $A_i=\{e,s_i\}=A_i'$, following the proof of Theorem \ref{Bnthm} shows we need only determine $\#\Hom(\mathcal{H}_i^n\uparrow \mathcal{Q};\mathcal{B}),$ $\#\Hom(\mathcal{J}_i^n\uparrow \mathcal{Q};\mathcal{B}),$ and $\#\Hom(\mathcal{K}^n\uparrow \mathcal{Q};\mathcal{B}),$
for $\mathcal{H}_i^n,\mathcal{J}_i^n,\mathcal{K}^n$ the quivers of Figure \ref{BnH}.
As before,
$$\#\Hom(\mathcal{K}^n\uparrow \mathcal{Q};\mathcal{B})=\vert D_n\vert.$$

By Lemma \ref{DnJlemma} and Corollary \ref{DnH3} of Appendix \ref{appC},
$$\begin{array}{l}\#\Hom(\mathcal{J}_i^n\uparrow \mathcal{Q};\mathcal{B})\leq 3\vert D_n\vert,\\
\#\Hom(\mathcal{H}_i^n\uparrow \mathcal{Q};\mathcal{B})\leq \displaystyle\frac{20(i-1)}{n}\vert D_n\vert,\end{array}$$
so by Theorem \ref{efficiencyfirststate} we may compute $\sum yF_y$ in at most 
$$\begin{array}{l}\#\Hom(\mathcal{K}^n\uparrow \mathcal{Q};\mathcal{B}) + \sum_{i=2}^n \#\Hom(\mathcal{H}_i^n\uparrow \mathcal{Q};\mathcal{B}) + \sum_{j=2}^n \#\Hom(\mathcal{J}_j^n\uparrow \mathcal{Q};\mathcal{B})\\=(13n-12)\vert D_n\vert\end{array}$$ 
multiplications (and fewer additions). Then by Lemma \ref{factorsum},
$$t_{D_n}(R)\leq t_{D_{n-1}}(R_{D_{n-1}}) +13n-12,$$
and so
$$t_{D_n}(R)\leq \frac{n(13n-11)}{2}.$$

\end{proof}

\subsection{The General Linear Group} 

Let $\mathbb{F}_q:=\mathbb{F}_{p^k}$ be a finite field of characteristic $p$ and order $q=p^k$. Let $Gl_n(q)$ denote the matrix group $Gl_n(\mathbb{F}_q)$ and consider $Gl_{n-1}(q)$ as a subgroup of $Gl_n(q)$ under the embedding $$A\rightarrow \begin{pmatrix}1&0\\0&A\end{pmatrix},$$ for $A\in Gl_{n-1}(q)$.
\begin{theorem}[cf. Theorem \ref{Gln}]  For the matrix group $Gl_n(q)$ and $R$ a complete set of irreducible matrix representations of $Gl_n(q)$ adapted to the subgroup chain $Gl_n(q)> Gl_{n-1}(q)>\cdots> \{e\}$

$$C(Gl_n(q))\leq T_{Gl_n(q)}(R)\leq \left(\frac{4^nq^{n+1}-q}{4q-1}\right)\vert Gl_n(q)\vert.$$
\end{theorem}
\begin{proof} Let $P$ be the set of permutation matrices of $Gl_n(q)$.
By Proposition \ref{easier} in Appendix \ref{appD}, for $p\neq 2$,  
$$Y=\{\pi s_i\vert\;1\leq i\leq n, (i-1)\text{ divisible by } p, \pi\in P\}$$
contains a complete set of coset representatives for $GL_n(q)/Gl_{n-1}(q)$, where $s_i$ has form 
$$u_2\cdots u_{p-1}u_{p+1}'t_pu_{p+1}\cdots u_{2p-1}u_{2p+1}'t_{2p}u_{2p+1}\cdots u_iv_{i+1}\cdots v_n,$$ for  $t_j$ the permutation matrix corresponding to $(j\; j-1)$, and $$u_j,u_j',v_j\in Gl_j(q)\cap\cent(Gl_{j-2}(q)),$$ with $(q-1)$ possible matrices for $u_j$ and $u_j'$, and $q^2$ possible matrices for $v_j$. 

Let $U_j$ (respectively $U_j'$, $V_j$) be the set of matrices $u_j$ (respectively $u_j'$, $v_j$), and let $T_j=\{t_j\}$. Let $\tilde{Y}=\{\tilde{y}\mid y\in Y\}$, and similarly define $\tilde{U}_j, \tilde{U}_j'$, $\tilde{V}_j$, $\tilde{T}_j,\tilde{P}$. Note that $\tilde{U}_j,\tilde{U}_j',\tilde{V}_j,\tilde{T}_j\in\mathbb{C}[\mathcal{B}_{j}]\cap\cent(\mathbb{C}[\mathcal{B}_{j-2}])$.

By Lemma \ref{factorsum} computation of the Fourier transform of a complex function $f$ on $GL_n(q)$ is equivalent to computation of:
 \begin{equation}\label{sum}\sum_{y\in\tilde{Y}} yF_y 
 \end{equation}
 for $F_{y}\in \mathbb{C}[\mathcal{B}_{n-1}]$. The number of operations to compute (\ref{sum}) is bounded by the number of operations to compute
$$\sum_{\pi}\sum_{\substack{1\leq i\leq n\\p\mid(i-1)}}\sum_{\substack{u_j, u_j',\\v_j, t_j}} \pi u_2\cdots v_{i+1}\cdots v_n F_{\pi u_2\cdots v_{i+1}\cdots v_n}$$
\begin{equation}\label{firstsum}=\sum_{\pi}\pi\sum_{\substack{1\leq i\leq n\\p\mid(i-1)}}\sum_{\substack{u_j, u_j'\\v_j, t_j}} u_2\cdots v_{i+1}\cdots v_n F_{\pi u_2\cdots v_{i+1}\cdots v_n}.\end{equation}

To compute (\ref{firstsum}), fix $\pi$ and
compute: 
\begin{equation}\label{sum1}\sum_{\substack{1\leq i\leq n\\p\mid(i-1)}}\sum_{\substack{u_j, u_j'\\v_j, t_j}} u_2\cdots v_{i+1}\cdots v_n F_{\pi u_2\cdots v_{i+1}\cdots v_n},
\end{equation}
then multiply by $\pi$ and sum. To compute sums of the form (\ref{sum1}):

\begin{itemize}
\item[I.]Let $X=\{(u_2,\dots,u_{p-1},u_{p+1}',t_p,u_{p+1},\dots,u_i,v_{i+1},\dots,v_n, F_{u_2\cdots v_n})\}$, ranging over $u_j\in\tilde{U}_j,u_j'\in\tilde{U}_j',v_j\in\tilde{V}_j,t_j\in\tilde{T}_j$
and $i\leq n$ with  $i\mid(p-1)\}.$
\item[II.] Fig. \ref{quiverstoglue2} shows the various component subquivers corresponding to the coset representatives. They combine together as per Fig. \ref{GlnG} to give the factorization of $yF_y$. Thus, the algorithm proceeds by gluing together quivers $Q_{i}$ of Figure \ref{quiverstoglue2} to build the quiver $\mathcal{Q}$ of Figure \ref{GlnG}. 
\begin{figure}\begin{center}
\begin{tikzpicture}[shorten >=1pt,node distance=2cm,on grid,auto,/tikz/initial text=] 
   \node at (-13,1.5) (q) {$\tilde{U}_2$};
   \node[pnt] at (-2,1.5) (01) {};
   \node[pnt] at (-12,1.5) (n1) {};
   \node[pnt] at (-3.5,1.5) (21) {};
   \node at (-13,0.5) (q) {$\tilde{U}_3$};
   \node[pnt] at (-2,0.5) (02) {};
   \node[pnt] at (-12,0.5) (n2) {};
   \node[pnt] at (-2.75,0.5) (11) {};
   \node[pnt] at (-4.25,0.5) (31) {};
   \node at (-7,0) (q) {$\vdots$};
  \begin{scope}[shift={(0,1.25)}] 
 \node at (-13,-2) (q) {$\tilde{U}_{p-1}$};
   \node[pnt] at (-2,-2) (03) {};
   \node[pnt] at (-12,-2) (n3) {};
   \node[pnt] at (-4.5,-2) (p-31) {};
   \node[pnt] at (-6,-2) (p-11) {};
   \node at (-13,-3.25) (q) {$\tilde{U}_{p+1}'$};
   \node[pnt] at (-2,-3.25) (04) {};
   \node[pnt] at (-12,-3.25) (n4) {};
   \node[pnt] at (-6,-3.25) (p-12) {};
   \node[pnt] at (-7.5,-3.25) (p+11) {};
   \node at (-13,-4.5) (q) {$\tilde{T}_{p}$};
   \node[pnt] at (-2,-4.5) (05) {};
   \node[pnt] at (-12,-4.5) (n5) {};
   \node[pnt] at (-5.25,-4.5) (p-21) {};
   \node[pnt] at (-6.75,-4.5) (p1) {};
\end{scope}
\begin{scope}[shift={(-5,2.25)}]
   \node at (-8,-6.5) (q) {$\tilde{U}_{p+1}$};
   \node[pnt] at (3,-6.5) (06) {};
   \node[pnt] at (-7,-6.5) (n6) {};
   \node[pnt] at (-1,-6.5) (p-13) {};
   \node[pnt] at (-2.5,-6.5) (p+12) {};
\end{scope}
\begin{scope}[shift={(-5,-7)}]
   \node at (-8,1.5) (q) {$\tilde{U}_{p+2}$};
   \node[pnt] at (3,1.5) (07) {};
   \node[pnt] at (-7,1.5) (n7) {};
   \node[pnt] at (-1.75,1.5) (p2) {};
   \node[pnt] at (-3.25,1.5) (p+21) {};
   \node at (-2,.75) (q) {$\vdots$};
   \end{scope}

   \draw (01) node[below] {\footnotesize $0$};
   \draw (n1) node[below] {\footnotesize $n$};
   \draw (21) node[below] {\footnotesize $2$};
   \draw (02) node[below] {\footnotesize $0$};
   \draw (n2) node[below] {\footnotesize $n$};
   \draw (11) node[below] {\footnotesize $1$};
   \draw (31) node[below] {\footnotesize $3$};
   \draw (03) node[below] {\footnotesize $0$};
   \draw (n3) node[below] {\footnotesize $n$};
   \draw (p-31) node[below] {\footnotesize $p-3$};
   \draw (p-11) node[below] {\footnotesize $p-1$};
   \draw (04) node[below] {\footnotesize $0$};
   \draw (n4) node[below] {\footnotesize $n$};
   \draw (p-12) node[below] {\footnotesize $p-1$};
   \draw (p+11) node[below] {\footnotesize $p+1$};
   \draw (05) node[below] {\footnotesize $0$};
   \draw (n5) node[below] {\footnotesize $n$};
   \draw (p-21) node[below] {\footnotesize $p-2$};
   \draw (p1) node[below] {\footnotesize $p$};
   \draw (06) node[below] {\footnotesize $0$};
   \draw (n6) node[below] {\footnotesize $n$};
   \draw (p-13) node[below] {\footnotesize $p-1$};
   \draw (p+12) node[below] {\footnotesize $p+1$};
   \draw (07) node[below] {\footnotesize $0$};
   \draw (n7) node[below] {\footnotesize $n$};
   \draw (p2) node[below] {\footnotesize $p$};
   \draw (p+21) node[below] {\footnotesize $p+2$};
   
   \path[every node/.style={font=\scriptsize}]
    (21) edge node {} (n1)
    (01) edge [bend left=45] node {} (21)
    (01) edge [bend right=45] node {} (21)        
    (02) edge node {} (11)
    (31) edge node {} (n2)
    (11) edge [bend left=45] node {} (31)
    (11) edge [bend right=45] node {} (31)    
    (03) edge node {} (p-31)
    (p-11) edge node {} (n3)
    (p-31) edge [bend left=45] node {} (p-11)
    (p-31) edge [bend right=45] node {} (p-11)    
    (04) edge node {} (p-12)
    (p+11) edge node {} (n4)
    (p-12) edge [bend left=45] node {} (p+11)
    (p-12) edge [bend right=45] node {} (p+11)
    (05) edge node {} (p-21)
    (p1) edge node {} (n5)
    (p-21) edge [bend left=45] node {} (p1)
    (p-21) edge [bend right=45] node {} (p1)
    (06) edge node {} (p-13)
    (p+12) edge node {} (n6)
    (p-13) edge [bend left=45] node {} (p+12)
    (p-13) edge [bend right=45] node {} (p+12)
    (07) edge node {} (p2)
    (p+21) edge node {} (n7)
    (p2) edge [bend left=45] node {} (p+21)
    (p2) edge [bend right=45] node {} (p+21);

\begin{scope}[shift={(-5,-6.5)}]
   
   \node at (-8,-.5) (q) {$\tilde{U}_{i}$};
   \node[pnt] at (3,-.5) (08) {};
   \node[pnt] at (-7,-.5) (n8) {};
   \node[pnt] at (-2.5,-.5) (i-21) {};
   \node[pnt] at (-4,-.5) (i1) {};
   \node at (-8,-1.5) (q) {$\tilde{V}_{i+1}$};
   \node[pnt] at (3,-1.5) (09) {};
   \node[pnt] at (-7,-1.5) (n9) {};
   \node[pnt] at (-3.25,-1.5) (i-11) {};
   \node[pnt] at (-4.75,-1.5) (i+11) {};
   \node at (-2,-2) (q) {$\vdots$};
  \begin{scope}[shift={(0,-.5)}]
  \node at (-8,-2) (q) {$\tilde{V}_{n}$};
   \node[pnt] at (3,-2) (010) {};
   \node[pnt] at (-7,-2) (n10) {};
   \node[pnt] at (-5.5,-2) (n-21) {};
   \node at (-8.2,-3) (q) {$F_{u_2\cdots v_n}$};
   \node[pnt] at (3,-3) (011) {};
   \node[pnt] at (-7,-3) (n11) {};
   \node[pnt] at (-4.25,-3) (n-11) {};
   \end{scope}
   
   \end{scope}
\draw (08) node[below] {\footnotesize $0$};
   \draw (n8) node[below] {\footnotesize $n$};
   \draw (i-21) node[below] {\footnotesize $i-2$};
   \draw (i1) node[below] {\footnotesize $i$};
   \draw (09) node[below] {\footnotesize $0$};
   \draw (n9) node[below] {\footnotesize $n$};
   \draw (i-11) node[below] {\footnotesize $i-1$};
   \draw (i+11) node[below] {\footnotesize $i+1$};
   \draw (010) node[below] {\footnotesize $0$};
   \draw (n10) node[below] {\footnotesize $n$};
   \draw (n-21) node[below] {\footnotesize $n-2$};
   \draw (011) node[below] {\footnotesize $0$};
   \draw (n11) node[below] {\footnotesize $n$};
   \draw (n-11) node[below] {\footnotesize $n-1$};
   
 \path[every node/.style={font=\scriptsize}]
 (08) edge node {} (i-21)
    (i1) edge node {} (n8)
    (i-21) edge [bend left=45] node {} (i1)
    (i-21) edge [bend right=45] node {} (i1)
    (09) edge node {} (i-11)
    (i+11) edge node {} (n9)
    (i-11) edge [bend left=45] node {} (i+11)
    (i-11) edge [bend right=45] node {} (i+11)
    (010) edge node {} (n-21)
    (n-21) edge [bend left=45] node {} (n10)
    (n-21) edge [bend right=45] node {} (n10)        
    (n-11) edge node {} (n11)
    (011) edge [bend left=10] node {} (n-11)
    (011) edge [bend right=10] node {} (n-11);    
  \end{tikzpicture}
\caption{Component subquivers of the factorization of $yF_y$.}
\label{quiverstoglue2}
\end{center}\end{figure}
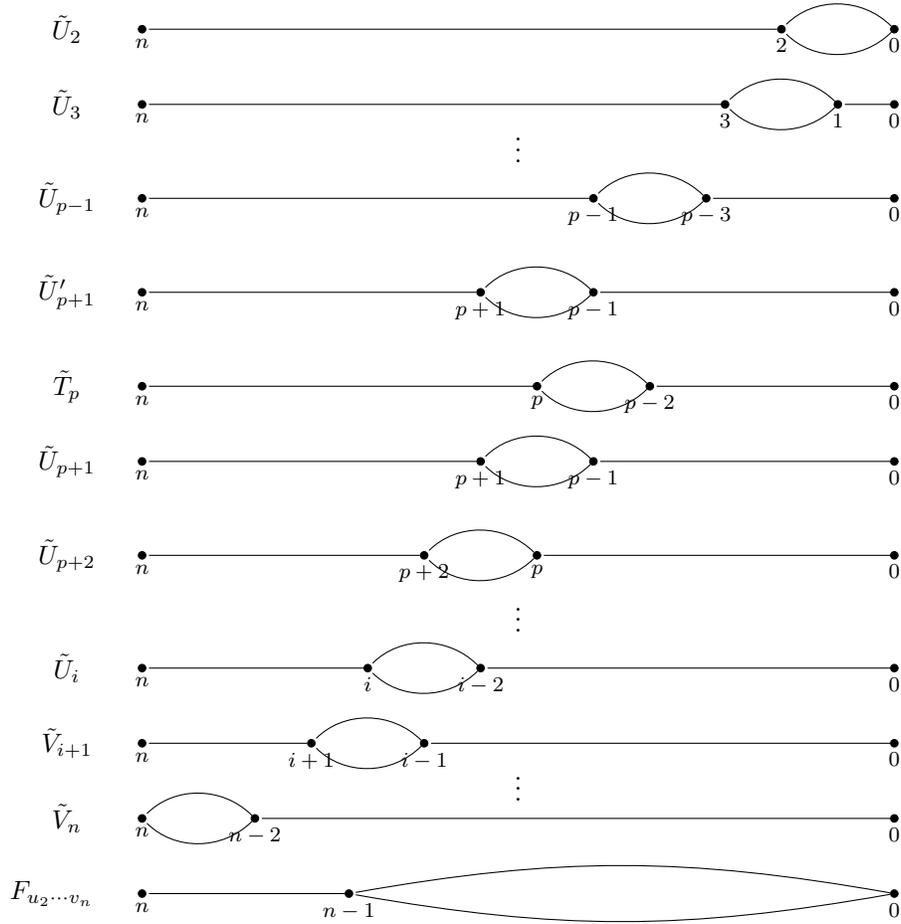  
\begin{figure}\begin{center}
\begin{tikzpicture}[shorten >=1pt,node distance=2cm,on grid,auto,/tikz/initial text=,scale=0.85] 
   
   \node at (-2,1.5) (q) {\large $\mathcal{Q}$};
   \node at (4.5,.25) (q) {\tiny $\tilde{U}_2$};
   \node at (3.75,.25) (q) {\tiny $\tilde{U}_3$};
   \node at (2,.25) (q) {\tiny $\tilde{U}_{p-1}$};
   \node at (1.25,.25) (q) {\tiny $\tilde{T}_{p}$};
   \node at (.5,.25) (q) {\tiny $\tilde{U}_{p+1}$};
   \node at (.5,.69) (q) {\tiny $\tilde{U}'_{p+1}$};
   \node at (-.25,.25) (q) {\tiny $\tilde{U}_{p+2}$};
   \node at (-2,.26) (q) {\tiny $\tilde{U}_{mp-1}$};
   \node at (-2.75,.24) (q) {\tiny $\tilde{T}_{mp}$};
   \node at (-3.5,.26) (q) {\tiny $\tilde{U}_{mp+1}$};
   \node at (-3.5,.69) (q) {\tiny $\tilde{U}'_{mp+1}$};
   \node at (-6,.25) (q) {\tiny $\tilde{U}_{i}$};
   \node at (-6.75,.25) (q) {\tiny $\tilde{V}_{i+1}$};
   \node at (-8.5,.25) (q) {\tiny $\tilde{V}_{n}$};
   \node at (-2,-1) (q) {\tiny $F_{u_2\cdots v_n}$};
   \node[pnt] at (5.25,0) (01) {};
   \node[pnt] at (4.5,0) (11) {};
   \node[pnt] at (3.75,0) (21) {};
   \node[pnt] at (4.5,.5) (12) {};    
   \node[pnt] at (3.75,.5) (22) {};
   \node[pnt] at (3,.5) (32) {};
   \node at (2.875,.25) (q) {$\dots$};
   \node[pnt] at (2.75,0) (p-31) {};
   \node[pnt] at (2,0) (p-21) {};
   \node[pnt] at (1.25,0) (p-11) {};
   \node[pnt] at (.5,0) (p1) {};
   \node[pnt] at (2,.5) (p-22) {};   
   \node[pnt] at (1.25,.5) (p-12) {};
   \node[pnt] at (.5,.5) (p2) {};
   \node[pnt] at (-.25,.5) (p+12) {};
   \node[pnt] at (-.25,0) (p+11) {};
   \node[pnt] at (-1,.5) (p+22) {};
   \node at (-1.125,.25) (q) {$\dots$};
   \node[pnt] at (-1.25,0) (mp-31) {};
   \node[pnt] at (-2,0) (mp-21) {};
   \node[pnt] at (-2.75,0) (mp-11) {};
   \node[pnt] at (-3.5,0) (mp1) {};
   \node[pnt] at (-2,.5) (mp-22) {};   
   \node[pnt] at (-2.75,.5) (mp-12) {};
   \node[pnt] at (-3.5,.5) (mp2) {};
   \node[pnt] at (-4.25,.5) (mp+12) {};
   \node[pnt] at (-4.25,0) (mp+11) {};
   \node[pnt] at (-5,.5) (mp+22) {};
   \node at (-5.125,.25) (q) {$\dots$};
   \node[pnt] at (-5.25,0) (i-21) {};
   \node[pnt] at (-6,0) (i-11) {};
   \node[pnt] at (-6.75,0) (i1) {};
   \node[pnt] at (-6,.5) (i-12) {};
   \node[pnt] at (-6.75,.5) (i2) {};
   \node[pnt] at (-7.5,.5) (i+12) {};
   \node at (-7.675,.25) (q) {$\dots$};
   \node[pnt] at (-7.75,0) (n-21) {};
   \node[pnt] at (-8.5,0) (n-11) {};
   \node[pnt] at (-9.25,.5) (n2) {};
   \node[pnt] at (-8.5,.5) (n-12) {};
   \draw (01) node[below] {\tiny $0$};
   \draw (11) node[below] {\tiny $1$};
   \draw (12) node[above] {\tiny $1$};
   \draw (21) node[below] {\tiny $2$};
   \draw (22) node[above] {\tiny $2$};
   \draw (32) node[above] {\tiny $3$};
   \draw (p-31) node[below] {\tiny $p-3$};
   \draw (p-22) node[above] {\tiny $p-2$};
   \draw (p-11) node[below] {\tiny $p-1$};
   \draw (p+11) node[below] {\tiny $p+1$};
   \draw (p+22) node[above] { \tiny $p+2$};
   \draw (mp-31) node[below] {\tiny $mp-3$};
   \draw (mp-22) node[above] {\tiny $mp-2$};
   \draw (mp-11) node[below] {\tiny $mp-1$};
   \draw (mp+11) node[below] {\tiny $mp+1$};
   \draw (mp+22) node[above] { \tiny $mp+2$};
   \draw (i-21) node[below] {\tiny $i-2$};
   \draw (i-12) node[above] {\tiny $i-1$};
   \draw (i1) node[below] {\tiny $i$};
   \draw (i+12) node[above] {\tiny $i+1$};
   \draw (n-12) node[above] {\tiny $n-1$};
   \draw (n-21) node[below] {\tiny $n-2$};
   \draw (n2) node[below] {\tiny $n$};
   \path[every node/.style={font=\scriptsize}]
    (01) edge node {} (12)
    (01) edge node {} (11)
    (11) edge node {} (21)
    (11) edge node {} (22)
    (21) edge node {} (32)
    (22) edge node {} (32)
    (12) edge node {} (22)
    (32) edge node {} (p-22)
    (21) edge node {} (p-31)    
    (p-31) edge node {} (p-21)
    (p-31) edge node {} (p-22)
    (p-21) edge node {} (p-11)
    (p-21) edge node {} (p-12)
    (p-11) edge node {} (p1)
    (p-11) edge node {} (p2)    
    (p-22) edge node {} (p-12)
    (p-12) edge [bend right=55] node {} (p+12)
    (p-12) edge node {} (p2)
    (p1) edge node {} (p+12)
    (p2) edge node {} (p+12)
    (p+12) edge node {} (22)
    (p+12) edge node {} (p+22)
    (p1) edge node {} (p+11)
    (p+11) edge node {} (p+22)
    (p+22) edge node {} (mp-22)
    (p+11) edge node {} (mp-31)    
    (mp-31) edge node {} (mp-21)
    (mp-31) edge node {} (mp-22)
    (mp-21) edge node {} (mp-11)
    (mp-21) edge node {} (mp-12)
    (mp-11) edge node {} (mp1)
    (mp-11) edge node {} (mp2)    
    (mp-22) edge node {} (mp-12)
    (mp-12) edge [bend right=55] node {} (mp+12)
    (mp-12) edge node {} (mp2)
    (mp1) edge node {} (mp+12)
    (mp2) edge node {} (mp+12)
    (mp+12) edge node {} (mp+22)
    (mp1) edge node {} (mp+11)
    (mp+11) edge node {} (mp+22)
    (i-21) edge node {} (i-11)
    (i-21) edge node {} (i-12)
    (i-11) edge node {} (i1)
    (i-11) edge node {} (i2)    
    (i-12) edge node {} (i2)
    (i1) edge node {} (i+12)
    (i2) edge node {} (i+12)
    (p+11) edge node {} (i-21)
    (p+22) edge node {} (i-12)
    (n-21) edge node {} (n-12)
    (n-21) edge node {} (n-11)
    (n-12) edge node {} (n2)    
    (i1) edge node {} (n-21)
    (i+12) edge node {} (n-12)    
    (n-11) edge node {} (n2)
    (01) edge [bend left=35] node {} (n-11);

\end{tikzpicture}
\caption{The full quiver factorization of $yF_y$.}
\label{GlnG}
\end{center}\end{figure}
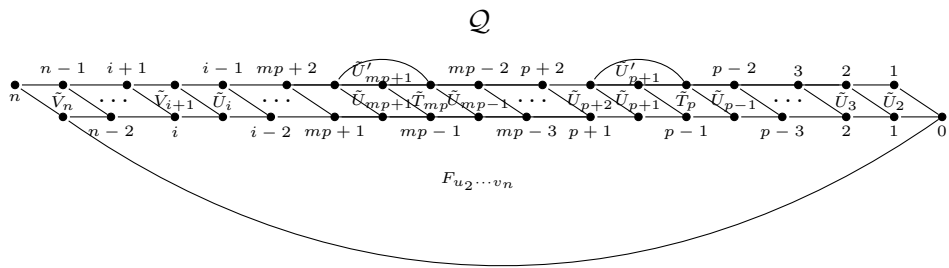

\item[III.] Let $\sigma\in S_{n+m-1}$ be the permutation reordering $X$ so that $$\begin{array}{l}W_0=\{( F_{u_2\cdots v_n}, u_2, u_3,\dots, u_{p-1}, t_p, u_{p+1}, u_{p+1}', u_{p+2},...,u_{i},v_{i+1},\dots,v_n)\}.\end{array}$$ Then
$$\begin{array}{l} W_1=\{(u_2, u_3,\dots, u_{p-1}, t_p, u_{p+1}, u_{p+1}', u_{p+2},...,u_{i},v_{i+1},\dots,v_n)\}\\
W_2=\{(u_3,\dots, u_{p-1}, t_p, u_{p+1}, u_{p+1}', u_{p+2},...,u_{i},v_{i+1},\dots,v_n)\}\\
\vdots\\
W_{m+n-2}=\{( v_n)\}.\end{array}$$
\end{itemize}

By Theorem \ref{efficiencyfirststate}, we may compute (\ref{sum1}) in at most $$\sum_{k=2}^{n+m-1}\vert W_{k-1}\vert \#\Hom((Q_1^\sigma\triangle\cdots\triangle Q_{k-1}^\sigma)\cup Q_{k}^\sigma;\mathcal{B})$$ multiplications, with $(Q_1^\sigma\triangle\cdots\triangle Q_{k-1}^\sigma)\cup Q_{k}^\sigma$ as in Figure \ref{GlnR}.
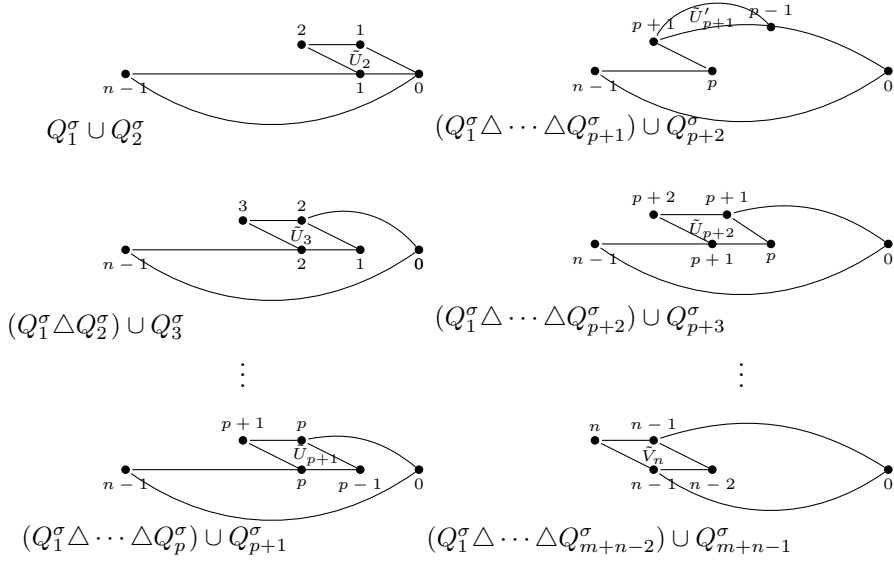
\begin{figure}\begin{center}
\begin{tikzpicture}[shorten >=1pt,node distance=2cm,on grid,auto,/tikz/initial text=,scale=0.78] 
   \node at (-7.5,2) (q) {$Q_1^\sigma\cup Q_{2}^\sigma$};
   \node at (-3,3.25) (q) {\tiny $\tilde{U}_2$};
   \node[pnt] at (-2,3) (00) {};
   \node[pnt] at (-3,3) (11) {};
   \node[pnt] at (-3,3.5) (12) {};
   \node[pnt] at (-4,3.5) (22) {};
   \node[pnt] at (-7,3) (n0) {};
\begin{scope}[shift={(0,-.5)}]
   \node at (-7.5,-.8) (q) {$(Q_1^\sigma\triangle Q_2^\sigma)\cup Q_3^\sigma$};
   \node at (-4,.75) (q) {\tiny $\tilde{U}_3$};
   \node[pnt] at (-2,.5) (01) {};
   \node[pnt] at (-7,.5) (n1) {};
   \node[pnt] at (-3,.5) (13) {};
   \node[pnt] at (-4,.5) (23) {};
   \node[pnt] at (-4,1) (24) {};
   \node[pnt] at (-5,1) (34) {};
   \node at (-5,-1.5) (q) {$\vdots$};
\end{scope}
\begin{scope}[shift={(0,-1.25)}]
  \node at (-6.5,-3.7) (q) {$(Q_1^\sigma\triangle\cdots\triangle Q_{p}^\sigma)\cup Q_{p+1}^\sigma$};
   \node at (-3.75,-2.25) (q) {\tiny $\tilde{U}_{p+1}$};
   \node[pnt] at (-2,-2.5) (02) {};
   \node[pnt] at (-7,-2.5) (n2) {};
   \node[pnt] at (-3,-2.5) (p-11) {};
   \node[pnt] at (-4,-2.5) (p1) {};
   \node[pnt] at (-4,-2) (p2) {};
   \node[pnt] at (-5,-2) (p+12) {};
  \end{scope}
  \begin{scope}[shift={(2,.25)}]
   \node at (-1.25,1.8) (q) {$(Q_1^\sigma\triangle\cdots\triangle Q_{p+1}^\sigma)\cup Q_{p+2}^\sigma$};
   \node at (1,3.7) (q) {\tiny $\tilde{U}_{p+1}'$};
   \node[pnt] at (4,2.8) (03) {};
   \node[pnt] at (2,3.55) (p-14) {};
   \node[pnt] at (0,3.3) (p+14) {};
   \node[pnt] at (1,2.8) (p3) {};
   \node[pnt] at (-1,2.8) (n3) {};
  \begin{scope}[shift={(1,-.65)}]
   \node at (-2.25,-.8) (q) {$(Q_1^\sigma\triangle\cdots\triangle Q_{p+2}^\sigma)\cup Q_{p+3}^\sigma$};
   \node at (0,.75) (q) {\tiny $\tilde{U}_{p+2}$};
   \node[pnt] at (3,.5) (04) {};
   \node[pnt] at (-2,.5) (n4) {};
   \node[pnt] at (1,.5) (p5) {};
   \node[pnt] at (0,.5) (p+15) {};
   \node[pnt] at (0.25,1) (p+16) {};
   \node[pnt] at (-1,1) (p+26) {};
  \end{scope}
  \node at (1.5,-2.25) (q) {$\vdots$};
    \begin{scope}[shift={(1,-1.5)}]
 \node at (-1.75,-3.7) (q) {$(Q_1^\sigma\triangle\cdots\triangle Q_{m+n-2}^\sigma)\cup Q_{m+n-1}^\sigma$};
   \node at (-1,-2.25) (q) {\tiny $\tilde{V}_{n}$};
   \node[pnt] at (3,-2.5) (05) {};
   \node[pnt] at (-1,-2.5) (n-11) {};
   \node[pnt] at (-2,-2) (n5) {};
   \node[pnt] at (-1,-2) (n-12) {};
   \node[pnt] at (0,-2.5) (n-21) {};  
  \end{scope}
    \end{scope}
   \draw (00) node[below] {\tiny $0$};
   \draw (n0) node[below] {\tiny $n-1$};
   \draw (01) node[below] {\tiny $0$};
   \draw (11) node[below] {\tiny $1$};
   \draw (12) node[above] {\tiny $1$};
   \draw (22) node[above] {\tiny $2$};
   \draw (01) node[below] {\tiny $0$};
   \draw (n1) node[below] {\tiny $n-1$};
   \draw (13) node[below] {\tiny $1$};
   \draw (23) node[below] {\tiny $2$};
   \draw (24) node[above] {\tiny $2$};
   \draw (34) node[above] {\tiny $3$};
   \draw (02) node[below] {\tiny $0$};
   \draw (n2) node[below] {\tiny $n-1$};
   \draw (p-11) node[below] {\tiny $p-1$};
   \draw (p1) node[below] {\tiny $p$};
   \draw (p2) node[above] {\tiny $p$};
   \draw (p+12) node[above] {\tiny $p+1$};
   \draw (03) node[below] {\tiny $0$};
   \draw (n3) node[below] {\tiny $n-1$};
   \draw (p-14) node[above] {\tiny $p-1$};
   \draw (p+14) node[above] {\tiny $p+1$};
   \draw (p3) node[below] {\tiny $p$};
   \draw (04) node[below] {\tiny $0$};
   \draw (n4) node[below] {\tiny $n-1$};
   \draw (p5) node[below] {\tiny $p$};
   \draw (p+15) node[below] {\tiny $p+1$};
   \draw (p+16) node[above] {\tiny $p+1$};
   \draw (p+26) node[above] {\tiny $p+2$};
   \draw (n-11) node[below] {\tiny $n-1$};
   \draw (n-12) node[above] {\tiny $n-1$};
   \draw (05) node[below] {\tiny $0$};
   \draw (n5) node[above] {\tiny $n$};
   \draw (n-21) node[below] {\tiny $n-2$};
   \path[every node/.style={font=\scriptsize}]
    (00) edge node {} (n0)
    (00) edge node {} (12)
    (12) edge node {} (22)
    (11) edge node {} (22)
    (00) edge [bend left=35] node {} (n0)
    (01) edge [bend right=35] node {} (24)
    (13) edge node {} (n1)
    (24) edge node {} (34)
    (13) edge node {} (24)
    (23) edge node {} (34)
    (01) edge [bend left=35] node {} (n1)
    (02) edge [bend right=25] node {} (p2)
    (p2) edge node {} (p+12)
    (p-11) edge node {} (n2)
    (p-11) edge node {} (p2)
    (p1) edge node {} (p+12)
    (02) edge [bend left=35] node {} (n2)
    (03) edge [bend left=35] node {} (n3)
    (03) edge [bend right=25] node {} (p+14)
    (p+14) edge node {} (p3)
    (p3) edge node {} (n3)
    (p-14) edge [bend right=55] node {} (p+14)
    (04) edge [bend right=25] node {} (p+16)
    (p+16) edge node {} (p+26)
    (p5) edge node {} (n4)
    (p5) edge node {} (p+16)
    (p+15) edge node {} (p+26)
    (04) edge [bend left=35] node {} (n4)
    (05) edge [bend right=25] node {} (n-12)
    (n-12) edge node {} (n5)
    (n-11) edge node {} (n5)
    (n-21) edge node {} (n-11)
    (n-21) edge node {} (n-12)
    (05) edge [bend left=35] node {} (n-11);
\end{tikzpicture}
\caption{Schematic of the stepwise aggregation of quivers as directed by SOV.}
\label{GlnR}
\end{center}\end{figure}
Then as in the proof of Theorem \ref{Bnthm2} for $\mathcal{H}_j^n,\mathcal{J}_j^n$ the quivers of Figure \ref{GlnFGH}, 
$$\#\Hom((Q_1^\sigma\triangle\cdots\triangle Q_{k-1}^\sigma)\cup Q_{k}^\sigma;\mathcal{B})=\#\Hom(\mathcal{H}_j^n\uparrow\mathcal{Q};\mathcal{B}) \text{ or } \#\Hom(\mathcal{J}_j^n\uparrow\mathcal{Q};\mathcal{B}).$$

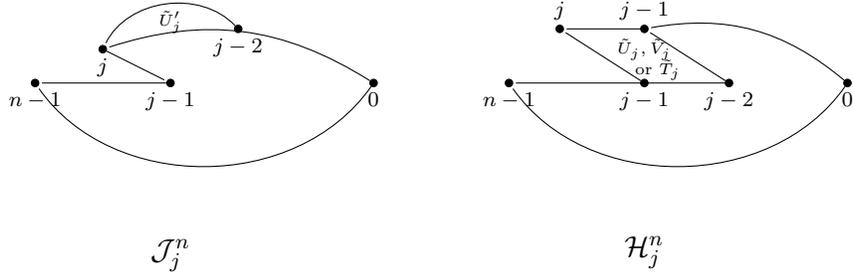
\begin{figure}\begin{center}
\begin{tikzpicture}[shorten >=1pt,node distance=2cm,on grid,auto,/tikz/initial text=,scale=0.9] 
   \node at (1,.25) (q) {\large $\mathcal{J}_j^n$};
   \node at (1,3.7) (q) {\tiny $\tilde{U}'_{j}$};
   \node[pnt] at (4,2.8) (03) {};
   \node[pnt] at (2,3.6) (p-14) {};
   \node[pnt] at (0,3.3) (p+14) {};
   \node[pnt] at (1,2.8) (p3) {};
   \node[pnt] at (-1,2.8) (n3) {};
\begin{scope}[shift={(7,2.3)}] 
  \node at (1,-2) (q) {\large $\mathcal{H}_j^n$};
   \node at (1,1) (q) {\tiny $\tilde{U}_j,  \tilde{V}_j$};
   \node at (1.2,.7) (q) {\tiny or $\tilde{T}_j$};
   \node[pnt] at (4,.5) (04) {};
   \node[pnt] at (-1,.5) (n4) {};
   \node[pnt] at (2.25,.5) (p5) {};
   \node[pnt] at (1,.5) (p+15) {};
   \node[pnt] at (1,1.3) (p+16) {};
   \node[pnt] at (-.25,1.3) (p+26) {};
   \end{scope}
   \draw (03) node[below] {\footnotesize $0$};
   \draw (n3) node[below] {\footnotesize $n-1$};
   \draw (p-14) node[below] {\footnotesize $j-2$};
   \draw (p+14) node[below] {\footnotesize $j$};
   \draw (p3) node[below] {\footnotesize $j-1$};
   \draw (04) node[below] {\footnotesize $0$};
   \draw (n4) node[below] {\footnotesize $n-1$};
   \draw (p5) node[below] {\footnotesize $j-2$};
   \draw (p+15) node[below] {\footnotesize $j-1$};
   \draw (p+16) node[above] {\footnotesize $j-1$};
   \draw (p+26) node[above] {\footnotesize $j$};
   \path[every node/.style={font=\scriptsize}]
    (03) edge [bend left=55] node {} (n3)
    (03) edge [bend right=25] node {} (p+14)
    (p+14) edge node {} (p3)
    (p3) edge node {} (n3)
    (p-14) edge [bend right=55] node {} (p+14)
    (04) edge [bend right=25] node {} (p+16)
    (p+16) edge node {} (p+26)
    (p5) edge node {} (n4)
    (p5) edge node {} (p+16)
    (p+15) edge node {} (p+26)
    (04) edge [bend left=55] node {} (n4);                

\end{tikzpicture}
\caption{Various subquiver schema in the $Gl_n$ calculation.}
\label{GlnFGH}
\end{center}\end{figure}

First consider the quiver $\mathcal{H}_j^n$ of Figure \ref{GlnFGH}, which corresponds to: 
$$\begin{array}{l} \tilde{U}_j, \text{ when }1\leq j\leq i,\; p\nmid j\\  \tilde{T}_j,  \text{ when }1\leq j\leq i, \; p\mid j,\\
\tilde{V}_j,  \text{ when } i<j\leq n.\end{array}$$ 

In Appendix \ref{appgenlin} we show
$$\begin{array}{l}\#\Hom(\mathcal{H}_j^n\uparrow \mathcal{Q};\mathcal{B})\leq 2^{2j-4}q^{j-2}\displaystyle\frac{q^{j-1}(q^j-1)}{q^{n-1}(q^n-1)}\vert Gl_{n}(q)\vert\end{array}.$$
Further, $$\vert \tilde{W}_j\vert:=
    \left\{
     \begin{array}{cccc}
        \#(u_j,\dots,u_i,v_{i+1},\dots,v_{n})\; & 1\leq j\leq i,\; p\nmid j,\\
        \#(t_j,\dots,u_i,v_{i+1},\dots,v_{n}) \;& 1\leq j\leq i, \; p\mid j,\\
         \#(v_j,\dots,v_n)\; & i<j\leq n. 
     \end{array}
   \right. $$
   
   In particular,
   $$\vert \tilde{W}_j\vert \leq\left\{
     \begin{array}{cccc}
        (q-1)^{i-j+1+m}(q^2)^{n-i} \; & 1\leq j\leq i,\\
        (q^2)^{n-j+1} \; & i<j\leq n, 
     \end{array}
     \right.$$
\noindent and so for all quivers $(Q_1^\sigma\triangle\cdots\triangle Q_{k-1}^\sigma)\cup Q_{k}^\sigma$ of form $\mathcal{H}_j^n$, 
 $$\vert W_{k-1}\vert \#\Hom((Q_1^\sigma\triangle\cdots\triangle Q_{k-1}^\sigma)\cup Q_{k}^\sigma;\mathcal{B})=\vert \tilde{W}_j\vert \#\Hom(\mathcal{H}_j^n,\mathcal{B})
%
%
     \leq 2^{2j-4}q^n\vert Gl_{n}(q)\vert.$$
In Appendix \ref{appgenlin} we also show
$$\#\Hom(\mathcal{J}_j\uparrow \mathcal{Q};\mathcal{B})\leq 2^{2j-3}q^{j-1}\frac{q^{j-1}(q^j-1)}{q^{n-1}(q^n-1)}\vert Gl_n(q)\vert.$$
Further, $$\vert \tilde{W}_j\vert :=\#(u_j',\dots,u_i,v_{i+1},\dots,v_n)\leq\left\{
     \begin{array}{cccc}
        (q-1)^{i-j+1+m}(q^2)^{n-i} \; & 1\leq j\leq i,\\
        (q^2)^{n-j+1} \; & i<j\leq n, 
     \end{array}
     \right. $$
and so for all quivers $(Q_1^\sigma\triangle\cdots\triangle Q_{k-1}^\sigma)\cup Q_{k}^\sigma$ of form $\mathcal{J}_j^n$, 
\begin{align*}
\vert W_{k-1}\vert \#\Hom((Q_1^\sigma\triangle\cdots\triangle Q_{k-1}^\sigma)\cup Q_{k}^\sigma;\mathcal{B})&=\vert \tilde{W}_j\vert \dim A(\mathcal{J}_j\uparrow \mathcal{Q};\mathcal{B})\\
&\leq 2^{2j-3}q^n\vert Gl_n(q)\vert.
\end{align*}

Thus, by Theorem \ref{efficiencyfirststate}, we may compute (\ref{sum1}) in at most
\begin{align*}
\sum_{k=1}^n|\tilde{W}_k|\#\Hom(\mathcal{H}_k^n\mathcal{Q};\mathcal{B})+\sum_{l=1}^m|\tilde{W}_l|\#\Hom(\mathcal{J}_l^n\mathcal{Q};\mathcal{B})
&\leq \frac{2^{2n}-1}{4}q^n\vert Gl_n(q)\vert\\
&\leq 4^{n-1}q^n\vert Gl_n(q)\vert
\end{align*}
operations. 

 To compute (\ref{sum}) we must multiply by $\pi$. Let $F_\pi\in \mathbb{C}[\mathcal{B}_{n}]$. To compute 
$\sum_\pi \pi F_\pi,$ note that $\pi$ is a permutation matrix, and so every row and column contains exactly one nonzero entry, and that entry is $1$. Then a single multiplication $\pi F_{\pi}$ requires no multiplications, and so $\sum_\pi \pi F_\pi$ does not add to the complexity. Then 
$$t_{Gl_n(q)}(R)\leq t_{Gl_{n-1}(q)}(R_{Gl_{n-1}(q)})+4^{n-1}q^n,$$
and so $$t_{Gl_n(q)}(R)\leq \frac{4^nq^{n+1}-q}{4q-1}.$$

Now suppose $p=2$. By Theorem \ref{p=2} in Appendix \ref{appD}, 
$$Y=\{\pi s_i\vert\; 1\leq i\leq n, (i-1)\text{ divisible by } p,\pi\in P\}$$
contains a complete set of coset representatives for $Gl_n(q)/Gl_{n-1}(q)$, where $s_i$ is of form
$$a_3b_2c_3a_5b_4c_5\cdots a_ib_{i-1}c_iv_{i+1}\cdots v_n,$$
for $a_j,b_j,c_j\in Gl_{j}(q)\cap\cent(Gl_{j-2}(q))$ with $(q-1)$ possible matrices for $a_j$ and $b_j$, $q^2$ possible matrices for $v_j$, and $c_j$ completely determined by $a_j$ and $b_{j-1}$. The same arguments as in the $p\neq 2$ case then yield the quiver $\mathcal{Q}$ of Figure \ref{p=2G}, from which it is clear that analogous arguments give the result.
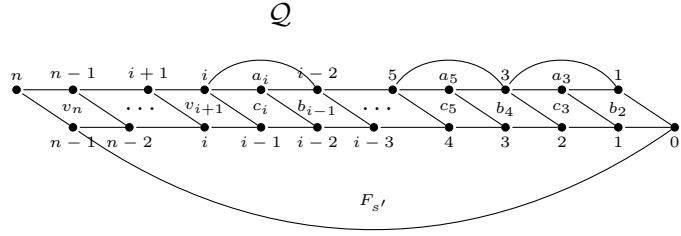
\begin{figure}[H]\begin{center}
\begin{tikzpicture}[shorten >=1pt,node distance=2cm,on grid,auto,/tikz/initial text=] 
   
   \node at (0,1.5) (q) {\large $\mathcal{Q}$};
   \node at (4.5,.25) (q) {\tiny $b_2$};
   \node at (3.75,.25) (q) {\tiny $c_3$};
   \node at (3.75,.65) (q) {\tiny $a_3$};
   \node at (3,.25) (q) {\tiny $b_4$};
   \node at (2.25,.25) (q) {\tiny $c_5$};
   \node at (2.25,.65) (q) {\tiny $a_5$};
   \node at (.5,.25) (q) {\tiny $b_{i-1}$};
   \node at (-.25,.25) (q) {\tiny $c_i$};
   \node at (-.25,.65) (q) {\tiny $a_i$};
   \node at (-1,.25) (q) {\tiny $v_{i+1}$};
   \node at (-2.75,.25) (q) {\tiny $v_n$};
   \node at (1.25,-1) (q) {\tiny $F_{s'}$};
   \node[pnt] at (5.25,0) (01) {};
   \node[pnt] at (4.5,0) (11) {};
   \node[pnt] at (3.75,0) (21) {};
   \node[pnt] at (4.5,.5) (12) {};    
   \node[pnt] at (3.75,.5) (22) {};
   \node[pnt] at (3,.5) (32) {};
   \node[pnt] at (2.25,0) (41) {};
   \node[pnt] at (2.25,.5) (42) {};
   \node[pnt] at (3,0) (31) {};
   \node[pnt] at (1.5,.5) (52) {};    
   \node at (1.325,.25) (q) {$\dots$};
   \node[pnt] at (1.25,0) (i-31) {};
   \node[pnt] at (.5,0) (i-21) {};
   \node[pnt] at (-.25,0) (i-11) {};
   \node[pnt] at (-1,0) (i1) {};
   \node[pnt] at (.5,.5) (i-22) {};   
   \node[pnt] at (-.25,.5) (i-12) {};
   \node[pnt] at (-1,.5) (i2) {};
   \node[pnt] at (-1.75,.5) (i+12) {};
   \node at (-1.825,.25) (q) {$\dots$};
   \node[pnt] at (-2,0) (n-21) {};
   \node[pnt] at (-2.75,0) (n-11) {};
   \node[pnt] at (-3.5,.5) (n2) {};
   \node[pnt] at (-2.75,.5) (n-12) {};
   \draw (01) node[below] {\tiny $0$};
   \draw (11) node[below] {\tiny $1$};
   \draw (12) node[above] {\tiny $1$};
   \draw (21) node[below] {\tiny $2$};
   \draw (32) node[above] {\tiny $3$};
   \draw (31) node[below] {\tiny $3$};
   \draw (41) node[below] {\tiny $4$};
   \draw (52) node[above] {\tiny $5$};
   \draw (i-31) node[below] {\tiny $i-3$};
   \draw (i-21) node[below] {\tiny $i-2$};
   \draw (i-11) node[below] {\tiny $i-1$};
   \draw (i1) node[below] {\tiny $i$};
   \draw (i-22) node[above] {\tiny $i-2$};
   \draw (i2) node[above] {\tiny $i$};
   \draw (i+12) node[above] {\tiny $i+1$};
   \draw (n-12) node[above] {\tiny $n-1$};
   \draw (n-21) node[below] {\tiny $n-2$};
   \draw (n-11) node[below] {\tiny $n-1$};
   \draw (n2) node[above] {\tiny $n$};
   \path[every node/.style={font=\scriptsize}]
    (01) edge node {} (12)
    (01) edge node {} (11)
    (11) edge node {} (21)
    (11) edge node {} (22)
    (21) edge node {} (32)
    (22) edge node {} (32)
    (12) edge node {} (22)
    (12) edge [bend right=55] node {} (32)
    (32) edge node {} (42)
    (21) edge node {} (31)    
    (31) edge node {} (41)
    (42) edge node {} (52)
    (31) edge node {} (42)
    (41) edge node {} (52)
    (32) edge [bend right=55] node {} (52)
    (41) edge node {} (i-31)    
    (52) edge node {} (i-22)    
    (i-31) edge node {} (i-21)
    (i-31) edge node {} (i-22)
    (i-21) edge node {} (i-11)
    (i-21) edge node {} (i-12)
    (i-11) edge node {} (i1)
    (i-11) edge node {} (i2)    
    (i-22) edge node {} (i-12)
    (i-22) edge [bend right=55] node {} (i2)
    (i-12) edge node {} (i2)
    (i1) edge node {} (i+12)
    (i2) edge node {} (i+12)
    (n-21) edge node {} (n-12)
    (n-21) edge node {} (n-11)
    (n-12) edge node {} (n2)    
    (i1) edge node {} (n-21)
    (i+12) edge node {} (n-12)    
    (n-11) edge node {} (n2)
    (01) edge [bend left=35] node {} (n-11);    
\end{tikzpicture}
\caption{The full quiver factorization for $p=2$.}
\label{p=2G}
\end{center}\end{figure}

\end{proof}

\subsection{Generalized Symmetric Group Case}\label{sym} We next give a general result (Theorem \ref{Gn}) to find efficient Fourier transforms on groups with special subgroup structure. As the proof follows the same structure of the proofs of Theorems \ref{Bn}, \ref{Dn}, and \ref{Gln}, we leave it as an exercise.

Suppose $G_n> G_{n-1}>\cdots> G_0=e$
is a chain of subgroups with subsets $A_i\subseteq G_i$ such that
\begin{itemize}
\item[(1)] $A_1=G_1$,
\item[(2)] $ G_i=A_2\cdots A_i G_{i-1}$ for $2\leq i\leq n$,
\item[(3)] $A_i$ commutes with $G_{i-2}.$
\end{itemize}

Let $\mathcal{B}$ be the Bratteli diagram associated to the chain 
$$\mathbb{C}[G_{n}]> \mathbb{C}[G_{n-1}]>\cdots> \mathbb{C},$$ and let $\{\mathbb{C}[\mathcal{B}_{i}]\}$ be the chain of path algebras associated to this group algebra chain. Let $M_\mathcal{B}(G_i,G_j):=\max M_\mathcal{B}(\alpha_i,\alpha_j)$ over all $\alpha_i\in\mathcal{B}^i, \alpha_j\in\mathcal{B}^j$ and let $\vert \hat{G}_i\vert$ denote the number of conjugacy classes of $G_i$; equivalently, the number of irreducible representations in a complete set of inequivalent irreducible representations of $G_i$.
\begin{theorem}\label{Gn}
Let $G_i$, $A_i$ be as described above. Then the Fourier transform of a complex function on $G_n$ may be computed at a complete set $R$ of irreducible representations of $G_n$ adapted to the chain
$G_n> G_{n-1}>\cdots> G_0=e$
in at most $$\vert G_n\vert \sum_{k=1}^n \sum_{i=2}^k M_\mathcal{B}(G_{i-1},G_{i-2})^2\vert \hat{G}_{i-2}\vert \frac{\vert G_i\vert}{\vert G_{i-1}\vert}\frac{\vert G_{k-1}\vert}{\vert G_k\vert} \prod_{j=i}^k \vert A_j\vert$$ operations.
\end{theorem}
\begin{note}
Theorem \ref{Gn} is a refinement of Theorem 3.1 of \cite{MR-adapted}: rather than considering the maximum length of a factorization in terms of coset representatives, we need only multiply by $\prod \vert A_j\vert$. Note that our choice of coset representatives in the proofs of Theorems \ref{Bnthm} and \ref{Dnthm} were such that  $\prod \vert A_j\vert=1$, much smaller than the length of the longest factorization in terms of coset representatives. 
\end{note}

\begin{note} For $G_i=S_i$, this theorem gives an efficient algorithm for the computation of the Fourier transform of a function on the symmetric group by letting $A_1=\{e\}$ and $A_i=\{e, t_{i-1}\}$ for $2\leq i\leq n$. 
\end{note}

\subsection{The Complexity of Fourier Transforms on Homogeneous Spaces}

We next consider the Fourier transform of a function on a homogeneous space, a special case of harmonic analysis on groups. This can be viewed as a coset space $G/K$, so a Fourier transform on a homogeneous space is a Fourier transform of the space of functions on $G/K$ or, equivalently, of the space of associated right-$K$ invariant functions on $G$. See \cite{maslen, sovi} for further background on Fourier transforms on homogeneous spaces and some of their applications.

\begin{definition} Let $G$ be a finite group with subgroup $K$ and let $f$ be a complex-valued function 
on $G/K$. The \textbf{Fourier transform of f} at a $K$-adapted representation $\rho$ of $G$, denoted $\hat{f}(\rho)^K$, or a $K$-adapted set $R$ of matrix representations of $G$, denoted $\mathcal{F}_R^K f$, is the Fourier transform of the right $K$-invariant function $\tilde{f}:G\rightarrow \mathbb{C}$ defined by
$$\tilde{f}(g)=\frac{1}{\vert K\vert}f(gK).$$
\end{definition}
Note that $\hat{f}(\rho)$ is zero unless the representation space, $V_\rho$, contains a nontrivial $K$-invariant vector. Such a representation is said to be \textbf{class 1 relative to K}, and we could restrict to class 1 representations if desired. 

\begin{definition} Let $G$ be a finite group with subgroup $K$ and let $R$ be a set of representations of $G$. 
\begin{itemize}
\item[(i)] The \textbf{arithmetic complexity} of a Fourier transform on $R$, denoted $T_{G/K}(R)$, is the minimum number of arithmetic multiplications (or additions, whichever is largest) needed to compute the Fourier transform of $f$ on $R$ via a straight-line program for an arbitrary complex-valued function $f$ defined on $G/K$.
\item[(ii)] The \textbf{reduced complexity}, denoted $t_{G/K}(R)$, is defined by $$t_G(R)=\frac{1}{\vert G/K\vert}T_{G/K}(R).$$
\end{itemize}
\end{definition}
Note that the complexity always satisfies the inequalities
$$\vert G/K\vert-1\leq T_{G/K}(R)\leq \vert G/K\vert^2.$$

Further, the proof of Lemma \ref{factorsum} gives an analogous result for the case of homogenous spaces: for $H$ a subgroup of $G$, $R$ a complete $H$-adapted set of inequivalent irreducible representations of $G$, and $Y\subseteq G$ a set of coset representatives,
$$t_{G/K}(R)\leq t_{H/K}(R_H)+m_{G/K}(R,\tilde{Y},H).$$

Let $G$ be a group with chain of subgroups $G=G_n>G_{n-1}>\cdots>G_0$. For $f$ a function on $G_{n}$ that is right $G_{n-k}$-invariant, the corresponding element $\sum_{s\in G_n}f(s)s$ in $\mathbb{C}[G_n]$ is invariant under right multiplication by elements of $\mathbb{C}[G_{n-k}]$. In particular, the elements $F_y$ in the proofs of Section \ref{applications} are $\mathbb{C}[\mathcal{B}_{n-k}]$-invariant, so as in \cite[Theorem 6.2]{sovi} the nonzero coefficients of $F_y$ correspond to paths passing through $1_{n-k}$.       
Using the SOV approach as in the proofs of Section \ref{applications}, the final quiver used when constructing $\mathcal{Q}$,  i.e., the quiver  corresponding to $F_y$, now has form as in Figure \ref{Qcn}, with $\#\Hom(Q_{i};\mathcal{B})$ counting only occurrences of $Q_i$ in $\mathcal{B}$ with $\alpha_{n-k}$ the vertex $1_{n-k}$. 
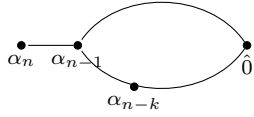
\begin{figure}[H]\begin{center}
\begin{tikzpicture}[shorten >=1pt,node distance=2cm,on grid,auto,/tikz/initial text=] 
   \node[pnt] at (3,-5.5) (011) {};
   \node[pnt] at (0,-5.5) (n11) {};
   \node[pnt] at (.75,-5.5) (n-11) {};
   \node[pnt] at (1.5,-6.05) (n-k1) {};
   \draw (011) node[below] {\footnotesize $\hat{0}$};
   \draw (n11) node[below] {\footnotesize $\alpha_n$};
   \draw (n-11) node[below] {\footnotesize $\alpha_{n-1}$};
   \draw (n-k1) node[below] {\footnotesize $\alpha_{n-k}$};
   \path[every node/.style={font=\scriptsize}]
    (n-11) edge node {} (n11)
    (011) edge [bend left=55] node {} (n-11)
    (011) edge [bend right=55] node {} (n-11);    
    
\end{tikzpicture}
\caption{The component subquiver associated to $F_y$ in the factorization of $yF_y$.}
\label{Qcn}
\end{center}\end{figure}
Then the proofs of Section \ref{applications} extend to the following results for homogenous spaces:

\begin{theorem}[cf. Theorem \ref{hombnthm}]
For the homogenous space $B_n/B_{n-k}$ of the Weyl group $B_n$ and $R$ a complete set of irreducible matrix representations of $B_n$ adapted to the subgroup chain $B_n> B_{n-1}>\cdots> \{e\},$

$$C(B_n/B_{n-k})\leq T_{B_n/B_{n-k}}(R)\leq k(4n-2k-1)\frac{\vert B_n\vert}{\vert B_{n-k}\vert}.$$ 
\end{theorem}
\begin{theorem}
For the homogenous space $D_n/D_{n-k}$ of the Weyl group $D_n$ and $R$ a complete set of irreducible matrix representations of $D_n$ adapted to the subgroup chain $D_n> D_{n-1}>\cdots> \{e\},$

$$C(D_n/D_{n-k})\leq T_{D_n/D_{n-k}}(R)\leq \frac{k(26n-13k-11)}{2}\frac{\vert D_n\vert}{\vert D_{n-k}\vert}.$$ 
\end{theorem}
\begin{theorem}
For the homogenous space $Gl_n(q)/Gl_{n-k}(q)$ of the general linear group $Gl_n(q)$ and $R$ a complete set of irreducible matrix representations of $Gl_n(q)$ adapted to the subgroup chain $Gl_n(q)> Gl_{n-1}(q)>\cdots> \{e\},$

$$\displaystyle C(Gl_n(q)/Gl_{n-k}(q))\leq T_{Gl_n(q)/Gl_{n-k}(q)}(R)$$
$$\displaystyle\leq \left(\frac{4^nq^{n+1}-4^{n-k}q^{n-k+1}}{4q-1}\right)\frac{\vert G_n\vert}{\vert G_{n-k}\vert}.$$

\end{theorem}

As in Section \ref{sym}, suppose 
$G_n> G_{n-1}>\cdots> G_1=e$
is a chain of groups with subsets $A_i\subseteq G_i$ such that

\begin{itemize}
\item[(1)] $A_1=G_1$
\item[(2)] $ G_i=A_2\cdots A_i G_{i-1}$ for $2\leq i\leq n$.
\item[(3)] $A_i$ commutes with $G_{i-2}.$
\end{itemize}
\begin{theorem}\label{GnHom} Let $G_i$, $A_i$ be as above. For the homogeneous space $G_n/G_{n-k}$ and $R$ a complete set of irreducible matrix representations of $G_n$ adapted to the chain
$G_n> G_{n-1}>\cdots> G_1=e$,
$$C(G_n/G_{n-k})\leq T_{G_n/G_{n-k}}(R)\leq \sum_{_{j=n-k+1}}^n \sum_{_{i=2}}^j M_\mathcal{B}(G_{i-1},G_{i-2})^2\vert \hat{G}_{i-2}\vert \frac{\vert G_i\vert}{\vert G_{i-1}\vert}\frac{\vert G_{j-1}\vert}{\vert G_j\vert} \prod_{l=i}^j \vert A_l\vert.$$ 
\end{theorem}

\section{Configuration Spaces and the Maps $*$}\label{proofsofstuff}

In Section \ref{sepvarstate} we gave an overview of the SOV algorithm, assuming the existence of bilinear maps $*$ with the properties described in part II of the SOV approach \ref{alg2}. In this section we determine such maps and investigate their properties.

Recall from Definition \ref{widef} that for a path algebra product $x_1\cdots x_m$,  $X_i:=\mathbb{C}[\mathcal{B}_{i^+}]\cap\cent(\mathbb{C}[\mathcal{B}_{i^-}]).$ We first show (Lemma \ref{configspaces} below), that each space $X_i$ is isomorphic to the \textit{configuration space} of a specific quiver $Q_i$, with dimension $\#\Hom(Q_i;\mathcal{B})$.

\begin{definition}\label{defmorph} For graded quivers $Q$ and $B$, a \textbf{morphism} $\phi:Q\rightarrow B$ is a mapping from arrows in $Q$ to paths in $B$, along with a grading-preserving mapping between vertices so that $\phi(t(e))=t(\phi(e))$ and $\phi(s(e))=s(\phi(e))$ for all arrows $e\in E(Q)$.
\end{definition}
\begin{example} For $Q$, $B$ as in Figure \ref{morphex}, let $\phi:Q\rightarrow B$ send the arrow $e_1$ to the path $f_3\circ f_2\circ f_1$. 
\begin{figure}[H]\begin{center}
\begin{tikzpicture}[shorten >=1pt,node distance=2cm,on grid,auto,/tikz/initial text=] 
   \node at (-2.25, .75)(q){Q};
   \node at (.25, .75)(q){B};
   \node[pnt] at (-3, 0)(30){};
   \node[pnt] at (-1.5, 0)(00){};
   \node[pnt] at (-.5,0)(31){};
   \node[pnt] at (0, 0)(21){};
   \node[pnt] at (.5, 0)(11){};
   \node[pnt] at (1, 0)(01){};
   \draw (30) node[above] {\footnotesize $3$};
   \draw (00) node[above] {\footnotesize $0$};
   \draw (31) node[above] {\footnotesize $3$};
   \draw (21) node[above] {\footnotesize $2$};
   \draw (11) node[above] {\footnotesize $1$};
   \draw (01) node[above] {\footnotesize $0$};
    \path[->,every node/.style={font=\scriptsize}]
    (00) edge  node {$e_1$} (30)
    (01) edge  node {$f_1$} (11)
    (11) edge  node {$f_2$} (21)
    (21) edge  node {$f_3$} (31);
    
\end{tikzpicture}
\caption{An example morphism would send $e_1$ to $f_3\circ f_2\circ f_1$.}
\label{morphex}
\end{center}\end{figure}
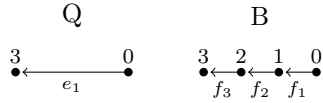

\end{example}

For two graded quivers $Q$ and $B$, let $\Hom(Q;B)$ denote the set of morphisms from $Q$ to $B$. For $Q, R$, and $B$ graded quivers such that $Q$ is a subquiver of $R$, let $\Hom(Q\uparrow R;B)$ denote the set of morphisms from $Q$ to $B$ that extend to $R$. 
\begin{definition} The \textbf{configuration space} associated to $Q$ and $R$ relative to $B$, denoted $A(Q\uparrow R;B)$, is the space of finitely supported formal $\mathbb{C}$-linear combinations of morphisms in $\Hom(Q\uparrow R;B)$.
\end{definition}
\begin{note}\label{finitedim} When $Q=R$, we simplify notation by writing $A(Q;B)$. If $Q$ is a finite subquiver of $R$ and $B$ is locally finite, i.e. each vertex has finitely many neighbors, then $\#\Hom(Q\uparrow R; B)=\dim A(Q\uparrow R; B)$.
\end{note}
\begin{lemma}\label{configspaces} Let  $\{\mathbb{C}[G_i]\}$ be a chain of group algebras with corresponding Bratteli diagram $\mathcal{B}$ of highest grading at least $n$. Consider the quivers $Q_{n0}$ and $Q_{ji}^n$ of Figure \ref{figgif}, along with the subquiver $Q_{ji}$ of $Q_{ji}^n$ consisting of the two vertices at level $i$ and level $j$, along with the two paths from level $i$ to level $j$:

\begin{figure}[H]\begin{center}
\begin{tikzpicture}[shorten >=1pt,node distance=2cm,on grid,auto,/tikz/initial text=] 
   \node[pnt] at (-.5, 0)(q_3){};
   \draw (q_3) node[below] {\footnotesize $n$};
   \node[pnt] at (2.5, 0)(q_4){};
   \draw (q_4) node[below] {\footnotesize $0$};
   \draw (1,1) node {\large $Q_{n0}$};
  \begin{scope}[shift={(1,0)}]
  \node[pnt] at (3, 0)(q_5){};
   \draw (q_5) node[below] {\footnotesize $n$};
   \node[pnt] at (7, 0)(q_6){};
   \draw (q_6) node[below] {\footnotesize $0$};
   \node[pnt] at (4, 0)(q_7){};
   \draw (q_7) node[below] {\footnotesize $j$};
   \node[pnt] at (6, 0)(q_8){};
   \draw (q_8) node[below] {\footnotesize $i$};
   \draw (5,1) node {\large $Q_{ji}\hookrightarrow Q_{ji}^n$};
   \end{scope}
   \path[->, every node/.style={font=\scriptsize}]
    (q_4) edge [bend left=45] node {} (q_3)
    (q_4) edge [bend right=45] node {} (q_3)
    (q_8) edge [bend right=65] node {} (q_7)
    (q_8) edge [bend left=65] node {} (q_7);
    \path[->, every node/.style={font=\scriptsize}]
    (q_7) edge node {} (q_5)
    (q_6) edge node {} (q_8);
    
\end{tikzpicture}
\caption{}
\label{figgif}
\end{center}\end{figure}
Then as vector spaces,
\begin{itemize}
\item[(i)]  $A(Q_{n0}; \mathcal{B})\cong\mathbb{C}[\mathcal{B}_n]$,
\item[(ii)]$A(Q_{ji}; \mathcal{B})=A(Q_{ji}\uparrow Q_{ji}^n; \mathcal{B})\cong \mathbb{C}[\mathcal{B}_j]\cap \cent(\mathbb{C}[\mathcal{B}_i]).$
\end{itemize}
\end{lemma}
\begin{proof} This follows from an application of standard facts about Gel'fand-Tsetlin bases. For futher details, see eg. \cite[Lemma 4.1]{maslen}, \cite[Proposition 2.3.12]{towers}.
\end{proof}


By Lemma \ref{configspaces}, for $X_i=\mathbb{C}[\mathcal{B}_{i^+}]\cap\cent(\mathbb{C}[\mathcal{B}_{i^-}])$ as in Definition \ref{widef} and for $Q_i:=Q_{i^+i^-}$,  $X_i\cong A(Q_{i}; \mathcal{B}).$
\subsection{The Bilinear Maps `$*$'}\label{maps}
\begin{definition}\label{symdif} For a graded quiver $R$ with subquivers $Q_1, Q_2$, the \textbf{symmetric difference} of $Q_1$ and $Q_2$ is $Q_1\triangle Q_2=(Q_1\setminus (Q_1\cap Q_2))\cup(Q_2\setminus (Q_1\cap Q_2)).$ 
\end{definition} 
\begin{example} The quiver $Q_2\triangle Q_3$ in Figure \ref{exQ} is the symmetric difference of $Q_2$ and $Q_3$, while the quivers $Q_1^\sigma\triangle Q_2^\sigma$ in Figures \ref{BnR} and \ref{GlnFGH} show the symmetric difference of $Q_1^\sigma$ and $Q_2^\sigma$. 
\end{example}
\begin{definition}\label{resprod} Let $B$ be a locally finite graded quiver, $R$  a graded quiver with finite subquivers $Q_1$ and $Q_2$, and $\iota_j$ the inclusion $Q_j\hookrightarrow R$, for $j=1,2$. For   $(f,g)\in A(Q_1\uparrow R;B)\times A(Q_2\uparrow R;B)$, define the \textbf{restricted product relative to R},
$*:A(Q_1\uparrow R;B)\times A(Q_2\uparrow R;B)\rightarrow A(Q_1\triangle Q_2\uparrow R; B)$, by
$$*(f,g)=f*g:=\sum_{\tau\in\Hom(Q_1\triangle Q_2\uparrow R; B)}\sum_{\substack{\eta\in\Hom(Q_1\cup Q_2\uparrow R;B)\\\eta\downarrow_{Q_1\triangle Q_2}=\tau}}f\vert_{\eta\circ\iota_1}g\vert_{\eta\circ\iota_2} \tau.$$ 
\end{definition}
\begin{note}\label{commrest}
It is clear from the definition that the restricted product is bilinear and commutative. In Appendix \ref{appA} we show that the restricted product is associative. 
\end{note}
\begin{lemma}\label{rprodcount} For $B$ a locally finite graded quiver, $R$ a graded quiver with finite subquivers $Q_1$ and $Q_2$, $f\in A(Q_1\uparrow R;B)$, and $g\in A(Q_2\uparrow R;B)$, the restricted product $f*g$ requires at most $\#\Hom((Q_1\cup Q_2)\uparrow R; B)$ scalar multiplications and at most $\#\Hom((Q_1\cup Q_2)\uparrow R; B)-\#\Hom((Q_1\triangle Q_2)\uparrow R; B)$ scalar additions.
\end{lemma}
\begin{proof} To compute $f*g$, first compute $(f\vert_{\eta\circ{\iota_1}})(g\vert_{\eta\circ{\iota_2}})$ for each $\eta\in\Hom(Q_1\cup Q_2\uparrow R; B)$. This requires $\#\Hom(Q_1\cup Q_2\uparrow R; B)$ scalar multiplications.

Next note that a scalar addition comes from each pair $\eta_i,\eta_j\in\Hom(Q_1\cup Q_2\uparrow R; B)$ with $\eta_i\downarrow_{Q_1\triangle Q_2}=\eta_j\downarrow_{Q_1\triangle Q_2}=\tau\in\Hom(Q_1\triangle Q_2\uparrow R; B);$ 
in total, $\#\Hom((Q_1\cup Q_2)\uparrow R; B)-\#\Hom((Q_1\triangle Q_2)\uparrow R; B)$ scalar additions.
\end{proof}

Lemma \ref{configspaces} gives a correspondence between $\mathbb{C}[\mathcal{B}_n]$ and the configuration space of the associated quiver $Q_{n0}$. With Theorem \ref{multthm} below, we see that under this isomorphism multiplication of path algebra elements corresponds to restricted products in the associated configuration spaces.
\begin{theorem}\label{multthm} Let $\mathcal{B}$ be a Bratteli diagram of highest grading at least $n$ and let $f,g\in\mathbb{C}[\mathcal{B}_n]$. Let $Q_1$ and $Q_2$ be the quivers of Figure \ref{bigthm} with paths $p$, $q$, and $p'$, $q'$, respectively. Let $q=p'$ and let $R=Q_1\cup Q_2$. 
\begin{figure}[H]\begin{center}
\begin{tikzpicture}[shorten >=1pt,node distance=2cm,on grid,auto,/tikz/initial text=] 
   \node at (-3,1) (q) {\large $Q_1$};
   \node at (1,1) (q) {\large$R$};     
   \node at (5,1) (q) {\large$Q_{1}\triangle Q_{2}$};    \node at (-3,-3.5) (q) {\large $Q_2$};
   \node[pnt] at (-2,0) (q_0) {};
   \node[pnt] (q_2) [left of=q_0] {};
   \node[pnt] (q_1) [below of=q_0] {};
   \node[pnt] (q_3) [below of=q_2] {};
   \node[pnt] at (0,-1) (y4) {};
   \node[pnt] (y5) [right of=y4] {};
   \node[pnt] at (4,-1) (y8) {};
   \node[pnt] (y9) [right of=y8] {};
   \draw (q_0) node[below] {\footnotesize $0$};
   \draw (q_2) node[below] {\footnotesize $n$};
   \draw (q_1) node[below] {\footnotesize $0$};
   \draw (q_3) node[below] {\footnotesize $n$};
   \draw (y4) node[below] {\footnotesize $n$};
   \draw (y5) node[below] {\footnotesize $0$};
   \draw (y9) node[below] {\footnotesize $0$};   
   \draw (y8) node[below] {\footnotesize $n$};
   \path[->,every node/.style={font=\scriptsize}]
    (q_0) edge [bend left=55] node {$q$} (q_2)
    (q_0) edge [bend right=55] node {$p$} (q_2)
    (q_1) edge [bend left=55] node {$q'$} (q_3)
	(q_1) edge [bend right=55] node {$p'$} (q_3)
    (y5) edge [bend left=55] node {$q'$} (y4)
    (y5) edge  node {$q=p'$} (y4)
    (y5) edge [bend right=55] node {$p$} (y4)
    (y9) edge [bend left=55] node {$q'$} (y8)
    (y9) edge  [bend right=55] node {$p$} (y8);    
\end{tikzpicture}
\caption{}
\label{bigthm}
\end{center}\end{figure}

Then under the isomorphisms $\phi_i:\mathbb{C}[\mathcal{B}_n]\rightarrow A(Q_i\uparrow R; \mathcal{B})$ and  $\phi: \mathbb{C}[\mathcal{B}_n]\rightarrow A(Q_1\triangle Q_2\uparrow R;\mathcal{B})$ of Lemma \ref{configspaces}, $$\phi(fg)=\phi_1(f)*\phi_2(g).$$
\end{theorem}
\begin{proof} For $P,Q$ paths of length $n$ in $\mathcal{B}$, let $\gamma_{PQ}\in\Hom(Q_1;\mathcal{B})$ denote the morphism that sends $p$ to $P$ and $q$ to $Q$. Similarly, let $\mu_{PQ}\in\Hom(Q_2;\mathcal{B})$ (respectively, $\tau_{PQ}\in\Hom(Q1\triangle Q_2;\mathcal{B})$) denote the morphism that sends $p'$ to $P$ and $q'$ to $Q$ (respectively, $p$ to $P$ and $q'$ to $Q$).

Let $f=\sum f\vert_{PQ}(P,Q)\in\mathbb{C}[\mathcal{B}_n]$, $g=\sum g\vert_{PQ}(P,Q)\in\mathbb{C}[\mathcal{B}_n]$. Then $$\begin{array}{ll}\displaystyle\phi_1(f)=\sum_{\gamma_{PQ}\in\Hom(Q_1;\mathcal{B})} f\vert_{PQ}\gamma_{PQ}, &\displaystyle\phi_2(g)=\sum_{\mu_{PQ}\in\Hom(Q_2;\mathcal{B})} g\vert_{PQ}\gamma_{PQ},\end{array}$$
and 
$$\phi_1(f)*\phi_2(g)=\sum_{\tau_{PQ'}\in\Hom(Q_1\triangle Q_2\uparrow R;\mathcal{B})}\sum_{\substack{\eta\in\Hom(R;\mathcal{B})\\\eta\downarrow_{Q_1\triangle Q_2}=\tau_{PQ'}}} f\vert_{\eta\circ\iota_1}g\vert_{\eta\circ\iota_2}\tau_{PQ'}.$$

A morphism $\eta\in\Hom(R;\mathcal{B})$ with $\eta\downarrow_{Q_1\triangle Q_2}=\tau_{PQ'}$ must send $p$ to $P$, $q'$ to $Q'$, and $q, p'$ to the same path, $Q$. Then $\eta\circ\iota_1=\gamma_{PQ}$, $\eta\circ\iota_2=\mu_{QQ'}$, and 

$$ \phi_1(f)*\phi_2(g)=\sum_{P,Q'}\left(\sum_{Q}f\vert_{PQ}f\vert_{QQ'}\right)\tau_{PQ'}=\phi(fg).$$
\end{proof}

\begin{corollary}\label{e} For $e$ the identity element of $\mathbb{C}[\mathcal{B}_{n}]$, a restricted product with $\phi(e)$ requires no arithmetic operations to compute.
\end{corollary}
\subsection{Use of `$*$' in the SOV Algorithm}\label{finalalg} In this section we combine the results of Section \ref{maps} with Section \ref{sepvarstate} to show how the restricted product is used in the SOV algorithm.

Note that by Lemma \ref{configspaces}, for $f\in X_i=\mathbb{C}[\mathcal{B}_{i^+}]\cap\cent(\mathbb{C}[\mathcal{B}_{i^-}])$, $f\vert_{PQ}=0$ unless $P$ and $Q$ are two paths that agree from level $0$ to level $i^-$ and from level $i^+$ to level $n$. The same is true for $g\in X_j$. Then $Q_1$ and $Q_2$ have form as in Figure \ref{fig:2}, and $R$ has one of three possible forms, depending on the relation between $i^+$, $i^-$, $j^+,$ and $j^-$ (see Figures \ref{fig:2}, \ref{fig:3}, \ref{fig:4}).

\begin{figure}[H]\begin{center}
\begin{tikzpicture}[shorten >=1pt,node distance=2cm,on grid,auto,/tikz/initial text=] 
   \node at (-3,1) (q) {\large $Q_1$};
   \node at (-3,-3) (q) {\large $Q_2$};
   \node at (1.5,1) (q) {\large $R$};
   \node at (5.5,1) (q) {\large $Q_{1}\triangle Q_{2}$};   
   \node[pnt] at (-2,0) (q_0) {};
   \node[pnt] (q_2) [left of=q_0] {};
   \node[pnt] (q_1) [below of=q_0] {};
   \node[pnt] (q_3) [below of=q_2] {};
   \node[pnt] at (-2.5,0) (i1) {};
   \node[pnt] at (-3.5,0) (j1) {};
   \node[pnt] at (-2.5,-2) (k1) {};
   \node[pnt] at (-3.5,-2) (l1) {};
   \node at (-1,-.7)(y3){};
   \node at (-1,-1)(y3){};   
   \node[pnt] at (0,-1) (y4) {};
   \node[pnt] at (3,-1) (y5) {};
   \node[pnt] at (2.5,-1) (i2) {};
   \node[pnt] at (.5,-1) (j2) {};  
   \node[pnt] at (2.25,-1.3) (k2) {};
   \node[pnt] at (.75,-1.3) (l2) {};
   \node at (3.5,-.7)(y3){};   
   \node at (3.5,-1)(y3){};
   \node[pnt] at (6.5,-1) (i3) {};
   \node[pnt] at (4.5,-1) (j3) {};  
   \node[pnt] at (6.25,-1.3) (k3) {};
   \node[pnt] at (4.75,-1.3) (l3) {};  
   \draw (q_0) node[below] {\footnotesize $0$};
   \draw (q_1) node[below] {\footnotesize $0$};
   \draw (q_2) node[below] {\footnotesize $n$};
   \draw (q_3) node[below] {\footnotesize $n$};
   \draw (y5) node[below] {\footnotesize $0$};
   \draw (y4) node[below] {\footnotesize $n$};
   \draw (i1) node[below] {\footnotesize $i^-$};
   \draw (i2) node[below] {\footnotesize $i^-$};
   \draw (i3) node[below] {\footnotesize $i^-$};
   \draw (j1) node[below] {\footnotesize $i^+$};
   \draw (j2) node[below] {\footnotesize $i^+$};
   \draw (j3) node[below] {\footnotesize $i^+$};
   \draw (l1) node[below] {\footnotesize $j^+$};
   \draw (l2) node[below] {\footnotesize $j^+$};
   \draw (l3) node[below] {\footnotesize $j^+$};
   \draw (k1) node[below] {\footnotesize $j^-$};
   \draw (k2) node[below] {\footnotesize $j^-$};
   \draw (k3) node[below] {\footnotesize $j^-$};
   \path[every node/.style={font=\scriptsize}]
    (i1) edge [bend left=65] node {} (j1)
    (i1) edge [bend right=65] node {} (j1)
    (k1) edge [bend left=65] node {} (l1)
    (k1) edge [bend right=65] node {} (l1)
    
    (k2) edge [bend left=75] node {} (l2)
    (j2) edge [bend left=55] node {} (i2)
    (j2) edge [bend right=55] node {} (i2)
    
    (k3) edge [bend left=75] node {} (l3)
    (k3) edge [bend right=35] node {} (i3)
    (j3) edge [bend left=55] node {} (i3)
    (j3) edge [bend right=35] node {} (l3);
   
   \path[->, every node/.style={font=\scriptsize}]
    (q_0) edge node {} (i1)
    (q_1) edge node {} (k1)
    (j1) edge node {} (q_2)
    (l1) edge node {} (q_3)
    (y5) edge node {} (i2)
    (j2) edge node {} (y4);

\end{tikzpicture}
\caption{$i^-\leq j^-\leq j^+\leq i^+$}
\label{fig:2}

\end{center}\end{figure}
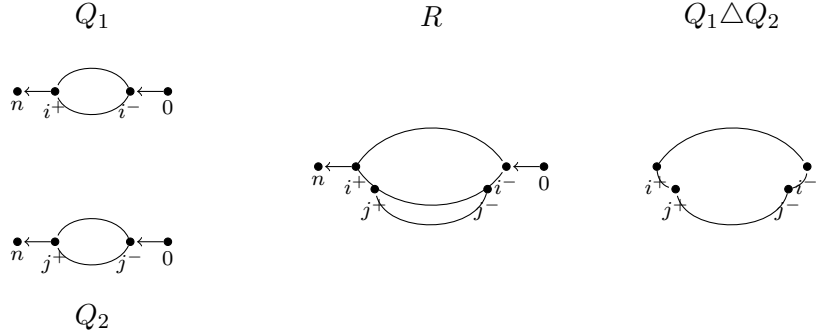

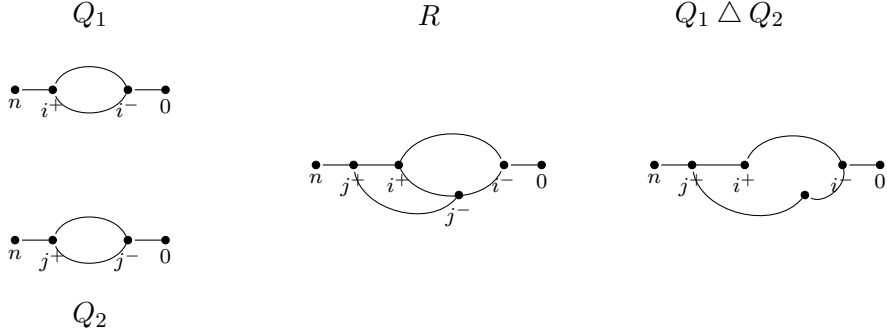
\begin{figure}[H]
\begin{tikzpicture}[shorten >=1pt,node distance=2cm,on grid,auto,/tikz/initial text=] 
 
   \node at (-3,1) (q) {\large $Q_1$};
   \node at (-3,-3) (q) {\large $Q_2$};
      \node at (1.5,1) (q) {\large $R$};
   \node at (5.5,1) (q) {\large $Q_1\bigtriangleup Q_2$};   
   \node[pnt] at (-2,0) (q_0) {};
   \node[pnt] (q_2) [left of=q_0] {};
   \node[pnt] (q_1) [below of=q_0] {};
   \node[pnt] (q_3) [below of=q_2] {};
   \node[pnt] at (-2.5,0) (i1) {};
   \node[pnt] at (-3.5,0) (j1) {};
   \node[pnt] at (-2.5,-2) (k1) {};
   \node[pnt] at (-3.5,-2) (l1) {};
   \node[pnt] at (0,-1) (y4) {};
   \node[pnt] at (3,-1) (y5) {};
   \node[pnt] at (2.5,-1) (i2) {};
   \node[pnt] at (1.1,-1) (j2) {};  
   \node[pnt] at (1.9,-1.4) (k2) {};
   \node[pnt] at (.5,-1) (l2) {};
   \node[pnt] at (4.5,-1) (y8) {};
   \node[pnt] at (7.5,-1) (y9) {};
   \node[pnt] at (7,-1) (i3) {};
   \node[pnt] at (5.7,-1) (j3) {};  
   \node[pnt] at (6.5,-1.4) (k3) {};
   \node[pnt] at (5,-1) (l3) {};
   
   \draw (q_0) node[below] {\footnotesize $0$};
   \draw (q_1) node[below] {\footnotesize $0$};
   \draw (q_2) node[below] {\footnotesize $n$};
   \draw (q_3) node[below] {\footnotesize $n$};
   \draw (y5) node[below] {\footnotesize $0$};
   \draw (y4) node[below] {\footnotesize $n$};
   \draw (y8) node[below] {\footnotesize $n$};
   \draw (y9) node[below] {\footnotesize $0$};
   \draw (i1) node[below] {\footnotesize $i^-$};
   \draw (i2) node[below] {\footnotesize $i^-$};
   \draw (i3) node[below] {\footnotesize $i^-$};
   \draw (j1) node[below] {\footnotesize $i^+$};
   \draw (j2) node[below] {\footnotesize $i^+$};
   \draw (j3) node[below] {\footnotesize $i^+$};
   \draw (l1) node[below] {\footnotesize $j^+$};
   \draw (l2) node[below] {\footnotesize $j^+$};
   \draw (l3) node[below] {\footnotesize $j^+$};
   \draw (k1) node[below] {\footnotesize $j^-$};
   \draw (k2) node[below] {\footnotesize $j^-$};
   \path[every node/.style={font=\scriptsize}]
    (q_0) edge node {} (i1)
    (j1) edge node {} (q_2)
    (i1) edge [bend left=65] node {} (j1)
    (i1) edge [bend right=65] node {} (j1)
    (q_1) edge node {} (k1)
    (l1) edge node {} (q_3)
    (k1) edge [bend left=65] node {} (l1)
    (k1) edge [bend right=65] node {} (l1)
    (y5) edge node {} (i2)
    (j2) edge node {} (y4)
    (j2) edge [bend left=65] node {} (i2)
    (j2) edge [bend right=65] node {} (i2)
    (k2) edge [bend left=65] node {} (l2)
    (i3) edge [bend right=65] node {} (j3)
    (i3) edge [bend left=65] node {} (k3)
    (k3) edge [bend left=65] node {} (l3)
    (y9) edge  node {} (i3)
    (j3) edge  node {} (y8); 
   
\end{tikzpicture}
\caption{$i^-\leq j^-\leq i^+\leq j^+$}
\label{fig:3}
\end{figure}
\begin{figure}[H]
\begin{tikzpicture}[shorten >=1pt,node distance=2cm,on grid,auto,/tikz/initial text=] 
  
   \node at (-3,1) (q) {\large $Q_1\sqcup Q_2$};
   \node at (-3,-3) (q) {\large $Q_2$};
   \node at (1.5,1) (q) {\large $R$};
   \node at (5.5,1) (q) {\large $Q_1\bigtriangleup Q_2$};   
   \node[pnt] at (-2,0) (q_0) {};
   \node[pnt] (q_2) [left of=q_0] {};
   \node[pnt] (q_1) [below of=q_0] {};
   \node[pnt] (q_3) [below of=q_2] {};
   \node[pnt] at (-2.5,0) (i1) {};
   \node[pnt] at (-3.5,0) (j1) {};
   \node[pnt] at (-2.5,-2) (k1) {};
   \node[pnt] at (-3.5,-2) (l1) {};
   \node[pnt] at (0,-1) (y4) {};
   \node[pnt] at (3,-1) (y5) {};
   \node[pnt] at (2.5,-1) (i2) {};
   \node[pnt] at (1.8,-1) (j2) {};  
   \node[pnt] at (1.3,-1) (k2) {};
   \node[pnt] at (.5,-1) (l2) {};
   \node[pnt] at (4.5,-1) (y8) {};
   \node[pnt] at (7.5,-1) (y9) {};
   \node[pnt] at (7,-1) (i3) {};
   \node[pnt] at (6.3,-1) (j3) {};  
   \node[pnt] at (5.8,-1) (k3) {};
   \node[pnt] at (5,-1) (l3) {};
   \draw (q_0) node[below] {\footnotesize $0$};
   \draw (q_1) node[below] {\footnotesize $0$};
   \draw (q_2) node[below] {\footnotesize $n$};
   \draw (q_3) node[below] {\footnotesize $n$};
   \draw (y5) node[below] {\footnotesize $0$};
   \draw (y4) node[below] {\footnotesize $n$};
   \draw (y8) node[below] {\footnotesize $n$};
   \draw (y9) node[below] {\footnotesize $0$};
   \draw (i1) node[below] {\footnotesize $i^-$};
   \draw (i2) node[below] {\footnotesize $i^-$};
   \draw (i3) node[below] {\footnotesize $i^-$};
   \draw (j1) node[below] {\footnotesize $i^+$};
   \draw (j2) node[below] {\footnotesize $i^+$};
   \draw (j3) node[below] {\footnotesize $i^+$};
   \draw (l1) node[below] {\footnotesize $j^+$};
   \draw (l2) node[below] {\footnotesize $j^+$};
   \draw (l3) node[below] {\footnotesize $j^+$};
   \draw (k1) node[below] {\footnotesize $j^-$};
   \draw (k2) node[below] {\footnotesize $j^-$};
   \draw (k3) node[below] {\footnotesize $j^-$};
   \path[every node/.style={font=\scriptsize}]
    (q_0) edge node {} (i1)
    (j1) edge node {} (q_2)
    (i1) edge [bend left=65] node {} (j1)
    (i1) edge [bend right=65] node {} (j1)
    (q_1) edge node {} (k1)
    (l1) edge node {} (q_3)
    (k1) edge [bend left=65] node {} (l1)
    (k1) edge [bend right=65] node {} (l1)
    (y5) edge node {} (i2)
    (l2) edge node {} (y4)
    (j2) edge node {} (k2)
    (j2) edge [bend left=55] node {} (i2)
    (k2) edge [bend left=55] node {} (l2)
    (k2) edge  node {} (l2)
    (i2) edge [bend left=55] node {} (j2)

    (y9) edge node {} (i3)
    (l3) edge node {} (y8)
    (j3) edge node {} (k3)
    (j3) edge [bend left=55] node {} (i3)
    (k3) edge [bend left=55] node {} (l3)
    (k3) edge  node {} (l3)
    (i3) edge [bend left=55] node {} (j3); 
   
\end{tikzpicture}
\caption{$i^-\leq i^+\leq j^-\leq j^+$}
\label{fig:4}
\end{figure}
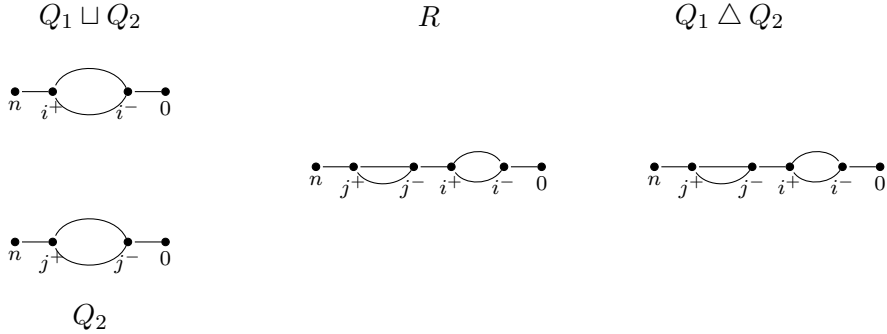

For $x_i\in X_i$, let $\tilde{x}_i:=\phi_i(x_i)\in A(Q_{i}\uparrow Q_{i}^n;\mathcal{B})$. By Theorem \ref{multthm}, for $x_1\in X_1$ and $x_2\in X_2$, the product $x_1x_2$ corresponds to the restricted product $\tilde{x}_1*\tilde{x}_2:A(Q_{1};\mathcal{B})\times A(Q_{2};\mathcal{B})\rightarrow A(Q_{1}\triangle Q_{2};\mathcal{B})$, with the paths of $Q_{1}$ and $Q_{2}$ identified as in Figures \ref{fig:2}, \ref{fig:3}, or \ref{fig:4}.

More generally, for $x_i\in X_i$ and the paths of $Q_{i}$ identified as in Figures \ref{fig:2}, \ref{fig:3}, \ref{fig:4}, the product $x_1x_2\cdots x_m$ corresponds to the restricted product $(x_1*x_2*\cdots*x_{m-1})*x_m:A(Q_1\triangle\cdots\triangle Q_{m-1};\mathcal{B})\times A(Q_{m};\mathcal{B})\rightarrow A(Q_1\triangle\cdots\triangle Q_m;\mathcal{B}).$



For $x_i\in X_i$, let $\tilde{x}_i=\phi_i(x_i)$ as in Theorem \ref{multthm}. By Note \ref{commrest}, commutativity of the restricted product ensures that $\tilde{x}_1*\cdots*\tilde{x}_m=\tilde{x}_{\sigma(1)}*\cdots*\tilde{x}_{\sigma(m)}$ for any $\sigma\in S_m$. As in Section \ref{sepvarstate} let $Q_i^\sigma$ denote the quiver associated to $X_{\sigma(i)}$.
\begin{theorem}\label{cond45}
For $x_i\in X_i=\mathbb{C}[\mathcal{B}_{i^+}]\cap\cent(\mathbb{C}[\mathcal{B}_{i^-}])$, $\sigma\in S_m$, and $\phi$ the isomorphism of Lemma \ref{configspaces},
\begin{itemize}
\item[1.] $x_1\cdots x_m= \phi^{-1}((((\tilde{x}_{\sigma(1)}*\tilde{x}_{\sigma(2)})*\tilde{x}_{\sigma(3)})\cdots \tilde{x}_{\sigma(m-1)})*\tilde{x}_{\sigma(m)})$
\item[2.] This may be computed in at most $\displaystyle\sum_{i=1}^{m-1}\dim A(Q_1^\sigma\triangle\cdots\triangle Q_i^\sigma\cup Q_{i+1}^\sigma;\mathcal{B})$ scalar multiplications, and fewer additions. 
\end{itemize} 

\end{theorem}
\begin{proof} Part 1 follows from Definition \ref{resprod} and Theorem \ref{multthm}. To prove Part 2, apply Lemma \ref{rprodcount}, note that the map $\phi^{-1}$ requires no operations to compute, and also note that for any quiver $R$ with subquiver $Q$, $\dim A(Q\uparrow R;\mathcal{B})\leq \dim A(Q;\mathcal{B}).$
\end{proof}

Now let  $V_j^\sigma:=A(Q_1^\sigma\triangle\cdots\triangle Q_j^\sigma;\mathcal{B})$ and for $2\leq j\leq m-1$ let $*_j:V_{j-1}^\sigma \times X_{\sigma(j)}\rightarrow V_{j}^\sigma$ send $(v_j,x_{\sigma(j+1)})$ to $v_j*\tilde{x}_{\sigma(j+1)}$. Let $*_{m}(v_m, x_{\sigma(m)})=\phi^{-1}(v_m*\tilde{x}_{\sigma(m)})$. Then by Theorem \ref{cond45}, for $w_i=x_{\sigma(i)}$, $$x_1\cdots x_m=(((w_1*_2w_2)*_3w_3)\cdots*_mw_m)$$ as required by part II of the SOV approach, giving Theorem \ref{efficiencyfirststate}:

\begin{theorem}\label{efficiency}[Restatement of Theorem \ref{efficiencyfirststate}]
For $x_i$ and $\sigma$ as in the SOV Approach \ref{alg2}, let $Q_i^\sigma$ denote the quiver such that $X_{\sigma(i)}\cong A(Q_i;\mathcal{B})$. Then we may compute $\sum_{y\in Y} \tilde{y}F_y$ in at most $$\sum_{i=1}^{m-1}\vert W_{i-1}\vert\dim A(Q_1^\sigma\triangle\cdots \triangle Q_i^\sigma\cup Q_{i+1}^\sigma;\mathcal{B})$$ multiplications and $$ (\vert W_0\vert-\vert W_1\vert)+\sum_{i=1}^{m-1} (\vert W_{i-1}\vert-\vert W_i)\vert(\dim A(Q_1^\sigma\triangle\cdots \triangle Q_i^\sigma\cup Q_{i+1}^\sigma;\mathcal{B})-\dim A(Q_1^\sigma\triangle\cdots \triangle Q_{i+1}^\sigma;\mathcal{B}))$$ additions. 
\end{theorem}


\section{Determining the Dimension of Configuration Spaces}\label{quivcounts}
In Section \ref{proofsofstuff} we developed aspects of the general quiver formalism to provide the technical bedrock for the SOV algorithm (esp., the definitions of configuration space and restricted product, and basic complexity counts in terms of morphisms). The final step in computing the complexities of the algorithms outlined in Section \ref{applications} is to finally rewrite the morphism counts in terms of multiplicities for the restriction of representations from one group algebra to another. That is the purpose of this section. Here we accomplish this by adapting earlier work of Stanley's on differential posets \cite{diffposets}, a context that can also be used for Bratelli diagrams. In this section our main result is the final Corollary (Corollary \ref{hom=dim}) that computes $\#\Hom(Q,\mathcal{B})$ (for a so-called "n-toothed quiver" $Q$ and Bratteli diagram $\mathcal{B}$) in terms of spectral information from "up" and "down" operators on the diagram. We apply these results in Appendix \ref{appC} and Appendix \ref{appgenlin} to give the explicit counts of Section \ref{applications}.
\subsection{General Morphism Counts}\label{generalcounts}
 Recall from Note \ref{finitedim} that if $Q$ is a finite subquiver of $R$ and $B$ a locally finite quiver, $\dim A(Q\uparrow R; B)=\#\Hom(Q\uparrow R; B).$ In the SOV approach, $B$ is the Bratteli diagram associated to a chain of semisimple algebras and hence locally finite, so in this section we give results to count $\#\Hom(Q\uparrow R;B)$. 

For $B$ a locally finite graded quiver, $\alpha,\beta\in V(B)$, let $M_B(\alpha,\beta)$ denote the number of paths from $\beta$ to $\alpha$ in $B$. Note that for $\mathcal{B}$ a Bratteli diagram, $\alpha,\beta\in V(\mathcal{B})$ correspond to irreducible representations $\gamma$, $\rho$ and
$M_\mathcal{B}(\alpha,\beta)=M(\gamma, \rho),$ as in Definition \ref{brattdef}.  
\begin{theorem}\label{countingmorph}
Let $Q,R,B$ be graded quivers with $Q$ a finite subquiver of $R$ and $B$ locally finite. Then 
$$\#\Hom(Q\uparrow R; B)=\sum_{\phi\in\Hom(V(Q)\uparrow R; B)}\prod_{\substack{\text{arrows}\\ \beta\rightarrow \alpha\\\text{in}\;\; Q}} M_B(\phi \alpha,\phi \beta)$$
\end{theorem} 
\begin{proof}
A morphism specifies the image of each vertex and each arrow. This may be counted by first fixing the image of each vertex and counting all possible arrow images, then varying over all possible images of $V(Q)$. 
\end{proof}
Theorem \ref{countingmorph} gives a procedure for computing $\#\Hom(Q\uparrow R; B)$. For a quiver $Q$, let $Q^i$ denote the vertices of $Q$ at level $i$. Then:
\begin{itemize}
\item[1.] label each vertex $\alpha_i\in Q^i$ with a vertex $\alpha_i'\in B^i$ such that this labeling could extend to a map from $R$ into $B$;
\item[2.] label each edge of $Q$ from $\beta$ to $\alpha$ by $M_B(\alpha',\beta')$;
\item[3.] multiply the labels and sum over all possible labellings.
\end{itemize}

\begin{example}\label{Q1R1} Let $Q_1=R_1$ be as in Figure \ref{labelling}. Steps $1$ and $2$ then give the labelling of the figure and by Theorem \ref{countingmorph}, $$\#\Hom(Q\uparrow R; B)=\sum_{\alpha_i'\in B^i}M_B(\alpha_5',\alpha_0')M_B(\alpha_5',\alpha_3')M_B(\alpha_4',\alpha_3')M_B(\alpha_4',\alpha_0').$$

\begin{figure}\begin{center}
\begin{tikzpicture}[shorten >=1pt,node distance=2cm,on grid,auto,/tikz/initial text=] 
   \node at (-7.875,2) (qr) {\large $Q_1=R_1$};
   \node at (-1.875,2) (label) {\large After Labelling};
   \node[pnt] at (-6,0) (a0) {};
   \node[pnt] at (-8.25,0) (a3) {};
   \node[pnt] at (-9,1) (a4) {};
   \node[pnt] at (-9.75,-1) (a5) {};
   \node[pnt] at (0,0) (a0') {};
   \node[pnt] at (-2.25,0) (a3') {};
   \node[pnt] at (-3,1) (a4') {};
   \node[pnt] at (-3.75,-1) (a5') {};
   \node at (-1.4,1.25) (m40) {\scriptsize $M_B(\alpha_4',\alpha_0')$};
   \draw (a0') node[below] {\footnotesize $\alpha_0'$};
   \draw (a3') node[below] {\footnotesize $\alpha_3'$};
   \draw (a4') node[above] {\footnotesize $\alpha_4'$};
   \draw (a5') node[below] {\footnotesize $\alpha_5'$};
   \draw (a0) node[below] {\footnotesize $\alpha_0$};
   \draw (a3) node[below] {\footnotesize $\alpha_3$};
   \draw (a4) node[above] {\footnotesize $\alpha_4$};
   \draw (a5) node[below] {\footnotesize $\alpha_5$};
   
   \path[->,every node/.style={font=\scriptsize}]
    (a0') edge [bend right=35] node {} (a4')
    (a3') edge  node {$M_B(\alpha_4',\alpha_3')$} (a4')    
    (a3') edge  node {$M_B(\alpha_5',\alpha_3')$} (a5')
    (a0') edge [bend left=35] node {$M_B(\alpha_5',\alpha_0')$} (a5')
    (a0) edge [bend right=35] node {} (a4)
    (a3) edge  node {} (a4)    
    (a3) edge  node {} (a5)
    (a0) edge [bend left=35] node {} (a5);

\end{tikzpicture}
\caption{Labelling $Q_1$ using the procedure of Theorem \ref{countingmorph}. }
\label{labelling}
\end{center}\end{figure}
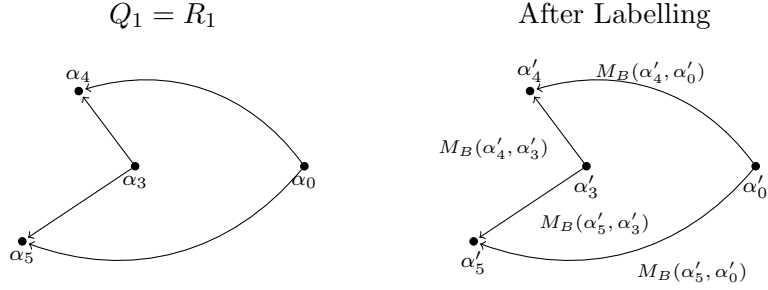

\end{example}
To further simplify counts, we first `smooth' quivers before counting morphisms, i.e. we remove superfluous vertices (see Corollary \ref{smooth} in Appendix \ref{quiversmooth}).

\begin{example} Let $Q_2=R_2$ be as in Figure \ref{Q2=R2}. Then for $Q_1,R_1$ as in Figure \ref{labelling}, Corollary \ref{smooth} gives the isomorphism
$$A(Q_2\uparrow R_2;B)\cong A(Q_1\uparrow R_1; B).$$
To compute $\#\Hom(Q_2\uparrow R_2;B),$ remove vertices $\alpha_2$ and $\alpha_1$, then use the labelling of Figure \ref{labelling}.
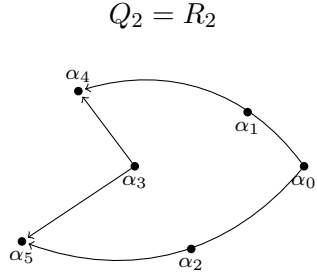
\begin{figure}[H]\begin{center}
\begin{tikzpicture}[shorten >=1pt,node distance=2cm,on grid,auto,/tikz/initial text=] 
   \node at (-1.875,2) (qr) {\large $Q_2=R_2$};
   \node[pnt] at (0,0) (a0) {};
   \node[pnt] at (-.75,.72) (a1) {};
   \node[pnt] at (-1.5,-1.1) (a2) {};
   \node[pnt] at (-2.25,0) (a3) {};
   \node[pnt] at (-3,1) (a4) {};
   \node[pnt] at (-3.75,-1) (a5) {};
   \draw (a0) node[below] {\footnotesize $\alpha_0$};
   \draw (a1) node[below] {\footnotesize $\alpha_1$};
   \draw (a2) node[below] {\footnotesize $\alpha_2$};
   \draw (a3) node[below] {\footnotesize $\alpha_3$};
   \draw (a4) node[above] {\footnotesize $\alpha_4$};
   \draw (a5) node[below] {\footnotesize $\alpha_5$};
   
   \path[->,every node/.style={font=\scriptsize}]
    (a0) edge [bend right=35] node {} (a4)
    (a3) edge  node {} (a4)    
    (a3) edge  node {} (a5)
    (a0) edge [bend left=35] node {} (a5);

\end{tikzpicture}
\caption{A quiver $Q_2$ that `smooths' to $Q_1$.}
\label{Q2=R2}
\end{center}\end{figure}
\end{example}
\subsection{Morphisms into Locally Free Bratteli Diagrams}\label{morphisms}
In Section \ref{generalcounts} we obtained general quiver morphism counting results for $B$ a locally finite quiver. For \textit{locally free} Bratteli diagrams we rewrite these results in terms of the dimensions of the corresponding subalgebras. 

\begin{definition}\label{locfree}
A Bratteli diagram $\mathcal{B}$ is \textbf{locally free} if for each $i\geq 1$, $\mathbb{C}[\mathcal{B}_i]$ is free as a module over $\mathbb{C}[\mathcal{B}_{i-1}]$. 
\end{definition}
\begin{example}
For $G$ a finite group, the Bratteli diagram associated to the chain of group algebras $\mathbb{C}[G]=\mathbb{C}[G_n]>\cdots> \mathbb{C}[G_0]=\mathbb{C}$ is locally free.
\end{example}

Let $\mathbb{C}[V(\mathcal{B})]$ denote the space of finitely supported linear combinations of vertices of $\mathcal{B}$, let $\mathcal{B}^i$ denote the vertices $\alpha\in V(\mathcal{B})$ with $gr(\alpha)=i$,  and let $\mathbb{C}[\mathcal{B}^i]$ denote the space of finitely supported linear combinations of vertices at level $i$ in $\mathcal{B}$. Define an inner product $\langle\;,\;\rangle$ on $\mathbb{C}[V(\mathcal{B})]$ making the vertices orthonormal. As in \cite{diffposets} define linear operators $U$ and $D$ on $\mathbb{C}[V(\mathcal{B})]$ by linearly extending the action on $\alpha\in \mathcal{B}^i$:
$$U\alpha=\sum_{\gamma\in \mathcal{B}^{i+1}} M_\mathcal{B}(\gamma,\alpha)\gamma,$$
$$D\alpha=\sum_{\beta\in \mathcal{B}^{i-1}} M_\mathcal{B}(\alpha,\beta)\beta,$$
where, by convention, if $\mathcal{B}$ has highest grading $n$, $\mathcal{B}^{-1}=\emptyset=\mathcal{B}^{n+1}=\mathcal{B}^{n+2}=\cdots.$

\begin{note}\label{virtualrep}
As the vertices of $\mathcal{B}$ are labeled by the irreducible representations of $\mathbb{C}[\mathcal{B}_{i}]$, elements of $\mathbb{C}[\mathcal{B}^i]$ correspond to representations of the path algebra $\mathbb{C}[\mathcal{B}_{i}]$. In this context, $U$ is induction and $D$ restriction (see \cite{towers} Proposition 2.3.1).
\end{note}
\begin{example}\label{words}
For $Q_1$ as in Example \ref{Q1R1}, trace each arrow on the quiver by starting at the root, moving up four levels to vertex $\alpha_4$, down to vertex $\alpha_3$, up two levels to vertex $\alpha_5$, and back down to the root. It is then easily checked that 
$$\begin{array}{ll}\langle D^5U^2DU^4\hat{0},\hat{0}\rangle & =\sum_{\alpha_i'\in B^i}M_B(\alpha_5',\alpha_0')M_B(\alpha_5',\alpha_3')M_B(\alpha_4',\alpha_3')M_B(\alpha_4',\alpha_0')\\
&=\#\Hom(Q\uparrow R; B).\end{array}$$
\end{example}
In Corollary \ref{hom=dim} we give explicit formulas for these inner products. For $\alpha\in V(B)$ and $\hat{0}$ the root of $\mathcal{B}$, let $d_\alpha=M_\mathcal{B}(\alpha,0)$ and let $d_i=\sum_{\alpha\in\mathcal{B}^i} d_\alpha \alpha$. 
\begin{lemma}\label{di=ui0}\begin{itemize}
\item[]
\item[(i)]For $\alpha\in\mathcal{B}^i$, $\langle d_i,\alpha\rangle = d_\alpha.$
\item[(ii)]$d_i=U^i\hat{0},$
\end{itemize}
\end{lemma}
\begin{proof}
Clear from definition and induction.
\end{proof}
\begin{proposition}\label{locfreeprop} Let $\mathcal{B}$ be a Bratteli diagram. Then the following properties are equivalent:
\begin{itemize}
\item[(i)] $\mathcal{B}$ is locally free.
\item[(ii)] For each $i$ and all $\beta\in \mathcal{B}^{i-1}$, there exists $\lambda_i\in \mathbb{C}$ such that $$\sum_{\alpha\in\mathcal{B}^i} M_\mathcal{B}(\alpha,\beta)d_\alpha=\lambda_i d_\beta.$$
\item[(iii)] For each $i$, $d_i$ is an eigenvector of $DU$. 
\item[(iv)] For each $i$ there exists $\lambda_i\in \mathbb{C}$ with $DU^i\hat{0}=\lambda_iU^{i-1}\hat{0}$.
\end{itemize}
\end{proposition}
\begin{proof} As this proof comes down to definitions and the fact that $D$ is restriction (cf. Note \ref{virtualrep}), we defer it to Appendix \ref{proplocfree}. 
\end{proof}

\begin{corollary}
Let $\mathcal{B}$ be a locally free Bratteli diagram and $\lambda_i$ the eigenvalue of $DU$ associated to $d_{i-1}$. Then $\lambda_i$ is integral and \begin{itemize}
\item[(i)] $\displaystyle\lambda_i=\frac{\dim_\mathbb{C}\mathbb{C}[\mathcal{B}_i]}{\dim_\mathbb{C}\mathbb{C}[\mathcal{B}_{i-1}]}$,
\item[(ii)] $\displaystyle\dim_\mathbb{C}\mathbb{C}[\mathcal{B}_i]=\prod_{j=1}^i\lambda_j$.
\end{itemize}
\end{corollary}
\begin{example} For a group algebra chain $\mathbb{C}[G_n]>\cdots>\mathbb{C}[G_0]$, the corresponding Bratteli diagram $\mathcal{B}$ is locally free and 
$$\lambda_i=\frac{\dim_\mathbb{C}\mathbb{C}[\mathcal{B}_i]}{\dim_\mathbb{C}\mathbb{C}[\mathcal{B}_{i-1}]}=\frac{\dim_\mathbb{C}\mathbb{C}[G_i]}{\dim_\mathbb{C}\mathbb{C}[G_{i-1}]}=\vert G_i/G_{i-1}\vert.$$ 
\end{example}
Theorem \ref{freecount} below generalizes Theorem 3.7 of \cite{diffposets} and Theorem 2.3 of  \cite{diffposets2}.
\begin{definition} Let $w=w_l\cdots w_1$ be a word in $U$ and $D$ and let $\mathcal{S}=\{i\mid w_i=D\}.$ For each $i\in \mathcal{S}$, let
$a_i=\# \{D\text{'s in} \;w\; \text{to the right of} \;w_i\},$ and similarly let $b_i=\# \{U\text{'s in}\; w\; \text{to the right of} \;w_i\}.$ 
If $b_i-a_i\geq 0$ for all $i\in \mathcal{S}$, we call $w$ an \textbf{admissible word}.
\end{definition}
\begin{theorem}\label{freecount}
Let $\mathcal{B}$ be a locally free Bratteli diagram and $w=D^{d_n}U^{u_n}\cdots D^{d_1}U^{u_1}$ an admissible word in $U$ and $D$. Then for $s=\sum_{i=1}^n u_i-d_i$ and $\alpha\in \mathcal{B}^{s}$, 
$$\langle w\hat{0},\alpha\rangle=d_\alpha\prod_{i\in \mathcal{S}}\lambda_{b_i-a_i}.$$
\end{theorem}
\begin{proof}
The proof comes down to inductively showing: $$w_k\hat{0}=\prod_{i\in \mathcal{S}_k}\lambda_{b_i-a_i}\sum_{\alpha\in\mathcal{B}^{s_k}} d_\alpha \alpha.$$ For full details, see Appendix \ref{proplocfree}
\end{proof}

For $\mathcal{B}$ a locally free Bratteli diagram and $Q$ an \textit{n-toothed} quiver, Theorem \ref{freecount} allows us to determine $\#\Hom(Q; \mathcal{B})$.
\begin{definition} A quiver $Q$ is \textbf{n-toothed} if it consists of $2n+1$ (not necessarily distinct) vertices $\gamma_0,\dots,\gamma_n,\beta_1,\dots,\beta_n$ and distinct arrows connecting $\gamma_{i-1}$ to $\beta_i$ and $\gamma_i$ to $\beta_i$.
\end{definition}
\begin{example} The quiver of Figure \ref{masex} is an example of a  $3$-toothed quiver.
\end{example}

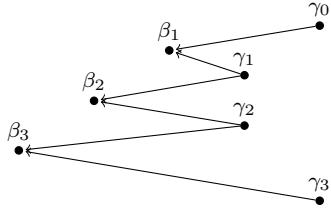
\begin{figure}[H]\begin{center}
\begin{tikzpicture}[shorten >=1pt,node distance=2cm,on grid,auto,/tikz/initial text=] 
   \node[pnt] at (3,0) (g0) {};
   \node[pnt] at (3,-2.33) (g3) {};
   \node[pnt] at (2,-.66) (g1) {};
   \node[pnt] at (2,-1.33) (g2) {};
   \node[pnt] at (1,-.33) (b1) {};
   \node[pnt] at (0,-1) (b2) {};
   \node[pnt] at (-1,-1.66) (b3) {};
   \draw (g0) node[above] {\footnotesize $\gamma_0$};
   \draw (g1) node[above] {\footnotesize $\gamma_1$};
   \draw (g2) node[above] {\footnotesize $\gamma_2$};
   \draw (g3) node[above] {\footnotesize $\gamma_3$};
   \draw (b1) node[above] {\footnotesize $\beta_1$};
   \draw (b2) node[above] {\footnotesize $\beta_2$};
   \draw (b3) node[above] {\footnotesize $\beta_3$};
   \path[->,every node/.style={font=\scriptsize}]
    (g0) edge  node {} (b1)
    (g1) edge  node {} (b1)
    (g1) edge  node {} (b2)
    (g2) edge  node {} (b2)
    (g2) edge  node {} (b3)
    (g3) edge  node {} (b3);
\end{tikzpicture}
\caption{A $3$-toothed quiver.}
\label{masex}
\end{center}\end{figure}
\begin{example}The quiver $Q_1$ of Figure \ref{labelling} is $2$-toothed, with $\gamma_0=\alpha_0, \gamma_1=\alpha_3, \gamma_2=\alpha_0, \beta_1=\alpha_4, \beta_2=\alpha_5$ .  
\end{example}

\begin{theorem}\label{hominner}
Let $\mathcal{B}$ be a locally free Bratteli diagram, $Q$ an $n$-toothed quiver with vertices $\gamma_i$ at level $l_i$, $\beta_i$ at level $m_i$. Then for $w=D^{{m_n}-{l_n}}U^{{m_n}-{l_{n-1}}}\cdots D^{{m_1}-{l_1}}U^{{m_1}-{l_0}},$ 
$$\sum_{\alpha\in \mathcal{B}^{l_n-l_0}}\langle w\hat{0},\alpha\rangle=\#\Hom(Q;\mathcal{B}).$$
\end{theorem}
\begin{proof} Follows from Theorem \ref{countingmorph} and induction. 
\end{proof}
\begin{corollary}\label{hom=dim} Let $\mathcal{B}$ be a locally free Bratteli diagram, $Q$ an $n$-toothed quiver with vertices $\gamma_i$ at level $l_i$, $\beta_i$ at level $m_i$. Then for $w=D^{{m_n}-{l_n}}U^{{m_n}-{l_{n-1}}}\cdots D^{{m_1}-{l_1}}U^{{m_1}-{l_0}},$ 
$$\#\Hom(Q;\mathcal{B})=\sum_{\alpha\in \mathcal{B}^{l_n-l_0}}\langle w\hat{0},\alpha\rangle =\prod_{i\in \mathcal{S}}\lambda_{b_i-a_i}\sum_{\alpha\in \mathcal{B}^{l_n-l_0}}d_\alpha.$$
\end{corollary}

\begin{example} For $Q_1$ as in Example \ref{Q1R1}, we see that $$\begin{array}{llll} l_o=l_2=0, &l_1=3,& m_1=4,& m_2=5.\end{array}$$ Then $\mathcal{B}^{l_2-l_0}=\mathcal{B}^0=\hat{0}$ and by Corollary \ref{hom=dim}, for $w=D^5U^2DU^4$, 
$$\#\Hom(Q_1;\mathcal{B})=\langle w\hat{0},\hat{0}\rangle=\lambda_4\lambda_5\lambda_4\lambda_3\lambda_2\lambda_1.$$
Note that this is the inner product of Example \ref{words}.\end{example}

We use Corollary \ref{hom=dim} in Appendix \ref{appC} and Appendix \ref{appgenlin} to give many of the complexity results needed for the proofs in Section \ref{applications}.
\section{Further Directions}
The SOV approach produces savings by first treating the Fourier transform as a collection of scalar equations and then recursively structuring the summation so as to collect together irreducible matrix elements, viewed under the translation to the path algebra as pairs of paths. Through this translation, a sequence of multiplications becomes a sequence of bilinear maps indexed by subgraphs. Efficiency counts are determined by the size of the factorization sets needed for these multiplications, as well as the number of occurrences of these subgraphs in the Bratteli diagram. The resultant savings are dependent on the choice of factorization as well as combinatorial path-counting methods used to provide the bounds in Appendix \ref{appC} and Appendix \ref{appgenlin}. Different choices of subgroups could provide better bounds, and in fact some applications of the Fourier transform require particular chains of parabolic subgroups \cite{diaconisspec, isotypicproj,lanczosit}, which we will investigate in further work. 

In addition, our results can be generalized beyond Fourier transforms on groups. In fact, the path algebra isomorphism of Corollary \ref{bratchain} is true for the Bratteli diagram associated to any semisimple algebra. In work being prepared for publication, we extend the SOV approach to Fourier transforms on semisimple algebras and determine complexity results for the Hecke, Brauer and Birman-Wenzl-Murakami algebras \cite{algpaper}. 

\appendix

\section{Gel'fand-Tsetlin Bases and Adapted Representations}\label{appB}
In Section \ref{prelim} we introduce adapted bases and systems of Gel'fand Tsetlin bases. Here we make the formal connection between adapted bases of a group algebra chain and systems of Gel'fand Tsetlin bases for the corresponding chain of path algebras.
\begin{definition} Given a Bratteli diagram $\mathcal{B}$, a \textbf{representation of $\mathbf{\mathcal{B}}$} assigns to each $\alpha\in V(\mathcal{B})$ a linear space $V_\alpha$ and to each edge $e\in E(\mathcal{B})$  a linear map $L_e:V_{s(e)}\rightarrow V_{t(e)}$. Given two representations $(\{V_\alpha\}_{\alpha\in V(\mathcal{B})}, \{L_e\}_{e\in E(\mathcal{B})})$, $(\{W_\alpha\}_{\alpha\in V(\mathcal{B})}, \{S_e\}_{e\in E(\mathcal{B})})$, of $\mathcal{B}$, a \textbf{morphism} $m:V\rightarrow W$ is a family of linear maps $\{m_\alpha:V_\alpha\rightarrow W_\alpha\}_{\alpha\in V(\mathcal{B})}$ such that the diagram
$$\begin{array}{llll}&V_{s(e)}&\xrightarrow{\;L_e\;} &V_{t(e)}\\
m_{s(e)}&\downarrow&&\downarrow \;\;m_{t(e)}\\
&W_{s(e)}&\xrightarrow{\;S_e\;}&W_{t(e)}\end{array}$$
commutes for all $e\in E(\mathcal{B})$.

A \textbf{model representation of} $\mathbf{\mathcal{B}}$ is a representation of $\mathcal{B}$ such that for all $e\in E(\mathcal{B})$, $L_e$ is injective, and for all nonroot vertices $\alpha\in V(\mathcal{B})$, 
$$V_\alpha=\displaystyle\bigoplus_{t(e)=\alpha} Im(L_e).$$ 
\end{definition}
\begin{definition} Given a chain of group algebras $\{\mathbb{C}[G_i]\}$, a \textbf{model representation for} $\{\mathbb{C}[G_i]\}$ is a model representation of the corresponding Bratteli diagram $\mathcal{B}$ such that: \begin{itemize}
\item[(i)] for each $\alpha\in V(\mathcal{B})$ at level $i$, $V_\alpha$ is the representation space of the representation of $\mathbb{C}[G_i]$ corresponding to $\alpha$,
\item[(ii)]for each $e\in E(\mathcal{B})$ from level $i$ to level $i+1$, $L_e$ is $\mathbb{C}[G_i]$-equivariant, i.e., for $\rho_{s(e)}$  the representation of $\mathbb{C}[G_i]$ corresponding to $s(e)$ and $\rho_{t(e)}$ the representation of $\mathbb{C}[G_{i+1}]$ corresponding to $t(e)$, the diagram 
$$\begin{array}{llll}&V_{s(e)}&\xrightarrow{\;L_e\;} &V_{t(e)}\\
\rho_{s(e)}&\downarrow&&\downarrow \;\;\rho_{t(e)}\\
&V_{s(e)}&\xrightarrow{\;L_e\;}&V_{t(e)}\end{array}$$
commutes for all $e\in E(\mathcal{B})$.
\end{itemize}
\end{definition}  
%
A model representation of an algebra chain has a natural basis of paths:
\begin{lemma}\label{modelpaths} Given a model representation of a chain of subalgebras with Bratteli diagram $\mathcal{B}$, the collection of distinct paths in $\mathcal{B}$ from the root to a vertex $\alpha\in V(\mathcal{B})$ corresponds to a choice of basis for $V_\alpha$.
\end{lemma}
\begin{proof}
Consider the space $V_\beta$ corresponding to the root $\beta$, i.e., $V_\beta$ is the representation space of $\mathbb{C}[\mathcal{B}_0]=\mathbb{C}$, so $V_\beta$ is one-dimensional. Now let $\alpha$ be a vertex in $\mathcal{B}$ with $gr(\alpha)=1$. Then $V_\alpha=\displaystyle\displaystyle\bigoplus_{t(e)=\alpha} Im(L_e)\cong\displaystyle\bigoplus_{t(e)=\alpha}V_\beta$ since $L_e$ is injective. 
Induction gives the result.  
\end{proof}
Thus, given an irreducible representation $\rho$ of $\mathbb{C}[G_i]$ corresponding to a vertex $\alpha$ in the Bratteli diagram associated to the chain of group algebras, there is a basis for the representation space of $\rho$ indexed by the paths from the root to $\alpha$. We call such a basis a Gel'fand-Tsetlin basis, as in Definition \ref{defgel}. 
Given a model representation of a Bratteli diagram $\mathcal{B}$, Lemma \ref{modelpaths} gives a system of Gel'fand-Tsetlin bases for $\mathcal{B}$. In fact, these are equivalent concepts:
\begin{theorem}\label{geleqmod}
A system of Gel'fand-Tsetlin bases for a Bratteli diagram $\mathcal{B}$ uniquely determines a model representation for $\mathcal{B}$. 
Conversely, a model representation uniquely determines a system of Gel'fand-Tsetlin bases for $\mathcal{B}$. 
\end{theorem} 
\begin{proof}
Both require a choice of vector space for each vertex of $\mathcal{B}$, so we need only show how a choice of basis corresponds with linear maps $L_e$ for each edge $e$.

Given a system of bases and an edge $e\in\mathcal{B}$, a basis vector for $V_{s(e)}$ corresponds to a path $P$ from the root to $s(e)$. Then $e\circ P$ is a path from the root to $t(e)$, which corresponds to a basis vector for $V_{t(e)}$. In other words, we have an injection of $V_{s(e)}$ into $V_{t(e)}$. 

Conversely, given a model representation and a vertex $\alpha$, every path from the root to $\alpha$ corresponds to an injection of $\mathbb{C}$ into $V_\alpha$. Since $$V_\alpha=\displaystyle\bigoplus_{t(e)=\alpha} Im(L_e),$$ the union of the distinct images of $1\in\mathbb{C}$ over the collection of injections gives a basis for $V_\alpha$ as we vary over all possible paths from the root to $\alpha$.
\end{proof}
%
\begin{remark}\label{gel=adapt} The equivalent definitions of Gel'fand-Tsetlin bases and model representations coincide with the notion of a complete set of adapted representations for chains of groups.
Clearly a model representation for the group algebra chain gives rise to an adapted basis since the isomorphism 
$$V_\alpha=\displaystyle\bigoplus_{t(e)=\alpha} Im(L_e)\cong\displaystyle\bigoplus_{\substack{e\in E(\mathcal{B}),\\ t(e)=\alpha}} V_{s(e)}$$
describes how the representation space $V_{\alpha}$ decomposes at level $i-1$. Equivariance of the maps $L_{e}$ then gives the decomposition of the representation $\rho_\alpha$. 

Further, a complete set $R$  of inequivalent irreducible representations adapted to a chain of subgroups $G_n>G_{n-1}>\cdots>G_0$ determines the paths in the Bratteli diagram $\mathcal{B}$ of the group algebra chain by drawing $M(\rho, \gamma)$ arrows from a representation $\gamma\in R$ of $G_i$ to a representation $\rho\in R$ of $G_{i+1}$. Then a set of bases for the representation spaces of the representations in $R$ determines a system of Gel-fand Tsetlin bases for the group algebra chain, and so by Theorem \ref{geleqmod} a model representation.

\end{remark}

\section{Restricted Product Lemmas}\label{appA}
In this Appendix, we prove the associativity of the restricted product defined in Section \ref{proofsofstuff}.
\begin{lemma}\label{rprodass} Let $B$ be a locally finite graded quiver and $R$ a graded quiver with finite subquivers $Q_1, Q_2,\dots,Q_m$ such that $Q_i\cap Q_j\cap Q_k$ has no edges for all $i\neq j\neq k$. Let $Q_i^\triangle$ denote the quiver $Q_1\triangle\cdots\triangle Q_i$ and let $Q_i^\cup$ denote the quiver $Q_1\cup\cdots\cup Q_i$. Then for $f_i\in A(Q_i;B)$, $f_1*f_2*\cdots *f_m$ is independent of bracketing. Moreover, for $\tau\in \Hom(Q_m^\triangle;B)$ and $\iota_k$ the natural injection $Q_k\hookrightarrow R$,
\begin{equation}\label{ass}(f_1*f_2*\cdots *f_m)\vert_{\tau}=\sum_{\substack{\eta\in\Hom(Q_m^\cup\uparrow R; B), \\ \eta\downarrow_{Q_m^\triangle}=\tau}}\prod_{k=1}^m f_k\vert_{\eta\circ\iota_k}.
\end{equation} 
\end{lemma}
\begin{proof}
We first prove (\ref{ass}) inductively, as associativity clearly follows. For $n=2$, (\ref{ass}) is the definition of the restricted product $f_1*f_2$. 

Now suppose (\ref{ass}) holds for $n-1$. Since $Q_i\cap Q_j\cap Q_k=\emptyset$, 
\begin{equation}\label{symdifprop1}[(Q_1\triangle\cdots\triangle Q_{n-1})\cup Q_n]\cap[Q_1\cup\cdots\cup Q_{n-1}]= Q_1\triangle\cdots\triangle Q_{n-1},
\end{equation} 
and
\begin{equation}\label{symdifprop2}
[(Q_1\triangle\cdots\triangle Q_{n-1})\cup Q_n]\cup[Q_1\cup\cdots\cup Q_{n-1}]= Q_1\cup\cdots\cup Q_{n}.
\end{equation}
By the induction hypothesis, 
$$(f_1*\cdots*f_{n-1}*f_n)\vert_\tau=\sum_{\substack{\eta\in\Hom(Q_{n-1}^\triangle\cup Q_n\uparrow R;B)\\ \eta\downarrow_{Q_n^\triangle}=\tau}}\sum_{\substack{\mu\in\Hom(Q_{n-1}^\cup\uparrow R; B)\\ \mu\downarrow_{Q_{n-1}^\triangle}=\eta\downarrow_{Q_{n-1}^\triangle}}}\left(\prod_{k=1}^{n-1} f_k\vert_{\mu\circ\iota_k}\cdot f_n\vert_{\eta\circ\iota_n}\right).$$
 By (\ref{symdifprop1}) and (\ref{symdifprop2}), each choice of $\mu$ and $\eta$ which agree on their intersection, the subquiver $Q_{n-1}^\triangle$, uniquely determines a morphism $\gamma\in\Hom(Q_n^\cup\uparrow R; B)$ such that \begin{equation*}\begin{array}{ll}\gamma\downarrow_{Q_1\triangle\cdots\triangle Q_{n}}=&\eta\downarrow_{Q_1\triangle\cdots\triangle Q_{n}}=\tau,\\\gamma\circ\iota_k=\mu\circ\iota_k,& \text{for } 1\leq k\leq n-1,\\ \gamma\circ\iota_n=\eta\circ\iota_n.\end{array}\end{equation*} 
\end{proof}

\section{Quiver Counts}
In Section \ref{quivcounts}., we rewrite morphism counts in terms of multiplicities of representations and dimensions of subgroup algebras (Corollary \ref{hom=dim}). Here we give the details needed for the proofs of Section \ref{quivcounts}.
\subsection{Smoothing Quivers}\label{quiversmooth} An important simplification in morphism counts is to remove superfluous vertices from quivers, i.e., `smooth' them.
\begin{definition} A quiver $B$ \textbf{factors at level} $\mathbf{i}$ if there are no arrows from a vertex $\alpha\in V(B)$ with $gr(\alpha)<i$ to a vertex $\beta\in V(B)$ with $gr(\beta)>i$.  
\end{definition}
\begin{example} Let $\mathcal{B}$ be a Bratteli diagram with highest grading $n$. Then for all $0\leq i \leq n$,  $\mathcal{B}$ factors at level $i$.
\end{example}
\begin{definition} Let $Q$ be a quiver with a vertex $v$ that is the target of exactly one arrow, $e_1$, and the source of exactly one arrow, $e_2$. To \textbf{smooth Q at v}, remove $v$ and replace $e_1$ and $e_2$ with an arrow from the source of $e_1$ to the target of $e_2$. To \textbf{smooth Q}, smooth $Q$ at all possible $v$. 
\end{definition}
\begin{example}
The quiver $Q'$ of Figure \ref{Q'} results from smoothing the quiver $Q$.
\begin{figure}[H]\begin{center}
\begin{tikzpicture}[shorten >=1pt,node distance=2cm,on grid,auto,/tikz/initial text=] 
   \node at (-2,1) (q) {$Q$};
   \node at (2,1) (q) {$Q'$};
   \node[pnt] at (-1,0) (01) {};
   \node[pnt] at (-2,0) (11) {};
   \node[pnt] at (-3,0) (21) {};
   \node[pnt] at (3,0) (02) {};
   \node[pnt] at (1,0) (22) {};
   \draw (01) node[above] {\footnotesize $0$};
   \draw (11) node[above] {\footnotesize $1$};
   \draw (21) node[above] {\footnotesize $2$};
   \draw (02) node[above] {\footnotesize $0$};
   \draw (22) node[above] {\footnotesize $2$};
   \path[<-,every node/.style={font=\scriptsize}]
    (22) edge  node {} (02)
    (11) edge  node {} (01)
    (21) edge  node {} (11);
\end{tikzpicture}
\caption{Smoothing $Q$}
\label{Q'}
\end{center}\end{figure}
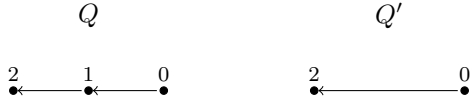
\end{example}
\begin{lemma}
Let $B$ be a graded quiver that factors at level $i$, $R$ a graded quiver with subquiver $Q$, and $v$ a vertex of $Q$ at level $i$ such that $Q$ can be smoothed at $v$. Let $Q'$ (respectively $R'$) be the quiver obtained by smoothing $Q$ (respectively $R$) at $v$. Then $\#\Hom(Q\uparrow R; B)=\#\Hom(Q'\uparrow R'; B).$
\end{lemma}
\begin{proof}
Let $\phi\in\Hom(Q'\uparrow R';B)$ and let $f$ be the arrow in $Q'$ resulting from smoothing $Q$ at $v$. Then $f$ replaced two arrows, $e_1, e_2$ in $Q$, with $t(e_1)=v, s(e_2)=v$. Further, $s(e_1)=s(f), t(e_2)=t(f)$, so $\phi(f)$ is a path in $B$ from a vertex $\alpha$ with $gr(\alpha)<i$ to a vertex $\beta$ with $gr(\beta)>i$. Since $B$ factors at level $i$, this path contains a vertex, $v'$, with $gr(v')=i$. Let $e_1'$ be the subpath of $f$ starting at the source of $f$ and ending at $v'$. Similarly, let $e_2'$ be the subpath of $f$ starting at $v'$ and ending at the target of $f$.

Denote by $\tilde{\phi}$ the morphism in $\Hom(Q\uparrow R;B)$ such that:
$$\begin{array}{ll}\tilde{\phi}(e_1)=e_1',& \tilde{\phi}(e_2)=e_2'\\ \tilde{\phi}(e_i)=\phi(e_i),&\text{for } i\neq 1,2,\\ \tilde{\phi}(v_j)=\phi(v_j),&\text{for } v_j\neq\alpha,\beta,v'.\end{array}$$

Clearly $\phi\rightarrow\tilde{\phi}$ is a bijection.
\end{proof}
\begin{corollary}\label{smooth}
Let $\mathcal{B}$ be a Bratteli diagram, $R$ a graded quiver with subquiver $Q$, and $Q'$ (respectively $R'$) the quiver obtained by smoothing $Q$ (respectively $R$). Then $$\#\Hom(Q\uparrow R; \mathcal{B})=\#\Hom(Q'\uparrow R';\mathcal{B}).$$

\end{corollary}

\subsection{Properties of Locally Free Quivers}\label{proplocfree}
We give the details of the proofs of Proposition \ref{locfreeprop} and Theorem \ref{freecount} of Section \ref{quivcounts}.

\begin{proposition}[Proposition \ref{locfreeprop}] Let $\mathcal{B}$ be a Bratteli diagram. Then the following properties are equivalent:
\begin{itemize}
\item[(i)] $\mathcal{B}$ is locally free.
\item[(ii)] For each $i$ and all $\beta\in \mathcal{B}^{i-1}$, there exists $\lambda_i\in \mathbb{C}$ such that $$\sum_{\alpha\in\mathcal{B}^i} M_\mathcal{B}(\alpha,\beta)d_\alpha=\lambda_i d_\beta.$$
\item[(iii)] For each $i$, $d_i$ is an eigenvector of $DU$. 
\item[(iv)] For each $i$ there exists $\lambda_i\in \mathbb{C}$ with $DU^i\hat{0}=\lambda_iU^{i-1}\hat{0}$.
\end{itemize}
\end{proposition}
\begin{proof}
Statements (ii), (iii), and (iv) are equivalent by definition and Lemma \ref{di=ui0}. For example:
\begin{itemize}
\item[(iii)$\Rightarrow$ (iv)] $DU^i\hat{0}=DUd_{i-1}=\lambda_id_{i-1}=\lambda_iU^{i-1}\hat{0}$
\item[(iv)$\Rightarrow$ (iii)]
$DUd_i=DU^{i+1}\hat{0}=\lambda_{i+1}U^i\hat{0}=\lambda_{i+1}d_i$
\end{itemize}
We leave the remaining equivalences of (ii), (iii), and (iv) to the reader.

To show the equivalence of (i) and (iv), recall from Note \ref{virtualrep} that elements of $\mathbb{C}[\mathcal{B}^i]$ correspond to representations of $\mathbb{C}[\mathcal{B}_i]$, i.e.
$\mathbb{C}[\mathcal{B}_i]$-modules. Under this identification, the regular representation of $\mathbb{C}[\mathcal{B}_i]$ corresponds to the sum 
$$\sum_{\alpha\in\mathcal{B}^i} d_\alpha \alpha=d_i\in \mathbb{C}[\mathcal{B}^i].$$
Since $D$ is restriction (\cite{towers}[Proposition 2.3.1]), the restriction of the regular representation of $\mathbb{C}[\mathcal{B}_i]$ to $\mathbb{C}[\mathcal{B}_{i-1}]$ corresponds to $Dd_i=DUd_{i-1}$.
\begin{itemize}
\item[(i)$\Rightarrow$(iv)] For $\mathcal{B}$ locally free, $\mathbb{C}[\mathcal{B}_i]$ is free as a module over $\mathbb{C}[\mathcal{B}_{i-1}]$ with rank $\lambda_i\in \mathbb{C}$. Then the regular representation of $\mathbb{C}[\mathcal{B}_i]$ decomposes in $\mathbb{C}[\mathcal{B}_{i-1}]$ as $\lambda_i$ copies of the regular representation of $\mathbb{C}[\mathcal{B}_{i-1}]$. Thus, $$DU^i\hat{0}=DUd_{i-1}=Dd_i=\lambda_id_{i-1}=\lambda_iU^{i-1}\hat{0}.$$ 
\item[(iv)$\Rightarrow$(i)] $DUd_{i-1}=Dd_i=\lambda_id_{i-1}$ and so $\dim_\mathbb{C}\mathbb{C}[\mathcal{B}_i]=\lambda_i\dim_\mathbb{C}\mathbb{C}[\mathcal{B}_{i-1}]$; hence $\lambda_i$ is rational and positive. To show $\lambda_i$ integral, let 
$$\begin{array}{lll}\lambda_i=\frac{p}{q},& \gcd(p,q)=1,& p,q>0.\end{array}$$
Let $m=\gcd(d_\beta)$ over all $\beta\in\mathcal{B}^{i-1}$, so $\frac{d_\beta}{m}$ an integer for all $\beta\in\mathcal{B}^{i-1}$. Then $$\begin{array}{ll}\displaystyle\frac{1}{m}\frac{p}{q}\sum_{\beta\in\mathcal{B}^{i-1}} d_\beta \beta &=\displaystyle\frac{1}{m}\lambda_id_{i-1}\\
\\
&=\displaystyle\frac{1}{m}DUd_{i-1}\\
\\
&=\displaystyle\frac{1}{m}\sum_{\alpha\in\mathcal{B}^i}\sum_{\beta\in\mathcal{B}^{i-1}}M_\mathcal{B}(\alpha,\beta)d_\alpha \beta\\
\\
&=\displaystyle\sum_{\alpha\in\mathcal{B}^i}\sum_{\beta\in\mathcal{B}^{i-1}}M_\mathcal{B}(\alpha,\beta)^2\frac{d_\beta}{m}\beta.\end{array}$$
Then the coefficient of $\beta$ is an integer and thus $q\vert \frac{d_\beta}{m}$ for all $\beta\in \mathcal{B}^i$. But $m=\gcd(d_\beta)$, and thus $q=1$, making $\lambda_i$ an integer.
\end{itemize}
\end{proof}

\begin{theorem}[cf. Theorem \ref{freecount}]
Let $\mathcal{B}$ be a locally free Bratteli diagram and $w=D^{d_n}U^{u_n}\cdots D^{d_1}U^{u_1}$ an admissible word in $U$ and $D$. Then for $s=\sum_{i=1}^n u_i-d_i$ and $\alpha\in \mathcal{B}^{s}$, 
$$\langle w\hat{0},\alpha\rangle=d_\alpha\prod_{i\in \mathcal{S}}\lambda_{b_i-a_i}.$$
\end{theorem}
\begin{proof}
To be admissible, $d_{i},u_{i}> 0$ for all $1\leq i\leq n$ and $$\sum_{j=1}^i d_j\leq \sum_{j=1}^i u_j.$$
Let 
$\mathcal{S}_k=\{i\in \mathcal{S}\vert i\leq \sum_{j=1}^k (d_j+u_j)\}$ 
let $s_k=\sum_{i=1}^k u_i-d_i$, and let $w_k=D^{d_k}U^{u_k}\cdots D^{d_1}U^{u_1}$. We prove inductively that $$w_k\hat{0}=\prod_{i\in \mathcal{S}_k}\lambda_{b_i-a_i}\sum_{\alpha\in\mathcal{B}^{s_k}} d_\alpha \alpha.$$

Note that $w_1=D^{d_1}U^{u_1}$. Then $\mathcal{S}_1=\{u_1+1,u_1+2\dots,u_1+d_1\}$
and for all $i\in \mathcal{S}_1$, $b_i=u_1$ and $a_i=i-u_1-1$. By Proposition \ref{locfreeprop}, Lemma \ref{di=ui0}, and induction, 
$$w_1\hat{0}=\displaystyle\prod_{i\in\mathcal{S}_1}\lambda_{b_i-a_i}\sum_{\alpha\in \mathcal{B}^{s_1}}d_\alpha \alpha $$
Now suppose true for $n-1$. Then 
$$\begin{array}{ll} w\hat{0}&=w_n\hat{0}\\
&=D^{d_n}U^{u_n}w_{n-1}\hat{0}\\
&=D^{d_n}U^{u_n}\displaystyle\prod_{i\in \mathcal{S}_{n-1}}\lambda_{b_i-a_i}U^{s_{n-1}}\hat{0}\\
&=\displaystyle\prod_{i\in \mathcal{S}_{n-1}}\lambda_{b_i-a_i}D^{d_n}U^{u_n+s_{n-1}}\hat{0},\end{array}$$
and the same argument as in the base case gives the result.
\end{proof}

\begin{section}{Combinatorial Lemmas for the Weyl Groups}\label{appC}

The SOV approach reduces Theorem \ref{Bnthm} (respectively, Theorem \ref{Dnthm}) to counting the number of morphisms of the quivers of Figure \ref{BnH} into the Bratteli diagram of $B_n$ (respectively, $D_n$).
In this section we consider the Bratteli diagrams associated to $B_n$ and $D_n$ to provide the bounds used in the proofs of Theorems \ref{Bnthm} and \ref{Dnthm}. Note that Lemmas \ref{BnJlemma}, \ref{BnH1}, \ref{DnJlemma}, \ref{DnH1} and Corollaries \ref{BnH3} and \ref{DnH3} all hold for $n\geq 2$, $i\geq 2$.
\subsection{The Weyl Group $B_n$}

The Bratteli diagram $\mathcal{B}$ associated to the chain 
$\mathbb{C}[B_n]>\mathbb{C}[B_{n-1}]>\cdots> \mathbb{C}$
is a generalization of Young's diagram --- inequivalent irreducible representations of $\mathbb{C}[B_i]$ are indexed by pairs of partitions $(\lambda_1,\lambda_2)$,  of $k$ and $l$, respectively, with $k+l=i$. Pairs $(\lambda_1,\lambda_2)$, $(\mu_1,\mu_2)$ are connected by an edge if either $\lambda_1$ may be obtained from $\mu_1$ by adding a box, or if $\lambda_2$ may be obtained from $\mu_2$ by adding a box \cite{seminormal} (see Figure \ref{BnBratt}). Note that this is a multiplicity-free diagram.
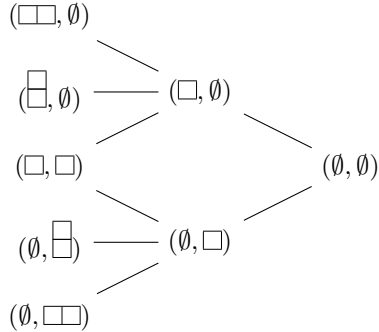
\begin{figure}[H]
\begin{center}
\begin{tikzpicture}[shorten >=1pt,node distance=2cm,on grid,auto,/tikz/initial text=] 
   
   \node at (-4,0) (33) {$(\tiny{\yng(1)},\tiny{\yng(1)})$};
   \node at (-4,1) (32) {$(\tiny{\yng(1,1)},\emptyset$)};
   \node at (-4,-1) (34) {$(\emptyset,\tiny{\yng(1,1)})$};
   \node at (-4,2) (31) {$(\tiny{\yng(2)},\emptyset$)};
   \node at (-4,-2) (35) {$(\emptyset,\tiny{\yng(2)})$};
   \node at (-2,1) (21) {$(\tiny{\yng(1)},\emptyset)$};
   \node at (-2,-1) (22) {$(\emptyset,\tiny{\yng(1)})$};
   \node at (0,0) (0) {$(\emptyset,\emptyset)$};
   \path[every node/.style={font=\scriptsize}]
    (0) edge  node {} (21)
    (0) edge node {} (22)
    (21) edge node {} (31)
    (21) edge node {} (32)
    (21) edge node {} (33)
    (22) edge node {} (33)
    (22) edge node {} (34)
    (22) edge node {} (35);
\end{tikzpicture}
\caption{Bratteli diagram for $B_2$}
\label{BnBratt}
\end{center}
\end{figure}
\begin{theorem} For $\mathcal{J}_i^n$ as in Figure \ref{BnH} and $\mathcal{B}$ the Bratteli diagram associated to the Weyl group $B_n$, $\# \Hom(\mathcal{J}_i^n\uparrow \mathcal{Q};\mathcal{B})\leq 2\vert B_n\vert$.
\end{theorem}
\begin{proof}
By Theorem \ref{countingmorph},  we see that 
$\# \Hom(\mathcal{J}_i^n\uparrow \mathcal{Q};\mathcal{B})$ is equal to the sum
$$\begin{array}{l}\displaystyle\sum_{\alpha_j,\beta_j\in\mathcal{B}^j}M_\mathcal{B}(\beta_{n},\beta_{i})M_\mathcal{B}(\beta_{i},\beta_{i-1})M_\mathcal{B}(\beta_{i-1},\alpha_{i-2})M_\mathcal{B}(\beta_{i},\alpha_{i-1})M_\mathcal{B}(\alpha_{i-1},\alpha_{i-2})d_{\alpha_{i-2}}d_{\beta_{n}}\\
\displaystyle=\sum_{\alpha_j,\beta_j\in\mathcal{B}^j}M_\mathcal{B}(\beta_{n},\beta_{i})M_\mathcal{B}(\beta_{i},\alpha_{i-2})M_\mathcal{B}(\beta_{i},\alpha_{i-1})M_\mathcal{B}(\alpha_{i-1},\alpha_{i-2})d_{\alpha_{i-2}}d_{\beta_{n}}\\
\displaystyle\leq M_\mathcal{B}(B_i,B_{i-2})\sum_{\alpha_j,\beta_j\in\mathcal{B}^j}M_\mathcal{B}(\beta_{n},\beta_{i}) M_\mathcal{B}(\beta_{i},\alpha_{i-1})M_\mathcal{B}(\alpha_{i-1},\alpha_{i-2})d_{\alpha_{i-2}}d_{\beta_n},\end{array}$$
for $M_\mathcal{B}(B_i,B_j):=\max M_\mathcal{B}(\alpha_i,\alpha_j)$ over all $\alpha_i\in\mathcal{B}^i, \alpha_j\in\mathcal{B}^j$.
By Corollary \ref{hom=dim}, 
$$\begin{array}{ll}\# \Hom(\mathcal{J}_i^n\uparrow \mathcal{Q};\mathcal{B})&= M_{\mathcal{B}}(B_{i},B_{i-2}) \langle D^{n}U^{n}\hat{0},\hat{0}\rangle\\
&=M_{\mathcal{B}}(B_{i},B_{i-2})\vert B_{n}\vert.\end{array}$$

Lemma \ref{BnJlemma} below shows that 
$M_{\mathcal{B}}(B_{i},B_{i-2})\leq 2.$ 
Thus 
$$\#\Hom(\mathcal{J}_i^n\uparrow \mathcal{Q};\mathcal{B})\leq 2\vert B_n\vert.$$\end{proof}

\begin{lemma}\label{BnJlemma}
 $M_\mathcal{B}(B_i,B_{i-2})\leq 2$
\end{lemma}
\begin{proof}
Suppose not. Then since $\mathcal{B}$ is multiplicity-free, we must have distinct pairs of partitions 
$\mathbf{\kappa}=(\kappa_1,\kappa_2),$, $\mathbf{\rho}=(\rho_1,\rho_2),$, $\mathbf{\gamma}=(\gamma_1,\gamma_2),$ $\mathbf{\eta}=(\eta_1,\eta_2),$ and $\mathbf{\lambda}=(\lambda_1,\lambda_2)$
 as in Figure \ref{Bnpart}.
 
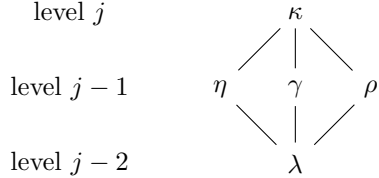
\begin{figure}[H]\begin{center}

\begin{tikzpicture}[shorten >=1pt,node distance=2cm,on grid,auto,/tikz/initial text=] 
   \node at (-3,0) (1) {level $j$};
   \node at (-3,-1) (1) {level $j-1$};
   \node at (-3,-2) (1) {level $j-2$};
   \node at (0,0) (1) {$\mathbf{\kappa}$};
   \node at (0,-1) (22) {$\mathbf{\gamma}$};
   \node at (1,-1) (21) {$\mathbf{\rho}$};
   \node at (-1,-1) (23) {$\mathbf{\eta}$};
   \node at (0,-2) (3) {$\mathbf{\lambda}$};
   \path[every node/.style={font=\scriptsize}]
    (1) edge  node {} (22)
    (1) edge  node {} (21)
    (1) edge node {} (23)
    (21) edge node {} (3)
    (23) edge node {} (3)
    (22) edge node {} (3);               

\end{tikzpicture}
\caption{Subquiver of $\mathcal{B}$ if $M_\mathcal{B}(B_i,B_{i-2})> 2$}
\label{Bnpart}
\end{center}
\end{figure}
Pairs of partitions are adjacent in $\mathcal{B}$ if one is acquired from the other by adding a single box; hence either $\rho_1=\lambda_1$ or $\rho_2=\lambda_2$. The same holds for $\eta,\gamma$. Similarly, $\kappa_1=\rho_1$ or $\kappa_2=\rho_2$ and the same holds for $\eta, \gamma$.

Without loss of generality, we need only consider the following two cases:
 
\noindent \textbf{Case 1:} $\rho_1,\eta_1,\gamma_1=\lambda_1$.

If $\kappa_1\neq\lambda_1$, then $\kappa_2=\rho_2,\gamma_2,\eta_2$, but then $\mathbf{\rho}=\mathbf{\eta}=\mathbf{\gamma}$, a contradiction.

If $\kappa_1=\lambda_1$ then $\kappa_2$ is obtained from $\lambda_2$ by adding two boxes, which may be done in at most two ways, so $\eta,\gamma,\rho$ are not all distinct, a contradiction. 

\noindent \textbf{Case 2:} $\rho_1=\lambda_1=\eta_1$ and $\gamma_2=\lambda_2$. 

If $\kappa_1=\lambda_1$ then since $\gamma_1\neq \lambda_1$, we see that $ \kappa_1\neq\gamma_1$. Thus, $\kappa_2=\gamma_2$. But then $(\kappa_1,\kappa_2)=(\lambda_1,\lambda_2)$, a contradiction.

Now if $\kappa_1\neq\lambda_1$, then $\kappa_2=\rho_2,\eta_2$, but then $\rho=\eta$, a contradiction.

\end{proof}
The following two lemmas provide a bound for $\dim A(\mathcal{H}_i^n\uparrow G;\mathcal{B})$, for $\mathcal{H}_i^n$ as in Figure \ref{BnH}.
\begin{lemma}\label{BnH1} \begin{itemize}
\item[]
\item[(1)] $\displaystyle\#\Hom(\mathcal{H}_i^n\uparrow G;\mathcal{B})=\frac{\vert B_{n-1}\vert}{\vert B_{i-1}\vert}\#\Hom(\mathcal{H}_i^i\uparrow G;\mathcal{B})$
\item[(2)] $\#\Hom(\mathcal{H}_i^i\uparrow G;\mathcal{B})=$
$$2(i-1)\vert B_{i-1}\vert+
\sum_{\substack{(\beta_{i-1}^1,\beta_{i-1}^2)=\\\mathbf{\beta}_{i-1}\in\mathcal{B}^{i-1}}} (\jmp(\beta_{i-1}^1)+\jmp(\beta_{i-1}^2))(\jmp(\beta_{i-1}^1)+\jmp(\beta_{i-1}^2)+1)d_{\beta_{i-1}}^2,$$
\end{itemize} where $\jmp$ denotes the jump of a partition, i.e, the number of ways to remove a single box to form a new partition.
\end{lemma}
\begin{proof}
To prove (1), first note by Theorem \ref{countingmorph}, $\# \Hom(\mathcal{H}_i^n\uparrow G;\mathcal{B})$ equals $$\begin{array}{lll}\displaystyle\sum_{\alpha_j,\beta_j\in\mathcal{B}^j}M_\mathcal{B}(\beta_{n-1},\beta_{i-1})M_\mathcal{B}(\beta_{i-1},\beta_{i-2})M_\mathcal{B}(\alpha_{i},\alpha_{i-1})M_\mathcal{B}(\alpha_{i},\beta_{i-1})M_\mathcal{B}(\alpha_{i-1},\beta_{i-2})d_{\alpha_{i-1}}d_{\beta_{n-1}}.\end{array}$$ By Corollary \ref{hom=dim}, 
$$\begin{array}{ll}\displaystyle\sum_{\beta_{n-1}\in\mathcal{B}^{n-1}} M_{\mathcal{B}}(\beta_{n-1},\beta_{i-1})d_{\beta_{n-1}}&= \langle D^{n-i}U^{n-1}\hat{0},\beta_{i-1}\rangle\\
&= \lambda_{n-1}\lambda_{n-2}\cdots\lambda_id_{\beta_{i-1}}\\
&=\displaystyle\frac{\vert B_{n-1}\vert}{\vert B_{i-1}\vert}d_{\beta_{i-1}}.\end{array}$$
Then $\# \Hom(\mathcal{H}_i^n\uparrow G;\mathcal{B})$ equals
$$\begin{array}{ll}\displaystyle\frac{\vert B_{n-1}\vert}{\vert B_{i-1}\vert}\sum_{\alpha_j,\beta_j\in\mathcal{B}^j}M_\mathcal{B}(\beta_{i-1},\beta_{i-2})M_\mathcal{B}(\alpha_{i},\alpha_{i-1})M_\mathcal{B}(\alpha_{i},\beta_{i-1})M_\mathcal{B}(\alpha_{i-1},\beta_{i-2})d_{\beta_{i-1}}d_{\alpha_{i-1}}\\
=\displaystyle\frac{\vert B_{n-1}\vert}{\vert B_{i-1}\vert}\#\Hom(\mathcal{H}_i^i\uparrow G;\mathcal{B}).\end{array}$$

To prove (2), $$\begin{array}{ll}\#\Hom(\mathcal{H}_i^i\uparrow G;\mathcal{B})&=\displaystyle\sum_{\alpha_j,\beta_j\in\mathcal{B}^j}M_\mathcal{B}(\beta_{i-1},\beta_{i-2})M_\mathcal{B}(\alpha_{i},\alpha_{i-1})M_\mathcal{B}(\alpha_{i},\beta_{i-1})M_\mathcal{B}(\alpha_{i-1},\beta_{i-2})d_{\beta_{i-1}}d_{\alpha_{i-1}}\\
&=\displaystyle\sum_{\alpha_{i-1}\neq\beta_{i-1}}+\sum_{\alpha_{i-1}=\beta_{i-1}},\end{array}$$
for $\displaystyle\sum_{\alpha_{i-1}\neq\beta_{i-1}}=\sum_{\substack{\alpha_j,\beta_j\in\mathcal{B}^j\\\alpha_{i-1}\neq\beta_{i-1}}}M_\mathcal{B}(\beta_{i-1},\beta_{i-2})M_\mathcal{B}(\alpha_{i},\alpha_{i-1})M_\mathcal{B}(\alpha_{i},\beta_{i-1})M_\mathcal{B}(\alpha_{i-1},\beta_{i-2})d_{\beta_{i-1}}d_{\alpha_{i-1}}$
and $\displaystyle\sum_{\alpha_{i-1}=\beta_{i-1}}=\sum_{\substack{\alpha_j,\beta_j\in\mathcal{B}^j\\\alpha_{i-1}=\beta_{i-1}}}M_\mathcal{B}(\beta_{i-1},\beta_{i-2})^2M_\mathcal{B}(\alpha_{i},\beta_{i-1})^2(d_{\beta_{i-1}})^2.$

First suppose $\mathbf{\alpha}_{i-1}=(\alpha_{i-1}^1,\alpha_{i-1}^2),\mathbf{\beta}_{i-1}=(\beta_{i-1}^1,\beta_{i-1}^2)$ are distinct pairs of partitions. Then they jointly determine $\mathbf{\alpha}_i=(\alpha_i^1,\alpha_i^2)$. Thus, the sum $\displaystyle\sum_{\alpha_{i-1}\neq\beta_{i-1}}$ becomes
$$\begin{array}{ll}&\displaystyle\sum_{\substack{\beta_{i-2}\in\mathcal{B}^{i-2}\\\alpha_{i-1}\neq\beta_{i-1}\in\mathcal{B}^{i-1}}}M_\mathcal{B}(\beta_{i-1},\beta_{i-2})M_\mathcal{B}(\alpha_{i-1},\beta_{i-2})d_{\alpha_{i-1}}d_{\beta_{i-1}}\\
\\
=&\displaystyle\sum_{\substack{\beta_{i-2}\in\mathcal{B}^{i-2}\\\alpha_{i-1},\beta_{i-1}\in\mathcal{B}^{i-1}}} M_\mathcal{B}(\beta_{i-1},\beta_{i-2})M_\mathcal{B}(\alpha_{i-1},\beta_{i-2})d_{\beta_{i-1}}d_{\alpha_{i-1}}-\displaystyle\sum_{\substack{\beta_{j}\in\mathcal{B}^{}\\\alpha_{i-1}=\beta_{i-1}}}M_\mathcal{B}(\beta_{i-1},\beta_{i-2})^2(d_{\beta_{i-1}})^2\\
\\
=&\displaystyle\frac{\vert B_{i-1}\vert}{\vert B_{i-2}\vert}\sum_{\beta_{j},\alpha_j\in\mathcal{B}^{j}}{M_\mathcal{B}(\alpha_{i-1},\beta_{i-2})d_{\beta_{i-2}}d_{\alpha_{i-1}}}-\sum_{\substack{\beta_{j}\in\mathcal{B}^{j}\\\alpha_{i-1}=\beta_{i-1}}}M_\mathcal{B}(\beta_{i-1},\beta_{i-2})^2(d_{\beta_{i-1}})^2\\
\\
=&\displaystyle\frac{\vert B_{i-1}\vert}{\vert B_{i-2}\vert}\sum_{\alpha_{i-1}\in\mathcal{B}^{i-1}}(d_{\alpha_{i-1}})^2-\sum_{\substack{\beta_{j}\in\mathcal{B}^{j}\\\alpha_{i-1}=\beta_{i-1}}}M_\mathcal{B}(\beta_{i-1},\beta_{i-2})^2(d_{\beta_{i-1}})^2\\
\\
=&\displaystyle\frac{{\vert B_{i-1}\vert}^2}{\vert B_{i-2}\vert}-\sum_{\beta_{i-1}\in\mathcal{B}^{i-1}}(\jmp(\beta_{i-1}^1)+\jmp(\beta_{i-1}^2))(d_{\beta_{i-1}})^2,\end{array}$$
and so
\begin{equation}\label{Bone}\begin{array}{l}\displaystyle\sum_{\alpha_{i-1}\neq\beta_{i-1}} =\displaystyle 2(i-1)\vert B_{i-1}\vert-\sum_{\beta_{i-1}\in\mathcal{B}^{i-1}}(\jmp(\beta_{i-1}^1)+\jmp(\beta_{i-1}^2))(d_{\beta_{i-1}})^2.\end{array} \end{equation}

Now suppose $\alpha_{i-1}=\beta_{i-1}$. Then $\alpha_{i}$ is obtained from $\beta_{i-1}$ by  adding a box to $\beta_{i-1}^1$ or $\beta_{i-1}^2$, while $\beta_{i-2}$ is obtained from $\beta_{i-1}$ by removing a box from $\beta_{i-1}^1$ or $\beta_{i-1}^2$. Thus,
\begin{equation}\label{Btwo}\sum_{\alpha_{i-1}=\beta_{i-1}}=\sum_{\beta_{i-1}\in\mathcal{B}^{i-1}} (\jmp(\beta_{i-1}^1)+\jmp(\beta_{i-2}^1)) (\jmp(\beta_{i-1}^1)+\jmp(\beta_{i-2}^1)+2)(d_{\beta_{i-1}})^2.\end{equation}

Summing equations (\ref{Bone}) and (\ref{Btwo}) gives (2).
\end{proof}
\begin{lemma}\label{BnH2} For any pair of partitions $(\beta_i^1,\beta_i^2)$ with $\vert \beta_i^1\vert+\vert\beta_i^2\vert=i$, $$(\jmp(\beta_{i}^1)+\jmp(\beta_{i}^2))(\jmp(\beta_{i}^1)+\jmp(\beta_{i}^2)+1)\leq 6i$$ 
\end{lemma}
\begin{proof} Let $k=\vert \beta_i^1\vert$, $l=\vert\beta_i^2\vert$, $a_k=\jmp(\beta_i^1)$, and $a_l=\jmp(\beta_i^2)$. Then $k+l=i$ and by \cite[Lemma 5.3]{maslen}, $$\begin{array}{ll} a_k(a_k+1)\leq 2k, & a_l(a_l+1)\leq 2l.\end{array}$$
Then  
$$\begin{array}{ll}(\jmp(\beta_{i}^1)+\jmp(\beta_{i}^2))(\jmp(\beta_{i}^1)+\jmp(\beta_{i}^2)+1)&= (a_k+a_l)(a_k+a_l+1)\\
&=a_k(a_k+1)+a_l(a_l+1)+2a_ka_l\\
&\leq 2k+2l+2(2i)\\
&\leq 6i.\end{array}$$

\end{proof}
Combining Lemmas \ref{BnH1} and \ref{BnH2} gives the following bound:
\begin{corollary}\label{BnH3} $\#\Hom(\mathcal{H}_i^n\uparrow G;\mathcal{B})\leq  \frac{4(i-1)}{n}\vert B_n\vert.$
\end{corollary}

\subsection{The Weyl Group $D_n$}
The Bratteli diagram $\mathcal{B}$ associated to the chain
$\mathbb{C}[D_n]> \mathbb{C}[D_{n-1}]>\cdots> \mathbb{C}$
 is similar to the Bratteli diagram associated to the Weyl group $B_n$ in that irreducible representations of $\mathbb{C}[D_i]$ are indexed by pairs of partitions, $(\lambda_1,\lambda_2)$ of $k$ and $l$, respectively, with $k+l=i$. However, if $\lambda_1\neq\lambda_2$, the irreducible representation indexed by $(\lambda_1,\lambda_2)$ is the same as that indexed by $(\lambda_2,\lambda_1)$. If $\lambda_1=\lambda_2=\lambda$ then two distinct irreducible representations are indexed by the pair $(\lambda,\lambda)$, and denoted by $(\lambda,\lambda)^+$ and $(\lambda,\lambda)^-$ \cite{seminormal} (see Figure \ref{DnBratt}). Note that this is a multiplicity-free diagram.
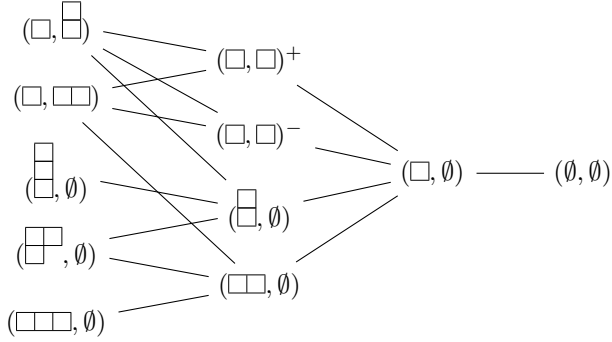
\begin{figure}[H]
\begin{center}
\begin{tikzpicture}[shorten >=1pt,node distance=2cm,on grid,auto,/tikz/initial text=] 
   \node at (-7,2) (35) {$(\tiny{\yng(1)},\tiny{\yng(1,1)})$};
   \node at (-7,1) (34) {$(\tiny{\yng(1)},\tiny{\yng(2)})$};
   \node at (-7,-2) (33) {$(\tiny{\yng(3)},\emptyset)$};
   \node at (-7,-1) (32) {$(\tiny{\yng(2,1)},\emptyset)$};
   \node at (-7,0) (31) {$(\tiny{\yng(1,1,1)},\emptyset)$};
   \node at (-4.3,1.5) (24) {$(\tiny{\yng(1)},\tiny{\yng(1)})^+$};
   \node at (-4.3,0.5) (23) {$(\tiny{\yng(1)},\tiny{\yng(1)})^-$};
   \node at (-4.3,-1.5) (22) {$(\tiny{\yng(2)},\emptyset)$};
   \node at (-4.3,-.5) (21) {$(\tiny{\yng(1,1)},\emptyset$)};
   \node at (-2,0) (1) {$(\tiny{\yng(1)},\emptyset)$};
   \node at (0,0) (0) {$(\emptyset,\emptyset)$};
   \path[every node/.style={font=\scriptsize}]
    (0) edge  node {} (1)
    (1) edge node {} (21)
    (1) edge node {} (22)
    (1) edge node {} (23)
    (1) edge node {} (24)
    (24) edge node {} (35)
    (24) edge node {} (34)
    (23) edge node {} (35)
    (23) edge node {} (34)
    (22) edge node {} (34)
    (22) edge node {} (33)
    (22) edge node {} (32)
    (21) edge node {} (35)
    (21) edge node {} (32)
    (21) edge node {} (31);
\end{tikzpicture}
\caption{Bratteli diagram for $D_3$}
\label{DnBratt}
\end{center}\end{figure}
\begin{lemma}\label{DnJlemma}
$M_\mathcal{B}(D_i,D_{i-2})\leq 3$.
\end{lemma}
\begin{proof}
Suppose not. Then since $\mathcal{B}$ is multiplicity-free, there exist pairs of partitions $\mathbf{\kappa}=(\kappa_1,\kappa_2),$ $\mathbf{\rho}=(\rho_1,\rho_2),$ $\mathbf{\gamma}=(\gamma_1,\gamma_2),$ $\mathbf{\eta}=(\eta_1,\eta_2),$ $\mathbf{\mu}=(\mu_1,\mu_2),$ and $\mathbf{\lambda}=(\lambda_1,\lambda_2)$ connected in $\mathcal{B}$ as in Figure \ref{Dnpart}.

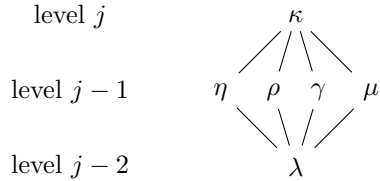
\begin{figure}[H]\begin{center}
\begin{tikzpicture}[shorten >=1pt,node distance=2cm,on grid,auto,/tikz/initial text=] 
   \node at (-3,0) (1) {level $j$};
   \node at (-3,-1) (1) {level $j-1$};
   \node at (-3,-2) (1) {level $j-2$};
   \node at (0,0) (1) {$\mathbf{\kappa}$};
   \node at (.3,-1) (22) {$\mathbf{\gamma}$};
   \node at (-.3,-1) (21) {$\mathbf{\rho}$};
   \node at (-1,-1) (23) {$\mathbf{\eta}$};
   \node at (1,-1) (24) {$\mathbf{\mu}$};
   \node at (0,-2) (3) {$\mathbf{\lambda}$};
   \path[every node/.style={font=\scriptsize}]
    (1) edge  node {} (22)
    (1) edge  node {} (21)
    (1) edge node {} (23)
    (1) edge node {} (24)
    (21) edge node {} (3)
    (23) edge node {} (3)
    (22) edge node {} (3)
    (24) edge node {} (3);               

\end{tikzpicture}
\caption{Subquiver of $\mathcal{B}$ if $M_\mathcal{B}(D_i,D_{i-2})> 3$.
}
\label{Dnpart}
\end{center}\end{figure}

However, the proof of Lemma \ref{BnJlemma} dictates that no three of $\eta,\mu,\gamma,\rho$ are distinct pairs of partitions. Thus, without loss of generality, 
$$\begin{array}{llll}(\eta_1,\eta_2)=(\alpha,\alpha)^+,& (\mu_1,\mu_2)=(\alpha,\alpha)^-,& (\gamma_1,\gamma_2)=(\beta,\beta)^+,&(\rho_1,\rho_2)=(\beta,\beta)^-,\end{array}$$ for $\alpha, \beta$ distinct partitions of $\frac{j-1}{2}$. Then as in the proof of Lemma \ref{BnJlemma}, either $\lambda_1=\eta_1=\alpha$ or $\lambda_2=\eta_2=\alpha$. Without loss, suppose $\lambda_1=\alpha$. Then since $\alpha\neq\beta$, $\lambda_2$ must be $\beta$. However, $\vert \lambda_1\vert +\vert\lambda_2\vert=\vert \alpha\vert+\vert\beta\vert>j-2$, a contradiction.
\end{proof}
Lemma \ref{DnJlemma} is used in the proof of Theorem \ref{Dnthm} to give a bound on $\dim A(\mathcal{J}_i^n\uparrow G;\mathcal{B})$, for $\mathcal{J}_i^n$ as in Figure \ref{BnH}. The following two lemmas provide a bound for $\dim A(\mathcal{H}_i^n\uparrow G;\mathcal{B})$, for $\mathcal{H}_i^n$ as in Figure \ref{BnH}.
\begin{lemma}\label{DnH1}
\begin{itemize}
\item[]
\item[(1)] $\displaystyle\#\Hom(\mathcal{H}_i^n\uparrow G;\mathcal{B})=\frac{\vert D_{n-1}\vert}{\vert D_{i-1}\vert}\#\Hom(\mathcal{H}_i^i\uparrow G;\mathcal{B}),$
\item[(2)] for $i$ odd, $\#\Hom(\mathcal{H}_i^i\uparrow G;\mathcal{B})$ is at most $$\begin{array}{ll}
\displaystyle\frac{\vert D_{i-1}\vert^2}{\vert D_{i-2}\vert}+\sum_{(\alpha,\alpha)^\pm\in \mathcal{B}^{i-1}} (\jmp(\alpha))(\jmp(\alpha)+1)(d_{(\alpha,\alpha)^+})^2+\\
\displaystyle\sum_{\substack{(\beta_{i-1}^1,\beta_{i-1}^2)=\\\mathbf{\beta}_{i-1}\in\mathcal{B}^{i-1}}} (\jmp(\beta_{i-1}^1)+\jmp(\beta_{i-1}^2))(\jmp(\beta_{i-1}^1)+\jmp(\beta_{i-1}^2)+1)d_{\beta_{i-1}}^2,\end{array}$$
\item[(3)] for $i$ even, $\#\Hom(\mathcal{H}_i^i\uparrow G;\mathcal{B})$ is at most
$$\begin{array}{l}\displaystyle 
2\displaystyle\frac{\vert D_{i-1}\vert^2}{\vert D_{i-2}\vert}+2\displaystyle\sum_{\substack{(\beta_{i-1}^1,\beta_{i-1}^2)=\\\mathbf{\beta}_{i-1}\in\mathcal{B}^{i-1}}} (\jmp(\beta_{i-1}^1)+\jmp(\beta_{i-1}^2))(2\jmp(\beta_{i-1}^1)+2\jmp(\beta_{i-1}^2)+3)d_{\beta_{i-1}}^2,\end{array}$$
where $\jmp$ denotes the jump of a partition, i.e., the number of ways to remove a single box to form a new partition.
\end{itemize} \end{lemma}
\begin{proof}
Part (1) follows from the proof of Lemma \ref{BnH1}.

To prove (2), consider note that $\#\Hom(\mathcal{H}_i^i\uparrow G;\mathcal{B})$ equals 
$$\begin{array}{l}=\displaystyle\sum_{\alpha_j,\beta_j\in\mathcal{B}^j}M_\mathcal{B}(\beta_{i-1},\beta_{i-2})M_\mathcal{B}(\alpha_{i},\alpha_{i-1})M_\mathcal{B}(\alpha_{i},\beta_{i-1})M_\mathcal{B}(\alpha_{i-1},\beta_{i-2})d_{\beta_{i-1}}d_{\alpha_{i-1}}\\
=\displaystyle\sum_{\alpha_{i-1}\neq\beta_{i-1}}+\sum_{\substack{\alpha_{i-1}\neq\beta_{i-1}\\\alpha_{i-1}=(\alpha,\alpha)^\pm=\beta_{i-1}}}+\sum_{\alpha_{i-1}=\beta_{i-1}},\end{array}$$

For
$$\begin{array}{l}\displaystyle\sum_{\alpha_{i-1}\neq\beta_{i-1}}:=\displaystyle\sum_{\substack{\alpha_j,\beta_j\in\mathcal{B}^j\\\alpha_{i-1}\neq\beta_{i-1}}}M_\mathcal{B}(\beta_{i-1},\beta_{i-2})M_\mathcal{B}(\alpha_{i},\alpha_{i-1})M_\mathcal{B}(\alpha_{i},\beta_{i-1})M_\mathcal{B}(\alpha_{i-1},\beta_{i-2})d_{\beta_{i-1}}d_{\alpha_{i-1}},\end{array}$$
 over partitions $\alpha_{i-1}\neq \beta_{i-1}$ such that if $\alpha_{i-1}=(\alpha,\alpha)^\pm$ then $\beta_{i-1}\neq (\alpha,\alpha)^\pm$,
$$\begin{array}{l}\displaystyle\sum_{\substack{\alpha_{i-1}\neq\beta_{i-1}\\\alpha_{i-1}:=(\alpha,\alpha)^\pm\\=\beta_{i-1}}}:=\displaystyle\sum_{\substack{\alpha_j,\beta_j\in\mathcal{B}^j\\\alpha_{i-1}\neq\beta_{i-1}\\\alpha_{i-1}=(\alpha,\alpha)^\pm=\beta_{i-1}}}M_\mathcal{B}(\beta_{i-1},\beta_{i-2})M_\mathcal{B}(\alpha_{i},\alpha_{i-1})M_\mathcal{B}(\alpha_{i},\beta_{i-1})M_\mathcal{B}(\alpha_{i-1},\beta_{i-2})d_{\beta_{i-1}}d_{\alpha_{i-1}},\end{array}$$
and
$$\begin{array}{l}\displaystyle\sum_{\alpha_{i-1}:=\beta_{i-1}}:=\displaystyle\sum_{\substack{\alpha_j,\beta_j\in\mathcal{B}^j\\\alpha_{i-1}=\beta_{i-1}}}M_\mathcal{B}(\beta_{i-1},\beta_{i-2})^2M_\mathcal{B}(\alpha_{i},\beta_{i-1})^2(d_{\beta_{i-1}})^2.\end{array}$$

As in the proof of Lemma \ref{BnH1},
\begin{equation}\label{Done}\sum_{\alpha_{i-1}\neq\beta_{i-1}}\leq\frac{\vert D_{i-1}\vert^2}{\vert D_{i-2}\vert}-\sum_{\beta_{i-1}\in\mathcal{B}^{i-1}}(\jmp(\beta_{i-1}^1)+\jmp(\beta_{i-1}^2))(d_{\beta_{i-1}})^2,\end{equation}
the inequality appearing because if $\beta_{i-1}=(\alpha, \alpha)$, $\jmp(\alpha)+\jmp(\alpha)$ is an overestimate since $(\alpha,\beta)$ represents the same representation as $(\beta, \alpha)$ in $\mathcal{B}$.
Similarly, the proof of Lemma \ref{BnH1} gives
\begin{equation}\label{Dtwo}\sum_{\alpha_{i-1}=\beta_{i-1}}\leq\sum_{\beta_{i-1}\in\mathcal{B}^{i-1}} (\jmp(\beta_{i-1}^1)+\jmp(\beta_{i-1}^2)) (\jmp(\beta_{i-1}^1)+\jmp(\beta_{i-1}^2)+2)(d_{\beta_{i-1}})^2.\end{equation}

Now suppose $\alpha_{i-1}\neq\beta_{i-1}$ and $\alpha_{i-1}=(\alpha,\alpha)^\pm=\beta_{i-1}$. Then $$\begin{array}{ll}\displaystyle\sum_{\substack{\alpha_{i-1}\neq\beta_{i-1}\\\alpha_{i-1}=(\alpha,\alpha)^\pm=\beta_{i-1}}}=&\displaystyle\sum_{\substack{\alpha_j,\beta_j\in\mathcal{B}^j\\(\alpha,\alpha)^\pm}}M_\mathcal{B}((\alpha,\alpha)^\pm,\beta_{i-2})^2M_\mathcal{B}(\alpha_{i},(\alpha,\alpha)^\pm)^2(d_{(\alpha,\alpha)^\pm})^2\end{array}$$
\begin{equation}\label{Dthree}
\leq\displaystyle\sum_{(\alpha,\alpha)^\pm\in\mathcal{B}^{i-1}} \jmp(\alpha)(\jmp(\alpha)+1)(d_{(\alpha,\alpha)^\pm})^2.\end{equation}

Summing equations (\ref{Done}), (\ref{Dtwo}), and (\ref{Dthree}) gives part (2).

To prove (3), note that in this case $$\#\Hom(\mathcal{H}_i^i\uparrow G;\mathcal{B})=\sum_{\alpha_{i-1}\neq\beta_{i-1}}+\sum_{\alpha_{i-1}=\beta_{i-1}},$$
since $i-1$ is odd so $(\alpha, \alpha)^\pm\notin \mathcal{B}^{i-1}$. However, pairs of partitions of this form may be found at levels $i$ and $i-2$.

First suppose $\alpha_{i-1}\neq \beta_{i-1}$. Then as in the proof of Lemma \ref{BnH1} they jointly determine $\alpha_i=(\alpha_i^1,\alpha_i^2)$. This means that they jointly determine at most two pairs of partitions (if $\alpha_i^1=\alpha_i^2$). Thus
\begin{equation}\label{Done1}\sum_{\alpha_{i-1}\neq\beta_{i-1}}\leq 2\sum_{\substack{\beta_{i-2}\in\mathcal{B}^{i-2}\\\alpha_{i-1}\neq\beta_{i-1}\in\mathcal{B}^{i-1}}}M_\mathcal{B}(\beta_{i-1},\beta_{i-2})M_\mathcal{B}(\alpha_{i-1},\beta_{i-2})d_{\alpha_{i-1}}d_{\beta_{i-1}}$$
$$=2\frac{\vert D_{i-1}\vert^2}{\vert D_{i-2}\vert}-2\sum_{\beta_{i-1}\in\mathcal{B}^{i-1}}(\jmp(\beta_{i-1}^1)+\jmp(\beta_{i-1}^2))(d_{\beta_{i-1}})^2,\end{equation}
as in the proof of Lemma \ref{BnH1}.

Now suppose $\alpha_{i-1}=\beta_{i-1}$. As before there are $\jmp(\beta_{i-1}^1)$ ways to obtain $\beta_i^1$ and $\jmp(\beta_{i-1}^2)$ ways to obtain $\beta_i^2$, but to account for when $\beta_i^1=\beta_i^2$, we overcount by multiplying by 2. The same holds for the number of ways to obtain $\alpha_{i-2}$ from $\beta_{i-1}$. Thus,
\begin{equation}\label{Dtwo1}\sum_{\alpha_{i-1}\beta_{i-1}}\leq \sum_{\beta_{i-1}\in\mathcal{B}^{i-1}} 2(\jmp(\beta_{i-1}^1)+\jmp(\beta_{i-1}^2)) 2(\jmp(\beta_{i-1}^1)+\jmp(\beta_{i-1}^2)+2)(d_{\beta_{i-1}})^2.\end{equation}

Summing equations (\ref{Done1}) and (\ref{Dtwo1}) gives part 3.
\end{proof}
Combining Lemma \ref{DnH1} with Lemma \ref{BnH2} gives the following bound:
\begin{corollary}\label{DnH3} 
$\#\Hom(\mathcal{H}_i^n\uparrow G;\mathcal{B})\leq \frac{20(i-1)}{n}\vert D_n\vert$.
\end{corollary}
\end{section}
\section{The General Linear Group}\label{appgenlin}
The SOV approach reduces Theorem \ref{Gln} to counting the number of morphisms of the quivers of Figure \ref{GlnFGH} into the Bratteli diagram of $Gl_n(q)$. In this section we use known results on the number of conjugacy classes and the multiplicities of representations of $Gl_n(q)$ to provide the bounds used in the proof of Theorem \ref{Gln}.

\begin{theorem} For  $\mathcal{H}_j^n$ the quiver of Figure \ref{GlnFGH} and $\mathcal{B}$ the Bratteli diagram for the subgroup chain $Gl_n(q)>Gl_{n-1}(q)>\cdots>\{e\}$, $$\begin{array}{ll}\dim A(\mathcal{H}_j^n\uparrow \mathcal{Q};\mathcal{B})\leq 2^{2j-4}q^{j-2}\displaystyle\frac{q^{j-1}(q^j-1)}{q^{n-1}(q^n-1)}\vert Gl_{n}(q)\vert.\end{array}$$
\end{theorem}
\begin{proof} By Theorem \ref{countingmorph},  we see that 
$\# \Hom(\mathcal{H}_j^n\uparrow \mathcal{Q};\mathcal{B})$ is equal to the sum
$$\begin{array}{l}\displaystyle\sum_{\alpha_i,\beta_i\in\mathcal{B}^i}M_\mathcal{B}(\beta_{n-1},\beta_{j-1})M_\mathcal{B}(\beta_{j-1},\beta_{j-2})M_\mathcal{B}(\alpha_{j},\beta_{j-1})M_\mathcal{B}(\alpha_{j},\alpha_{j-1}) M_\mathcal{B}(\alpha_{j-1},\beta_{j-2})d_{\alpha_{j-1}}d_{\beta_{n-1}}\\
\\
\leq M_{\mathcal{B}}(Gl_{j-1},Gl_{j-2})^2\vert \hat{Gl}_{j-2}(q)\vert\displaystyle\sum_{\alpha_i,\beta_i\in\mathcal{B}^i}M_\mathcal{B}(\beta_{n-1},\beta_{j-1})M_\mathcal{B}(\alpha_{j},\beta_{j-1})M_\mathcal{B}(\alpha_{j},\alpha_{j-1})d_{\alpha_{j-1}}d_{\beta_{n-1}},\end{array}$$ for $M_\mathcal{B}(G_i,G_j)=M_\mathcal{B}(G_i(q),G_j(q)):=\max M_\mathcal{B}(\alpha_i,\alpha_j)$ over all $\alpha_i\in\mathcal{B}^i, \alpha_j\in\mathcal{B}^j$ and $\vert \hat{G}_i(q)\vert$ the number of conjugacy classes of $G_i(q).$
By Corollary \ref{hom=dim},
$$\begin{array}{ll}\# \Hom(\mathcal{H}_j^n\uparrow \mathcal{Q};\mathcal{B})&\leq M_{\mathcal{B}}(Gl_{j-1}(q),Gl_{j-2}(q))^2\vert \hat{Gl}_{j-2}(q)\vert \langle D^{n-1}U^{n-j}DU^j\hat{0},\hat{0}\rangle\\
\\
&=M_{\mathcal{B}}(Gl_{j-1}(q),Gl_{j-2}(q))^2\vert \hat{Gl}_{j-2}(q)\vert \lambda_{j}\lambda_{n-1}\lambda_{n-2}\cdots\lambda_{1}\\
\\
&=M_{\mathcal{B}}(Gl_{j-1}(q),Gl_{j-2}(q))^2\vert \hat{Gl}_{j-2}(q)\vert\displaystyle \frac{\vert Gl_{j}(q)\vert}{\vert Gl_{j-1}(q)\vert}\vert Gl_{n-1}(q)\vert.\end{array}$$
By \cite[Lemma 5.9]{sovi}, $M(Gl_j(q),Gl_{j-1}(q))\leq 2^{j-1}$ and $\vert \hat{Gl}_{j}(q)\vert \leq q^j$. Thus, since $\dim A(\mathcal{H}_j^n\uparrow \mathcal{Q};\mathcal{B})=\# \Hom(\mathcal{H}_j^n\uparrow \mathcal{Q};\mathcal{B})$,
$$\begin{array}{ll}\dim A(\mathcal{H}_j^n\uparrow \mathcal{Q};\mathcal{B})&\leq 2^{2j-4}q^{j-2}q^{j-1}(q^j-1)\vert Gl_{n-1}(q)\vert\\
\\
&=2^{2j-4}q^{j-2}\displaystyle\frac{q^{j-1}(q^j-1)}{q^{n-1}(q^n-1)}\vert Gl_{n}(q)\vert\end{array}$$
\end{proof}

\begin{theorem} For  $\mathcal{J}_j^n$ the quiver of Figure \ref{GlnFGH},  $$\begin{array}{ll}\dim A(\mathcal{H}_j^n\uparrow \mathcal{Q};\mathcal{B})\leq 2^{2j-4}q^{j-2}\displaystyle\frac{q^{j-1}(q^j-1)}{q^{n-1}(q^n-1)}\vert Gl_{n}(q)\vert.\end{array}$$
\end{theorem}
\begin{proof}
$$\dim A(\mathcal{J}_j\uparrow \mathcal{Q};\mathcal{B})=\sum_{\alpha_i,\beta_i\in\mathcal{B}^i}M_\mathcal{B}(\beta_{n-1},\beta_{j-1})M_\mathcal{B}(\alpha_{j},\beta_{j-1})M_\mathcal{B}(\alpha_{j},\alpha_{j-2})^2d_{\alpha_{j-2}}d_{\beta_{n-1}}$$
$$\leq M_{\mathcal{B}}(Gl_{j}(q),Gl_{j-2}(q))\sum_{\alpha_i,\beta_i\in\mathcal{B}^i}M_\mathcal{B}(\beta_{n-1},\beta_{j-1})M_\mathcal{B}(\alpha_{j},\beta_{j-1})M_\mathcal{B}(\alpha_{j},\alpha_{j-2})d_{\alpha_{j-2}}d_{\beta_{n-1}}$$ 
$$\begin{array}{l}\displaystyle=M_{\mathcal{B}}(Gl_{j}(q),Gl_{j-2}(q)) \langle D^{n-1}U^{n-j}DU^j\hat{0},\hat{0}\rangle\\
\displaystyle=M_{\mathcal{B}}(Gl_{j}(q),Gl_{j-2}(q))\frac{\vert Gl_{j}(q)\vert}{\vert Gl_{j-1}(q)\vert}\vert Gl_{n-1}(q)\vert.\end{array}$$
By \cite[Lemma 5.9]{sovi}, $M(Gl_{j}(q),Gl_{j-2}(q))\leq 2^{2j-3}q^{j-1}$. Thus,
$$\dim A(\mathcal{J}_j\uparrow \mathcal{Q};\mathcal{B})\leq 2^{2j-3}q^{j-1}\frac{q^{j-1}(q^j-1)}{q^{n-1}(q^n-1)}\vert Gl_n(q)\vert.$$
\end{proof}

\subsection{Factoring Coset Representatives of $GL_n(\mathbb{F}_q)$}\label{appD}
In this section we provide the set of coset representatives and their factorizations used in the proof of Theorem \ref{Gln} by developing a correspondence between $Gl_n(q)/Gl_{n-1}(q)$ and the set $Z_n=\{\mathbf{z}=(x_1y_1,\dots, x_ny_n)\mid \mathbf{x},\mathbf{y}\in (\mathbb{F}_q)^n,\;y\cdot x=1\}.$

Define an action of $Gl_n(q)$ on $Z_n$ via $A.\mathbf{z}=(\tilde{x}_1\tilde{y}_1,\dots,\tilde{x}_n\tilde{y}_n),$ for $\tilde{\mathbf{y}}=\mathbf{y}A^{-1},\tilde{\mathbf{x}}=(A\mathbf{x}^T)^T $.  Note that the action of $A$ preserves $y\cdot x$. For $\mathbf{1}=(0,\dots,0,1)$, we show (Theorem \ref{factor}) that
$$Z_n=\orb(\mathbf{1}).$$
Note that $Gl_{n-1}(q)$, viewed as a subgroup of $Gl_n(q)$, stabilizes $\mathbf{1},$ so the orbit-stabilizer theorem gives a bijection between
$Z_n=\orb(\mathbf{1})$ and $Gl_n(q)/Gl_{n-1}(q)$ through the correspondence
\begin{equation}\label{correspondence} g.\mathbf{1}\longleftrightarrow gGl_{n-1}(q).\end{equation}
Thus, writing $\mathbf{z}=A_1\cdots A_m.\mathbf{1}$ for each $\mathbf{z}\in Z_n$ gives a factorization of the corresponding coset representative. We find a factorization in which each matrix $A_i=A\bigoplus I_{n-2}$ for $A\in Gl_2(q)$.
\begin{lemma}\label{2x2} Suppose $\mathbf{z}=(x_1y_1,x_2y_2)\in Z_2$ with $x_1y_1+x_2y_2\neq 0$. Then there exists a matrix  $A\in Gl_2(q)$ and $y_2',x_2'\in\mathbb{F}_q^\times$ such that $$A.\mathbf{z}=(0,x_2'y_2').$$
\end{lemma}

\begin{proof} 
\begin{itemize}
\item[]
\item[\textbf{Case 1:}] $x_1=0$. Let $A=\begin{pmatrix}
1&0\\\frac{y_1}{y_2}&1
\end{pmatrix}$. Note that for all possible choices of $\mathbf{z}\in Z_2$, there are $q$ possibilities for $A$.
\item[\textbf{Case 2:}] $x_1\neq0$, $y_1=0$. Let $A=\begin{pmatrix}
1&\frac{-x_1}{x_2}\\0&1
\end{pmatrix}$. Note there are $q-1$ possibilities for $A$.
\item[\textbf{Case 3:}] $x_1\neq0$, $y_1\neq0$. Let $A=\begin{pmatrix}
\frac{-x_2}{x_1}&1\\1&\frac{y_2}{y_1}
\end{pmatrix}$. Note there are $q^2$ possibilities for $A$. Note further that for $z_1:=x_1y_1$ and $z_2:=x_2y_2$ fixed and nonzero,  $$A=\begin{pmatrix}
\frac{-x_2}{x_1}&1\\1&\frac{z_2}{z_1}\frac{x_1}{x_2}
\end{pmatrix},$$ and there are $q-1$ possibilities for $A$.
\end{itemize}
 \end{proof}

We use Lemma \ref{2x2} to systematically write $\mathbf{z}\in Z_n$ in form $$\mathbf{z}=A_1A_2\cdots A_k.\mathbf{1},$$ with $A_i\in Gl_n(q)$. Recall that $p$ is the characteristic of $\mathbb{F}_q$.
\begin{proposition}\label{easier} Let $\tilde{\mathbf{z}}\in Z_n$. Then there is a permutation matrix $\pi\in Gl_n(q)$, $b\in \mathbb{F}_q^\times$, and $i\geq 1$ such that $\pi.\tilde{\mathbf{z}}=\mathbf{z}$ with:

\begin{itemize}
\item[(i)] $z_1+\dots+z_j\neq 0$ for all $i\leq j\leq n$,
\item[(ii)] $z_1=\dots=z_i=b$,
\item[(iii)] $p\vert(i-1)$.
\end{itemize}
\end{proposition}
\begin{proof}
Let $\tilde{\mathbf{z}}=(\tilde{z}_1,\dots,\tilde{z}_n)\in Z_n$. Note that $\tilde{z}_1+\cdots+\tilde{z}_n=1\neq 0$. Let $j$ be an index (if it exists) such that $\tilde{z}_1+\cdots\tilde{z}_n-\tilde{z}_j\neq 0$. Note that for a permutation matrix $\pi$,  $$\pi.\tilde{\mathbf{z}}=(\tilde{z}_{\pi(1)},\dots, \tilde{z}_{\pi(n)}).$$ Permute $\tilde{\mathbf{z}}$ to make $\tilde{z}_j$ the last entry, then delete $\tilde{z}_j$ to produce a vector of length $n-1$. Repeat until no such index exists, and let $i$ be the length of the resultant vector, $\mathbf{z}$. Then clearly $z_1+\cdots+z_j\neq 0$ for all $i\leq j\leq n$. Further, $z_1+\cdots+z_i-z_k=0$ for all $1\leq k\leq i$; in particular, $z_1=\cdots=z_i=b\in \mathbb{F}_q^\times$. Finally note that $z_1+\cdots+z_{i-1}=(i-1)b=0$ and so $p\vert (i-1)$.
\end{proof}
In light of Proposition \ref{easier}, let $$S_i(n)=\{\mathbf{z}\in Z_n\vert\;\mathbf{z}\; \text{satisfies (i) and (ii) of Proposition \ref{easier}}\}.$$ 
\begin{theorem}\label{factor} For $p\neq 2$ and $\mathbf{z}\in S_i(n)$, there exist invertible matrices 
$u_j,u_j', v_j, t_j\in Gl_j(q)\cap\cent(Gl_{j-2}(q)),$ such that 
$$v_n\cdots v_{i+1}u_{i}\cdots u_{2p+1}t_{2p}u_{2p+1}'u_{2p-1}\cdots u_{p+1}t_p(u_{p+1}')u_{p-1}\cdots u_2.\mathbf{z}=\mathbf{1}.$$
\end{theorem}
\begin{proof}
Let $\mathbf{z}\in S_i(n)$ and let $i>p$. Note that $\mathbf{z}=\begin{pmatrix}
b,\dots,b,z_{i+1}, \dots,z_n
\end{pmatrix},$ and since $z_1+z_2=2b\neq 0$, by Lemma \ref{2x2}, there is a matrix $ A\in Gl_2(q)$ such that $A.(z_1,z_2)=\begin{pmatrix}0,x_2'y_2'\end{pmatrix}$ with $y_2'x_2'=2b$. Let $u_2= A\bigoplus I_{n-2}\in Gl_n(q).$ Then $$u_2.\mathbf{z}=\begin{pmatrix}0,2b,b,\dots,b,z_{i+1},\dots,z_n\end{pmatrix}.$$
Repeat this process, defining matrices $u_3,u_4,\dots ,u_{p-1}$ (i.e., find the matrix $A$ guaranteed by Lemma \ref{2x2}, and let $u_j=I_{j-2}\bigoplus A\bigoplus I_{n-j}$). Note that 
$$u_{p-1}\cdots u_3u_2.\mathbf{z}=\begin{pmatrix}
0,\dots,0,(p-1)b,b,b,b,\dots,b,z_{i+1},\dots,z_n
\end{pmatrix}.$$
Since $z_{p-1}+z_p=pb=0$, we cannot use Lemma \ref{2x2}. Instead, define $(u_{p+1}')$ as above and let $t_p$ be the permutation matrix of the transposition $(p-1\; p)$. Then
$$t_pu_{p+1}'u_{p-1}\cdots u_2.\mathbf{z}=\begin{pmatrix}
0,\dots,0,(p-1)b,2b,b,b,\dots,b,z_{i+1},\dots,z_n
\end{pmatrix},$$
and since now $z_{p-1}+z_p\neq 0$, define $u_{p+1}$ as before so that 
$$u_{p+1}t_pu_{p+1}'u_{p-1}\cdots u_2.\mathbf{z}
=\begin{pmatrix}
0,\dots,0,0,(p+1)b=b,b,b,\dots,b,z_{i+1},\dots,z_n
\end{pmatrix}$$
Repeat this process through definition of the matrix $u_i$, so that 
$$u_i\cdots u_2.\mathbf{z}=\begin{pmatrix}
0,\dots,0,z_1+\cdots+z_i,z_{i+1},\dots,z_n
\end{pmatrix}.$$
Since $z_1+\cdots+z_j\neq 0$ for all $i\leq j\leq n$, we use Lemma \ref{2x2}  to find the appropriate 2x2 matrix $A_j$ so that for $v_j=I_{j-2}\bigoplus A_j\bigoplus I_{n-j}$,
$$v_n\cdots v_{i+1}u_i\cdots u_2.\mathbf{z}=
\begin{pmatrix}
0,\dots,0,z_1+\cdots+z_n
\end{pmatrix}=\begin{pmatrix}
0,\dots,0,1
\end{pmatrix}.$$

For $i<p$ analogous arguments apply without needing the matrices $t_p$. 

\end{proof}
\begin{remark} By Lemma \ref{2x2}, there are $(q-1)$ possibilities for each $u_j$ and $q^2$ possibilities for each $v_j$. \end{remark} 
By Proposition \ref{easier}, 
$$X_n=\bigcup_{\pi\in S_n}\bigcup_{\substack{ 1\leq i\leq n \\ p\vert (i-1)}}\pi S_i(n)$$ 
and so by Expression \ref{correspondence} 
 a complete set of coset representatives for \\$Gl_n(q)/Gl_{n-1}(q)$ is contained in $\{\pi s_i\vert\;1\leq i\leq n, p\mid (i-1), s_i\in S_i(n)\},$ with each $s_i$ of form:
$$s_i=u_2\cdots u_{p-1}u_{p+1}'t_pu_{p+1}\cdots u_iv_{i+1}\cdots v_n.$$

Finally, we note that similar results hold in the $p=2$ case.
\begin{theorem}\label{p=2}
For $p=2$, $i\geq 3$ odd, $(\mathbf{y},\mathbf{x})\in S_i(n)$, there exist invertible matrices 
$$a_j,b_j,c_j,v_j\in Gl_j(q)\cap\cent(Gl_{j-2}(q))$$
such that
$$\mathbf{z}=a_3b_2c_3\cdots a_ib_{i-1}c_iv_{i+1}\cdots v_n.\mathbf{1}.$$

\end{theorem}
%
%
%

Note that there are $(q-1)$ choices for $a_j$ and $b_j$, that $c_j$ is completely determined by $a_j$ and $b_j$, and that there are $q^2$ choices for $v_j$.

\bibliographystyle{abbrv}
\bibliography{SOVIIbib}

\def\cprime{$'$}
\begin{thebibliography}{10}

\bibitem{auslander}
L.~Auslander and R.~Tolimieri.
\newblock Is computing with the finite {F}ourier transform pure or applied
  mathematics?
\newblock {\em Bull. Amer. Math. Soc. (N.S.)}, 1(6):847--897, 1979.

\bibitem{barros}
D.~Barros, S.~Wilson, and J.~Kahn.
\newblock Comparison of orthogonal frequency-division multiplexing and
  pulse-amplitude modulation in indoor optical wireless links.
\newblock {\em IEEE Trans. Commun.}, 60(1):153--163, January 2012.

\bibitem{bratteli}
O.~Bratteli.
\newblock Inductive limits of finite dimensional {$C^*$}-algebras.
\newblock {\em Trans. Amer. Math. Soc. .}, 171:195--234, 1972.

\bibitem{algebraiccomplexity}
P.~B{\"u}rgisser, M.~Clausen, and M.~Shokrollahi.
\newblock {\em Algebraic {C}omplexity {T}heory}, volume 315 of {\em Grundlehren
  der Mathematischen Wissenschaften [Fundamental Principles of Mathematical
  Sciences]}.
\newblock Springer-Verlag, Berlin, 1997.
\newblock With the collaboration of Thomas Lickteig.

\bibitem{astrophysics}
K.~Cannon, R.~Cariou, A.~Chapman, M.~Crispin-Ortuzar, N.~Fotopoulos, M.~Frei,
  C.~Hanna, E.~Kara, D.~Keppel, L.~Liao, S.~Privitera, A.~Searle, L.~Singer,
  and A.~Weinstein.
\newblock Toward early-warning detection of gravitational waves from compact
  binary coalescence.
\newblock {\em The Astrophysical Journal}, 748(2):136, 2012.

\bibitem{clausen}
M.~Clausen.
\newblock Fast generalized {F}ourier transforms.
\newblock {\em Theoret. Comput. Sci.}, 67(1):55--63, 1989.

\bibitem{cooleytukey}
J.~Cooley and J.~Tukey.
\newblock An algorithm for the machine calculation of complex {F}ourier series.
\newblock {\em Math. Comp.}, 19:297--301, 1965.

\bibitem{blindimage}
A.~Danelakis, M.~Mitrouli, and D.~Triantafyllou.
\newblock Blind image deconvolution using a banded matrix method.
\newblock {\em Numer. Algorithms}, 64(1):43--72, 2013.

\bibitem{orell}
Z.~Daugherty and R.~Orellana.
\newblock The quasi-partition algebra.
\newblock {\em J. Algebra}, 404:124--151, 2014.

\bibitem{diaconis-fft}
P.~Diaconis.
\newblock Average running time of the fast {F}ourier transform.
\newblock {\em J. Algorithms}, 1:187--208, 1980.

\bibitem{diaconis}
P.~Diaconis.
\newblock {\em Group {R}epresentations in {P}robability and {S}tatistics}.
\newblock Institute of Mathematical Statistics Lecture Notes---Monograph
  Series, 11. Institute of Mathematical Statistics, Hayward, CA, 1988.

\bibitem{diaconisspec}
P.~Diaconis.
\newblock A generalization of spectral analysis with application to ranked
  data.
\newblock {\em Ann. Statist.}, 17(3):949--979, 1989.

\bibitem{diarock}
P.~Diaconis and D.~Rockmore.
\newblock Efficient computation of the {F}ourier transform on finite groups.
\newblock {\em J. Amer. Math. Soc.}, 3(2):297--332, 1990.

\bibitem{isotypicproj}
P.~Diaconis and D.~Rockmore.
\newblock Efficient computation of isotypic projections for the symmetric
  group.
\newblock In {\em Groups and computation ({N}ew {B}runswick, {NJ}, 1991)},
  volume~11 of {\em DIMACS Ser. Discrete Math. Theoret. Comput. Sci.}, pages
  87--104. Amer. Math. Soc., Providence, RI, 1993.

\bibitem{rao}
D.~Elliott and K.~Rao.
\newblock {\em Fast {T}ransforms: {A}lgorithms, {A}nalyses, {A}pplications}.
\newblock Academic Press Inc. [Harcourt Brace Jovanovich Publishers], New York,
  1982.

\bibitem{elliot}
G.~Elliott.
\newblock On the classification of inductive limits of sequences of semisimple
  finite-dimensional algebras.
\newblock {\em J. Algebra}, 38(1):29--44, 1976.

\bibitem{gabriel}
P.~Gabriel.
\newblock Unzerlegbare darstellungen {I}.
\newblock {\em Manuscripta Math.}, 6(1):71--103, 1972.

\bibitem{tsetlin}
I.~Gel{\cprime}fand and M.~Cetlin.
\newblock Finite-dimensional representations of the group of unimodular
  matrices.
\newblock {\em Doklady Akad. Nauk SSSR (N.S.)}, 71:825--828, 1950.

\bibitem{towers}
F.~Goodman, P.~de~la Harpe, and V.~Jones.
\newblock {\em Coxeter {G}raphs and {T}owers of {A}lgebras}, volume~14 of {\em
  Mathematical Sciences Research Institute Publications}.
\newblock Springer-Verlag, New York, 1989.

\bibitem{grood}
C.~Grood.
\newblock The rook partition algebra.
\newblock {\em J. Combin. Theory Ser. A}, 113(2):325--351, 2006.

\bibitem{halver}
T.~Halverson and E.~delMas.
\newblock Representations of the {R}ook-{B}rauer algebra.
\newblock {\em Comm. Algebra}, 42(1):423--443, 2014.

\bibitem{burr}
M.~Heideman, D.~Johnson, and C.~Burrus.
\newblock Gauss and the history of the fast {F}ourier transform.
\newblock {\em Arch. Hist. Exact Sci.}, 34(6):265--277, 1985.

\bibitem{humphreys}
J.~Humphreys.
\newblock {\em Reflection {G}roups and {C}oxeter {G}roups}, volume~29 of {\em
  Cambridge Studies in Advanced Mathematics}.
\newblock Cambridge University Press, Cambridge, 1990.

\bibitem{jameskerber}
G.~James and A.~Kerber.
\newblock {\em The {R}epresentation {T}heory of the {S}ymmetric {G}roup},
  volume~16 of {\em Encyclopedia of Mathematics and its Applications}.
\newblock Addison-Wesley Publishing Co., Reading, Mass., 1981.

\bibitem{frigo}
S.~Johnson and M.~Frigo.
\newblock A modified split-radix {FFT} with fewer arithmetic operations.
\newblock {\em IEEE Trans. Signal Process.}, 55(1):111--119, 2007.

\bibitem{leduc-ram}
R.~Leduc and A.~Ram.
\newblock A ribbon {H}opf algebra approach to the irreducible representations
  of centralizer algebras: The {B}rauer,{B}irman-{W}enzl, and type {A}
  {I}wahori-{H}ecke algebras.
\newblock {\em Adv. Math.}, 125:1--94, 1997.

\bibitem{lundy}
T.~Lundy and J.~Van~Buskirk.
\newblock A new matrix approach to real {FFT}s and convolutions of length
  {$2^k$}.
\newblock {\em Computing}, 80(1):23--45, 2007.

\bibitem{compactgroups}
D.~Maslen.
\newblock Efficient computation of {F}ourier transforms on compact groups.
\newblock {\em J. Fourier Anal. Appl.}, 4(1):19--52, 1998.

\bibitem{maslen}
D.~Maslen.
\newblock The efficient computation of {F}ourier transforms on the symmetric
  group.
\newblock {\em Math. Comp.}, 67(223):1121--1147, 1998.

\bibitem{lanczosit}
D.~Maslen, M.~Orrison, and D.~Rockmore.
\newblock Computing isotypic projections with the {L}anczos iteration.
\newblock {\em SIAM J. Matrix Anal. Appl.}, 25(3):784--803, 2003.

\bibitem{MR-adapted}
D.~Maslen and D.~Rockmore.
\newblock Adapted diameters and {F}{F}{T}s on groups.
\newblock In {\em Proc. 6th {A}{C}{M}-{S}{I}{A}{M} {S}{O}{D}{A}}, pages
  253--262. ACM, 1995.

\bibitem{rockmassurvey}
D.~Maslen and D.~Rockmore.
\newblock Generalized {FFT}s---a survey of some recent results.
\newblock In {\em Groups and computation, {II} ({N}ew {B}runswick, {NJ},
  1995)}, volume~28 of {\em DIMACS Ser. Discrete Math. Theoret. Comput. Sci.},
  pages 183--237. Amer. Math. Soc., Providence, RI, 1997.

\bibitem{sovi}
D.~Maslen and D.~Rockmore.
\newblock Separation of variables and the computation of {F}ourier transforms
  on finite groups. {I}.
\newblock {\em J. Amer. Math. Soc.}, 10(1):169--214, 1997.

\bibitem{MR-duco}
D.~Maslen and D.~Rockmore.
\newblock Double coset decompositions and computational harmonic analysis on
  groups.
\newblock {\em Journal of Fourier Analysis and Applications}, 6(4):349--388,
  2000.

\bibitem{MaslenNotices}
D.~Maslen and D.~Rockmore.
\newblock The {C}ooley-{T}ukey {FFT} and group theory.
\newblock {\em Notices of the Amer. Math. Soc.}, 48(10):1151--1160, 2001.

\bibitem{algpaper}
D.~Maslen, D.~Rockmore, and S.~Wolff.
\newblock Separation of variables and the computation of {F}ourier transforms
  on semisimple algebras.
\newblock In preparation.

\bibitem{seminormal}
A.~Ram.
\newblock Seminormal representations of {W}eyl groups and {I}wahori-{H}ecke
  algebras.
\newblock {\em Proc. London Math. Soc. (3)}, 75(1):99--133, 1997.

\bibitem{rocksurvey}
D.~Rockmore.
\newblock Some applications of generalized {FFT}s.
\newblock In {\em Groups and computation, {II} ({N}ew {B}runswick, {NJ},
  1995)}, volume~28 of {\em DIMACS Ser. Discrete Math. Theoret. Comput. Sci.},
  pages 329--369. Amer. Math. Soc., Providence, RI, 1997.

\bibitem{RockmoreIEEE}
D.~Rockmore.
\newblock The {FFT}: An algorithm the whole family can use.
\newblock {\em Computing in Science and Eng.}, 2(1):60--64, Jan. 2000.

\bibitem{classification}
M.~R{\o}rdam and E.~St{\o}rmer.
\newblock {\em Classification of {N}uclear {$C^*$}-algebras. {E}ntropy in
  {O}perator {A}lgebras}, volume 126 of {\em Encyclopaedia of Mathematical
  Sciences}.
\newblock Springer-Verlag, Berlin, 2002.

\bibitem{serre}
J.~Serre.
\newblock {\em Linear {R}epresentations of {F}inite {G}roups}.
\newblock Springer-Verlag, New York, 1977.
\newblock Translated from the second French edition by Leonard L. Scott,
  Graduate Texts in Mathematics, Vol. 42.

\bibitem{diffposets}
R.~Stanley.
\newblock Differential posets.
\newblock {\em J. Amer. Math. Soc.}, 1(4):919--961, 1988.

\bibitem{diffposets2}
R.~Stanley.
\newblock Variations on differential posets.
\newblock In {\em Invariant theory and tableaux ({M}inneapolis, {MN}, 1988)},
  volume~19 of {\em IMA Vol. Math. Appl.}, pages 145--165. Springer, New York,
  1990.

\bibitem{toli}
R.~Tolimieri, M.~An, and C.~Lu.
\newblock {\em Algorithms for {D}iscrete {F}ourier {T}ransform and
  {C}onvolution}.
\newblock Signal Processing and Digital Filtering. Springer-Verlag, New York,
  second edition, 1997.

\bibitem{vanloan}
C.~Van~Loan.
\newblock {\em Computational {F}rameworks for the {F}ast {F}ourier
  {T}ransform}, volume~10 of {\em Frontiers in Applied Mathematics}.
\newblock Society for Industrial and Applied Mathematics (SIAM), Philadelphia,
  PA, 1992.

\bibitem{yavne}
R.~Yavne.
\newblock An economical method for calculating the discrete {F}ourier
  transform.
\newblock {\em Proc. AFIPS Fall Joint Computer Conf.}, 33:115--125, 1968.

\bibitem{yg1}
A.~Young.
\newblock On quantitative substitutional {analysis}.
\newblock {\em Proc. London Math. Soc.}, 31(2):273--288, 1929.

\end{thebibliography}

\end{document}